\documentclass[11pt]{amsart}
\usepackage[T1]{fontenc}
\usepackage{amsmath,amssymb,amsthm}
\usepackage[margin=1in]{geometry}
\usepackage[hidelinks]{hyperref}
\usepackage{xcolor}
\usepackage{xparse}
\makeatletter
\newtheorem{theorem}{Theorem}[section]
\newtheorem{proposition}[theorem]{Proposition}
\newtheorem{lemma}[theorem]{Lemma}
\newtheorem{corollary}[theorem]{Corollary}
\newtheorem{introtheorem}{Theorem}

\theoremstyle{remark}
\newtheorem{remark}[theorem]{Remark}
\title{Positivity of Relative Pluricanonical Direct Images via Fixed-Point Metrics}
\author{Jixiang Fu and Yunyang Xiao}
\date{}
\begin{document}
\begin{abstract}
We develop a direct analytic approach to positivity of relative pluricanonical direct images through a fixed-point Hermitian metric by direct differential curvature calculations. The curvature resolvent formula from fixed-point equation leads to two-sided estimates of entire Nakano curvature by determinant curvature. In this way, we identify generic rank of determinant curvature with birational variation dimension for a projective algebraic fibration between smooth connected projective varieties with maximal birational variation and general fiber of general type, assuming $mK_F$ is globally generated or the $m$-canonical map is birational for general fiber $F$. In the global generation case, Green--Griffiths Conjecture asks for the tautological bigness of the ordinary symmetric powers $\operatorname{Sym}^{\ell}f_*\omega_{X/Y}^{\otimes m}$, without passing to reflexive hulls, for $\ell\geq P_m(F)=h^0(F,mK_F)$ whenever $m\geq 1$. We prove it  for all $\ell\geq1$ at levels $m\geq 2$, also in the case of birational $m$-canonical map, but construct a counterexample at level $m=1$. Over curves, curvature and boundary estimates give a direct analytic proof of Viehweg's ampleness theorem requested by Fujino. For smooth projective families, we also establish numerical dimension and vanishing results related to the birational variation dimension.
\end{abstract}
\maketitle

\section{Introduction}\label{sec:introduction}

The positivity of direct images of relative pluricanonical sheaves reflects the variation of projective families and the degeneration of morphisms. Foundational results include Fujita and Kawamata's Hodge-theoretic treatment \cite{Fujita1978,Kawamata1982}, Viehweg's weak positivity and ampleness theorem \cite{Viehweg1983,Viehweg1983II,Viehweg1995}, and Koll\'ar's relationship between determinant bigness and variation \cite{Kollar1987}. We are motivated by two problems. For families with maximal birational variation and general fibers $F$ of general type, the Green--Griffiths conjecture predicts tautological bigness of symmetric powers of the relative pluricanonical direct images when $mK_F$ is globally generated \cite[Conjecture~VI.C.30]{GreenGriffiths2018}. Fujino asks for a direct analytic proof of Viehweg's ampleness theorem \cite[Problem~4.11]{FujinoSemipositivity2015}.

To derive bigness or ampleness from a single metric, we use two properties: strict positivity on the smooth locus and control near singular fibers. We use a metric for which the full Nakano curvature has quantitative upper and lower bounds in terms of determinant curvature, and the kernel of determinant curvature corresponds to directions of trivial variation. Similar construction of such metric has already been considered by Berman \cite{Berman2013} as a quantization of the K\"ahler--Ricci flow, motivated from Donaldson's iteration \cite{Donaldson2005}. Here our aim is to obtain geometric conclusions from direct differential calculations of the curvature of this metric.

 One of the advantages of the fixed-point metric is that it is smooth on a dense Zariski-open subset, which allows classical curvature computations. It also admits the semipositive singular extension needed for the global applications, which is ensured by comprison with the Narasimhan--Simha metric theory of Berndtsson--P\u{a}un \cite{BerndtssonPaun2008}, P\u{a}un--Takayama \cite{PaunTakayama2018} and Hacon--Popa--Schnell \cite{HaconPopaSchnell2018}. Differentiating the fixed-point equation derives a representation of the curvature by a resolvent formula of its local linearized operator. A lower bound for that operator yields the two-sided control of the full Nakano curvature by determinant curvature. Such lower bound also transfers the positivity of the linearized operator, which results from Berndtsson's seminal work \cite{Berndtsson2009} and its generalizations \cite{CaoGuenanciaPaun2021,Zou2021}, to the global Nakano curvature tensor. These curvature arguments establish positivity directly on the original direct image, without an auxiliary cyclic covering or a fiber product of the family.

\subsection{Fixed-point metrics and curvature}
All varieties are over $\mathbb C$. For a projective morphism $f:X\to Y$ between smooth varieties with connected fibers, write
\[
 E_m=f_*\omega_{X/Y}^{\otimes m},\qquad
 P_m(F)=h^0(F,mK_F).
\]
We denote by $\varphi_m$ the rational map induced by the complete linear system $|mK_F|$ and consider the following conditions on a general fiber $F$:
\[
 \text{(G)}\quad mK_F\text{ is globally generated};\qquad
 \text{(B)}\quad \varphi_m\text{ is birational onto its image}.
\]
Under (G), $\varphi_m$ is a morphism. Under (B), the linear system may have base points and fixed components. We write $\operatorname{Var}(f)$ for the birational variation dimension.

We define $G$ fiberwise as the unique fixed point of the Hilbert--Fubini--Study operator $T_y$ arising from the Narasimhan--Simha construction for $m\geq2$ and almost every $y\in Y$. This operator contracts the cone of positive definite Hermitian metrics on $E_{m,y}$ in the Thompson metric. The cone is a complete metric space \cite{Thompson1963}, so the contraction mapping theorem provides the required fixed point and defines $G$ almost everywhere on $E_m$. We prove in Section~\ref{sec:construction} that $G$ is locally $C^{0,2/m}$ on its defining locus and smooth on a dense Zariski-open subset $U_m^{\mathrm{reg}}\subset Y$. Appendix~\ref{app:refined-regularity} gives a refinement of regularity using fiberwise log canonical thresholds, and the metric construction is also available for effective klt pairs.

Choose a holomorphic Chern-normal frame for $G$ at $y_0$, so that $G_{y_0}=I$ and its first derivatives vanish, and define the \emph{frozen metric} by $H_{y_0}(y)=T_y(I)$. This metirc can be viewed as the first order estimate of $G$. We obtain a curvature resolvent formula to connect these two metric, which allows us to lift the semipositivity of the local frozon metric from classical methods to the strict positivity of fixed-point metric. It also leads to two-sided comparison between full Nakano curvature and determine curvature and identification of generic rank of determinant curvature with birational variation dimensions, as summarized in the following theorem.

\begin{introtheorem}[Curvature and variation]\label{thm:intro-metric}
Let $f:X\to Y_0$ be a smooth projective morphism of smooth complex varieties with connected fibers. Let $\Delta\geq0$ be a horizontal rational divisor such that $(X,\Delta)$ and every fiber pair are klt. Fix $m\geq2$ with $m\Delta$ integral, and assume that
\[
 E=f_*\mathcal O_X\bigl(m(K_{X/Y_0}+\Delta)\bigr)
\]
is locally free of positive rank $r$, with fiberwise base-change isomorphisms
\[
 E\otimes\mathbb C(y)\simeq
 H^0\!\left(X_y,m(K_{X_y}+\Delta_y)\right)
 \qquad(y\in Y_0).
\]
Let $G$ be the fixed-point metric and put $\alpha=(m-1)/m$.
\begin{enumerate}
\item On $U_m^{\mathrm{reg}}$, let $K$ be the curvature tensor of $G$, viewed as a Hermitian form on $T^{1,0}Y_0\otimes E$, and let $Q$ be the frozen curvature tensor in a $G$-Chern-normal frame. Then $Q$ is Nakano semipositive and
\[
 K=\bigl(\mathrm{id}\otimes(I-L)^{-1}\bigr)Q,
\]
where $L$ is the linearized Hilbert--Fubini--Study operator. Both $L$ and $(I-L)^{-1}$ are completely positive, and $\|L\|=\alpha$. For $\tau=\operatorname{Tr}_E K$, there is a locally uniform $\sigma>0$ such that
\[
 \sigma\tau\otimes I
 \leq K
 \leq\bigl[r-(r^2-1)\sigma\bigr]\tau\otimes I
\]
in the Nakano sense. In particular,
\[
 \ker K=(\ker\tau)\otimes E.
\]
If $\Delta=0$ and $Y_0$ is compact, a uniform constant $\sigma_*>0$ works throughout $U_m^{\mathrm{reg}}$.
\item Suppose additionally that $\Delta=0$, the general fiber is of general type, and (G) or (B) holds. On a nonempty Zariski-open subset $U_m^\circ\subseteq U_m^{\mathrm{reg}}$,
\[
 \operatorname{rank}\tau=\operatorname{Var}(f),\qquad
 \operatorname{rank}K=r\operatorname{Var}(f).
\]
Under (G), $\ker\tau=\ker\rho$, where $\rho$ is the Kodaira--Spencer map. Under (B), $\ker\tau=\ker db_m$ where $b_m$ is the birational classifying map of Section~\ref{sec:variation}. In both cases, the determinant form $i\operatorname{Tr}\Theta_{E,G}$ descends locally to a K\"ahler form on a submersion base of dimension $\operatorname{Var}(f)$.
\end{enumerate}
\end{introtheorem}

Roughly speaking, the resolvent $(I-L)^{-1}=\sum_{j\geq0}L^j$ expresses the effect of the fixed-point condition on curvature, which is essentially a analytic version of the iteration of Hilbert--Fubini--Study operator from the intrinsic geometry of the variation. Because the resolvent is completely positive, it preserves positivity on all of $T^{1,0}Y_0\otimes E$, not only on decomposable tensors, which thus transfers the Nakano positivity of curvatures.

 The evaluation quadratic form introduced in Section~\ref{sec:curvature} is positive definite on $\operatorname{End}(E_y)$. Its lower bound yields the coefficient $\sigma$ in Theorem~\ref{thm:intro-metric}. Thus a base direction with positive determinant curvature has positive curvature in every section direction. The two-sided estimate also determines the entire curvature kernel. The resolvent formula and determinant pinching are proved in Theorems~\ref{thm:curvature-resolvent} and \ref{thm:determinant-pinching}.

The two geometric hypotheses enter only when identifying these kernels. Under (G), equality in the direct-image curvature formula gives holomorphic lifts that glue across the ramification locus, while projectable holomorphic flows give the reverse implication. Under (B), vanishing curvature forces the infinitesimal deformation of the relative pluricanonical image to vanish modulo projective transformations. Birational invariance and uniqueness of the fixed-point metric give the converse. A sufficiently large multiple $N=km$ is used in case (G) to identify the common rank with birational variation. The curvature in the conclusion is still that of $E_m$. These arguments yield Corollary~\ref{cor:unified-variation-rank}.

Berman established contraction, convergence, and relative positivity for the earlier construction \cite[Proposition~4.12, Theorem~4.14, Corollary~4.16 and \S4E]{Berman2013}. Remark~\ref{rem:normalization} explains its relation to ours.

 Curvature descriptions of deformation are also central to Berndtsson's direct-image theory \cite{Berndtsson2009,Berndtsson2011} and Schumacher's work on canonically polarized families \cite{Schumacher2012,Schumacher2013Erratum}. For general fibers of general type, our two-sided Nakano-curvature comparison and variation statement do not require canonical polarization.
 
 The frozen semipositivity follows from the adjoint direct-image theorem after resolving the base ideal and identifying the complete section spaces \cite[Case Z, Theorem~1.1]{Zou2021}. In the globally generation case, it can be derived from \cite[Theorem~1.2]{Berndtsson2009}.

\subsection{Green--Griffiths bigness conjecture}
Let $f:X\to Y$ be a surjective projective morphism between smooth connected projective varieties, with connected fibers. Assume that its general fiber is of general type and that $f$ has maximal birational variation. In the 2018 version~2 of their paper, Green--Griffiths conjectured the tautological bigness of
\[
 \operatorname{Sym}^{\ell}E_m,
 \qquad \ell\geq P_m(F),
\]
for every $m\geq1$, assuming that $mK_F$ is globally generated for general fibers $F$ \cite[Conjecture~VI.C.30]{GreenGriffiths2018}. The conjecture uses ordinary symmetric powers, without passing to reflexive hulls.

To formulate our result for sheaves that need not be locally free, let
\[
 \mathbb P_q(\mathcal F)
 =\operatorname{Proj}_Y\operatorname{Sym}^{\bullet}\mathcal F
\]
be the projectivization in the quotient convention. For a coherent sheaf $\mathcal F$ of positive rank, its \emph{principal component} is the reduced closure of the projective bundle over a dense locally free locus. We call $\mathcal F$ \emph{tautologically big} if $\mathcal O(1)$ is big on this component. Throughout, $\operatorname{Sym}^{\ell}\mathcal F$ denotes the ordinary coherent symmetric power. These conventions are explained in Section~\ref{sec:bigness}.

\begin{introtheorem}[Tautological bigness]\label{thm:intro-bigness}
Let $f:X\to Y$ be a surjective projective morphism of smooth connected projective varieties, with connected positive-dimensional fibers and $\dim Y>0$. Assume that the general fiber $F$ is of general type and $\operatorname{Var}(f)=\dim Y$. If $m\geq2$, and (G) or (B) holds, then
\[
 \operatorname{Sym}^{\ell}E_m
 \quad\text{is tautologically big for every }\ell\geq1.
\]
\end{introtheorem}

Under (G), this proves the conjectural conclusion for $m\geq2$ and replaces $\ell\geq P_m(F)$ by $\ell\geq1$. The alternative (B) includes linear systems with base points. The restriction $m\geq2$ cannot be dropped, even when the general fiber has ample and globally generated canonical bundle.

\label{subsec:intro-level-one}
\begin{introtheorem}[Level-one counterexample]\label{thm:intro-counterexample}
There is a flat surjective projective morphism $f:X\to\mathbb P^1$ with $X$ smooth connected projective, connected fibers, and reduced normal-crossing singular fibers. Its smooth general fiber $S$ is a surface with ample and globally generated $K_S$, and
\[
 \operatorname{Var}(f)=1,\qquad
 f_*\omega_{X/\mathbb P^1}\simeq\mathcal O_{\mathbb P^1}^{\oplus4}.
\]
Thus every positive symmetric power of this direct image fails to be tautologically big.
\end{introtheorem}

The construction combines a pencil on a fixed abelian surface with a diagonal quotient. The resulting surfaces vary, but the quotient selects a constant part of the canonical sections. The calculation over the complete base gives a trivial direct image, so the failure of bigness persists after the singular fibers are included. Since $P_1(S)=4$, the example also lies in the conjectural range $\ell\geq4$. Its Prym and product-quotient geometry has antecedents in Mumford \cite[\S7]{Mumford1974}, Ciliberto--Pardini--Tovena \cite[Example~3.5 and Theorem~6.1]{CilibertoPardiniTovena2000}, and Bauer--Catanese \cite[Proposition~2.3]{BauerCatanese2004}. The full construction is given in Theorem~\ref{thm:level-one-counterexample}. The example shows that canonical direct-image positivity need not reflect all birational variation.

For Theorem~\ref{thm:intro-bigness}, we use a metric that is strictly Nakano positive in the interior and admits a semipositive singular extension. Comparison with the Hermitian Narasimhan--Simha metric and the minimal extension property of Hacon--Popa--Schnell \cite{HaconPopaSchnell2018} give a locally uniform lower bound on
\[
 V_m=\pi^*E_m/\mathrm{tors}
\]
after a smooth projective modification $\pi:\widetilde Y\to Y$ for which $V_m$ is locally free. The dual plurisubharmonic norms then extend to $V_m$. The quotient metric is a metric on the exact pullback of the tautological line bundle on the principal component, without an exceptional-divisor correction. Its curvature is strictly positive over a nonempty Zariski-open subset, which gives bigness. This is the extension argument in Theorem~\ref{thm:actual-sheaf-extension} and the proof of Theorem~\ref{thm:prescribed-bigness}.

Taking a reflexive hull may remove codimension-two vanishing conditions and change tautological bigness, as Remark~\ref{rem:bigness-boundaries} illustrates. When the direct image is locally free, bigness conclusions in overlapping settings already follow from classical determinant bigness, Cao--P\u{a}un's transfer argument \cite[Theorem~3.4 and Lemmas~5.3--5.4]{CaoPaun2017}, and P\u{a}un--Takayama's metric criterion \cite{PaunTakayama2018}. The fixed-point proof gives, in addition to the bigness conclusion, a curvature tensor with explicit determinant bounds and geometric null directions, together with extension on the coherent direct image.

Uniform pluricanonical birationality \cite[Theorem~1.1]{HaconMcKernan2006} makes (B) available for every sufficiently large $m$, with a bound depending only on the fiber dimension. Under the hypotheses of Theorem~\ref{thm:intro-bigness}, this gives tautological bigness of every positive symmetric power for all such $m$ (Corollary~\ref{cor:all-large-levels}). If $f$ is smooth everywhere, then $\det E_m$ is big and every $\operatorname{Sym}^{\ell}E_m$ is Viehweg big (Corollary~\ref{cor:smooth-viehweg-bigness}).

\subsection{Viehweg's ampleness theorem}
Let $f:X\to C$ be a surjective morphism from a smooth connected projective variety to a smooth connected projective curve, with connected fibers. The metric extension applies directly to quotient bundles: their determinant currents are positive and have total mass equal to their degrees. This gives nefness of
\[
 f_*\omega_{X/C}^{\otimes j}\qquad(j\geq1)
\]
without an assumption on the singular fibers (Theorem~\ref{thm:curve-unweighted-nef}). The proof uses $G$ for $j\geq2$ and the canonical $L^2$ metric for $j=1$.

For ampleness, we must distinguish positive determinant degree from curvature mass where $G$ is smooth: the degree may include mass at singular fibers, where determinant pinching does not apply. To prove ampleness, we therefore use the following hypotheses, which give logarithmic boundary growth and exclude such point masses.

\begin{introtheorem}[Ampleness or flatness]\label{thm:intro-curves}
For the morphism $f:X\to C$ above, let $\Delta\geq0$ be a horizontal rational divisor such that $(X,\Delta)$ is klt and $(X,\Delta+X_c)$ is log canonical for every scheme fiber $X_c=f^*c$, with its multiplicities. Fix $m\geq2$ with $m\Delta$ integral, put $E_m=f_*\mathcal O_X(m(K_{X/C}+\Delta))$, and assume $E_m\ne0$. Then $E_m$ is nef, and exactly one of the following holds:
\begin{enumerate}
\item $\deg E_m>0$ and $E_m$ is ample;
\item $\deg E_m=0$ and its fixed-point metric $G$ extends smoothly as a Hermitian-flat metric on $E_m$.
\end{enumerate}
\end{introtheorem}

The boundary estimates identify the degree of $E_m$ with the determinant curvature mass on the regularity locus. If the degree is positive, integrating the pinching inequality gives a positive lower bound for every quotient slope. If it is zero, semipositivity forces the curvature to vanish, and the same boundary estimates give smooth extension of the flat metric. This proves Theorem~\ref{thm:curve-dichotomy} and identifies the metric in both alternatives.

Under the same pair hypotheses, $\deg E_k>0$ for some $k\geq1$ with $k\Delta$ integral implies that $E_{qk}$ is ample for every positive integer $q$ with $qk\geq2$ (Corollary~\ref{cor:curve-multiples}). For $\Delta=0$ and semistable $f$, this conclusion, together with nefness in general situations, gives Viehweg's ampleness theorem in the form stated by Fujino \cite[Theorem~4.10]{FujinoSemipositivity2015}. Corollary~\ref{cor:curve-fujino} thus provides the direct analytic proof requested in his Problem~4.11. Same-level ampleness and determinant estimates occur in Viehweg--Zuo \cite[Proposition~2.7]{ViehwegZuo2001}, and degree-zero flatness has predecessors in Cao--P\u{a}un \cite[Theorem~5.2]{CaoPaun2017} and Hacon--Popa--Schnell \cite[Theorem~26.1 and Corollary~27.2]{HaconPopaSchnell2018}. Our conclusions here all follow from the curvature and boundary behavior of a single metric.

\subsection{Numerical dimension and vanishing}
The curvature-rank formula also describes positivity when variation is not maximal. For a smooth projective family over a projective base, continuity of $G$ gives global determinant and tautological metrics with continuous local weights. Their numerical dimensions are determined by the curvature rank on $U_m^\circ$.

\begin{introtheorem}[Numerical dimension and vanishing]\label{thm:intro-vanishing}
Let $f:X\to Y$ be a smooth surjective projective morphism of smooth connected projective varieties, where $d=\dim Y\geq1$. Assume that its fibers are connected, positive-dimensional, and of general type. Fix $m\geq2$ such that $E=E_m$ has positive rank $r$ and (G) or (B) holds on a general fiber. Put $M=\det E$ and $v=\operatorname{Var}(f)$, and let $\sigma_*>0$ be the uniform pinching constant of Theorem~\ref{thm:intro-metric}, chosen as in \eqref{eq:global-full-moment}.
\begin{enumerate}
\item The line bundles $M$ and $\mathcal O_{\mathbb P_q(E)}(1)$ are nef, and
\[
 \nu(M)=v,\qquad
 \nu\!\left(\mathcal O_{\mathbb P_q(E)}(1)\right)=r-1+v.
\]
\item For integers $a\geq1$ and $b\geq0$ satisfying $b<a\sigma_*$, and every $P\in\operatorname{Pic}^0(Y)$,
\[
 H^q\!\left(Y,K_Y\otimes\operatorname{Sym}^aE
                  \otimes M^{-b}\otimes P\right)=0
 \qquad(q>d-v).
\]
\end{enumerate}
\end{introtheorem}

The first conclusion extends to the tautological line bundle of every $\operatorname{Sym}^{\ell}E$: its numerical dimension is $\operatorname{rank}(\operatorname{Sym}^{\ell}E)-1+v$. The vertical Fubini--Study curvature accounts for the projective-fiber dimension, and the horizontal curvature contributes exactly $v$ directions. Bedford--Taylor products of the continuous local potentials compute the ordinary intersection numbers and put no mass on the excluded analytic subset. This gives both bounds needed for the exact numerical dimension (Corollaries~\ref{cor:numerical-dimension} and \ref{cor:tautological-numerical-dimension}).

The vanishing proof begins with the induced Nakano estimate on symmetric powers, whose coefficient is $a\sigma_*-b>0$. Its commutator on forms of bidegree $(d,q)$ is positive definite almost everywhere for $q>d-v$. We construct approximate $L^2$ solutions on $U_m^{\mathrm{reg}}$, extend their equations across the complement, and apply the closed range theorem, see Lemma~\ref{lem:bounded-metric-vanishing}. Global continuity of the metric is essential to this argument.

Sections~\ref{sec:construction} and \ref{sec:curvature} establish the fixed-point construction, its regularity, and the curvature estimates. Section~\ref{sec:variation} relates the curvature kernels to variation. Section~\ref{sec:bigness} proves bigness and constructs the level-one counterexample, while Section~\ref{sec:curves} treats positivity over curves. Section~\ref{sec:numerical-vanishing} proves the numerical dimension and vanishing theorems. Appendix~\ref{app:refined-regularity} contains the proof of the refined regularity criterion.

\section{Fixed-point metrics and regularity}
\label{sec:construction}

\subsection{Section spaces and integral conventions}
Fix an integer $m\geq2$ and put $\alpha=(m-1)/m$ and $\beta=2/m$. Initially, $Z$ is a connected smooth projective variety of dimension $n$, $\Delta=\sum_j\delta_jD_j$ is an effective rational divisor, $(Z,\Delta)$ is klt, and $m\Delta$ is integral. Set
\[
 V=H^0\bigl(Z,m(K_Z+\Delta)\bigr),\qquad r=\dim V>0.
\]
The case $\Delta=0$ is included throughout. Write $\zeta=dz_1\wedge \cdots\wedge dz_n$ and $e_m=\zeta^{\otimes m}\prod_jd_j^{-m\delta_j}$, where $d_j$ is a local equation of $D_j$. For a basis $e=(s_1,\ldots,s_r)$ write $s_i=f_i e_m$ and $F=(f_1,\ldots,f_r)$. Our volume convention and boundary measure are
\[
 dV_n=\prod_{j=1}^n\frac i2 dz_j\wedge d\bar z_j,
 \qquad d\mu_\Delta=\prod_j|d_j|^{-2\delta_j}dV_n.
\]
These chartwise expressions define the densities used below. Hermitian forms are conjugate linear in the first argument, with squared norm $A(s,s)=c^*Ac$ for $s=ec$ and a positive definite matrix $A$.

Set
\begin{equation}\label{eq:hilb-fs}
 B_A=FA^{-1}F^*,\qquad \varphi_A=\frac1m\log B_A,
 \qquad T(A)=\int_Z F^*F B_A^{-\alpha}\,d\mu_\Delta.
\end{equation}
The function $\varphi_A$ is the local weight of the Fubini--Study metric on the rational line bundle $K_Z+\Delta$. A fiber vector is represented locally by the formal expression $\xi=t e_m^{1/m}$, with $\xi^{\otimes m}=t^m e_m$. A different formal root multiplies $t$ by an $m$-th root of unity and leaves its absolute norm unchanged. No global root line bundle is assumed. The squared norm at $x$ is
\[
 |\xi|_{\mathrm{FS}(A),x}^2
 =\inf\{A(s,s)^{1/m}:s(x)=\xi^{\otimes m}\}.
\]
At a base point the infimum is infinite for $\xi\ne0$. In the integral we extend $F^*F B_A^{-\alpha}$ by zero at $F=0$. Divisors have measure zero, so no pointwise value of their weight is needed.

For a compact set $\mathcal A$ of positive matrices,
\begin{equation}\label{eq:kernel-domination}
 \|F^*F B_A^{-\alpha}\|
 \leq C_{\mathcal A}|F|^{2/m}
 \leq C_{\mathcal A}\sum_i|f_i|^{2/m},\qquad A\in\mathcal A.
\end{equation}
The klt hypothesis makes the right side locally integrable against $d\mu_\Delta$. On a log resolution, the boundary density pulls back to a smooth factor times $\prod|w_E|^{2(k_E-b_E)}dV_n$, where $k_E$ is the Jacobian order, $b_E=\operatorname{ord}_E\Delta$, and $k_E-b_E>-1$ by klt. Since the coefficients $f_i$ are locally bounded, this proves integrability even for a boundary that is not SNC. A nonzero section is nonzero on an open subset, so $T(A)$ is positive definite.

Under a basis change $e'=eC$, we have $F'=FC$, $A'=C^*AC$, $B_{A'}'=B_A$ and $T'(A')=C^*T(A)C$. Replacing the canonical frame by $a\zeta$ and the boundary equations by $d'_j=u_jd_j$ changes $e_m$ to $\lambda^m e_m$, where locally $\lambda=a\prod u_j^{-\delta_j}$. Thus $F'=\lambda^{-m}F$ and $d\mu'_\Delta=|\lambda|^2d\mu_\Delta$. The integral is unchanged because
\begin{equation}\label{eq:frame-cancellation}
 |\lambda|^{-2m}|\lambda|^{2m\alpha}|\lambda|^2=1.
\end{equation}
Only the absolute value of a choice of the rational power is involved.

\subsection{The fixed-point metric and comparison estimates}
On the positive Hermitian cone define the Thompson distance and the order norm on Hermitian tangent matrices by
\[
 d_{\mathrm{Th}}(A,B)=\inf\{t\geq0:e^{-t}A\leq B\leq e^tA\},
 \qquad \|S\|_{A,\infty}=\inf\{c\geq0:-cA\leq S\leq cA\}.
\]

\begin{theorem}[Fixed point and convergence]\label{thm:fixed-point-construction}
The operator $T$ in \eqref{eq:hilb-fs} is order preserving, homogeneous of degree $\alpha$, and $\alpha$-contracting for $d_{\mathrm{Th}}$. It has a unique positive definite fixed point $G$. For any starting form $A_0>0$ and $A_{j+1}=T(A_j)$,
\begin{equation}\label{eq:metric-convergence}
 d_{\mathrm{Th}}(A_j,G)
 \leq\alpha^j d_{\mathrm{Th}}(A_0,G)
 \leq m\alpha^j d_{\mathrm{Th}}(A_0,A_1).
\end{equation}
If $d_{\mathrm{Th}}(A,B)\leq t$, then
\begin{equation}\label{eq:fs-convergence}
 |\varphi_A-\varphi_B|\leq t/m
\end{equation}
off the common base locus. The corresponding two extended-real weight inequalities hold everywhere.
\end{theorem}
\begin{proof}
If $A\leq B$, inversion gives $B_A\geq B_B$, so $T(A)\leq T(B)$. Also $T(cA)=c^\alpha T(A)$ for $c>0$. Order preservation and homogeneity imply contraction in $d_{\mathrm{Th}}$. To see that the cone is complete for this distance, observe that a Thompson-Cauchy sequence lies, from some index onward, between two fixed positive multiples of a positive matrix. On that order interval it is Cauchy in matrix norm, and its positive definite limit is also its Thompson limit. This is the finite-dimensional case of \cite[Lemma~3]{Thompson1963}. The contraction theorem therefore gives a unique fixed point. Summing the distances between successive iterates gives $d_{\mathrm{Th}}(A_0,G)\leq(1-\alpha)^{-1}d_{\mathrm{Th}}(A_0,A_1)$, which proves \eqref{eq:metric-convergence}. Finally inversion gives $e^{-t}B_A\leq B_B\leq e^tB_A$ and hence \eqref{eq:fs-convergence}.
\end{proof}

Define the scalar pluricanonical gauge and its extremal function by
\begin{equation}\label{eq:ns-gauge}
 \ell(s)=\left(\int_Z|f_s|^{2/m}d\mu_\Delta\right)^{m/2},
 \quad \mathcal N_m(s)=\ell(s)^2,
 \quad q(x)=\sup_{\ell(s)\leq1}|f_s(x)|.
\end{equation}
The gauge is positive, homogeneous and continuous on the finite-dimensional space $V$, so its unit ball is compact. It need not be convex. The Hermitian Narasimhan--Simha form used here is
\begin{equation}\label{eq:ns-hermitian}
 G_{\mathrm{NS}}(s,t)=\int_Z\overline{f_s}f_t q^{2/m-2}d\mu_\Delta.
\end{equation}
This pairing is induced by the extremal metric, as in \cite[\S28]{HaconPopaSchnell2018}. It is distinct from $\mathcal N_m$.

\begin{proposition}\label{prop:ns-comparison}
For every $s\in V$,
\begin{align}
 G(s,s)&\leq\mathcal N_m(s)\leq r^{m-1}G(s,s),
       \label{eq:balanced-comparison}\\
 r^{-(m-1)}G_{\mathrm{NS}}&\leq G\leq G_{\mathrm{NS}}.
       \label{eq:ns-comparison}
\end{align}
Starting with $G_0=G_{\mathrm{NS}}$ and putting $G_{j+1}=T(G_j)$ gives
\begin{align}
 r^{-(m-1)^2(1-\alpha^j)}G_0&\leq G_j\leq G_0,
       \label{eq:iterate-comparison}\\
 d_{\mathrm{Th}}(G_j,G)&\leq(m-1)\alpha^j\log r.
       \label{eq:ns-rate}
\end{align}
These constants apply to the weighted and unweighted constructions.
\end{proposition}
\begin{proof}
Put $p=2/m\leq1$. Choose a basis $u_1,\ldots,u_r$ in the compact gauge unit ball maximizing the absolute value of its determinant. The maximum is nonzero. Replacing one vector by any element of the ball shows that all its coordinates in this basis have modulus at most one. Consequently $q^p\leq\sum_j|f_{u_j}|^p$ and
\[
 M:=\int_Zq^p d\mu_\Delta\leq r.
\]
The bound $|f_s|\leq\ell(s)q$ and H\"older's inequality give
\begin{equation}\label{eq:ns-gauge-comparison}
 G_0(s,s)\leq\ell(s)^2\leq M^{m-1}G_0(s,s)
 \leq r^{m-1}G_0(s,s).
\end{equation}
For the second inequality, apply H\"older with exponents $2/p$ and $2/(2-p)$ to
\[
 |f_s|^p=(|f_s|^2q^{p-2})^{p/2}q^{p(2-p)/2}.
\]
These estimates also prove that \eqref{eq:ns-hermitian} is finite and positive definite.

In a $G$-orthonormal basis put
\begin{equation}\label{eq:evaluation-moment}
 P_x=\frac{F(x)^*F(x)}{F(x)F(x)^*},\qquad
 d\nu=(FF^*)^{1/m}d\mu_\Delta.
\end{equation}
The fixed-point equation gives $\int P_xd\nu=I$, and taking traces gives $\nu(Z)=r$. For a unit vector $v$, let $h(x)=v^*P_xv$. Then $0\leq h\leq1$, $\int h\,d\nu=1$, and
\[
 1\leq\int h^{1/m}d\nu
 \leq r\left(\frac1r\int h\,d\nu\right)^{1/m}=r^\alpha.
\]
The integral is $\int|Fv|^{2/m}d\mu_\Delta$. Homogeneity proves \eqref{eq:balanced-comparison}. Combining its upper bound with $G_0\leq\mathcal N_m$ gives the lower bound in \eqref{eq:ns-comparison}.

Dualizing \eqref{eq:ns-gauge-comparison} at evaluation gives $r^{-(m-1)}B_{G_0}\leq q^2\leq B_{G_0}$. Therefore
\[
 r^{-(m-1)^2/m}G_0\leq T(G_0)\leq G_0.
\]
Order preservation makes $G_j$ decreasing, so its fixed-point limit is at most $G_0$. Iterating the lower bound using homogeneity proves \eqref{eq:iterate-comparison}, since $\sum_{k=0}^{j-1}\alpha^k=m(1-\alpha^j)$. Now \eqref{eq:ns-comparison} gives $d_{\mathrm{Th}}(G_0,G)\leq(m-1)\log r$, and contraction gives \eqref{eq:ns-rate}. 
\end{proof}

\begin{remark}[Normalization and the case of level one]\label{rem:normalization}
Berman's construction includes balanced pluricanonical iteration at levels with base loci \cite[Proposition~4.12, Theorem~4.14 and \S4E]{Berman2013}. Proposition~4.12 gives the contraction estimate used in the proof of his Theorem~4.20. In the common unweighted setting, replacing $\varphi_A$ by $m^{-1}\log(B_A/r)$ replaces $T$ by $T_{\mathrm B}=r^\alpha T$. Thus its fixed matrix is $G_{\mathrm B}=r^{m-1}G$, and $(B_{G_{\mathrm B}}/r)^{1/m}dV_n=r^{-1}B_G^{1/m}dV_n$ has mass one. The constant rescaling leaves curvature unchanged.

For $m=1$ and $\Delta=0$, the integral operator is the ordinary canonical $L^2$ pairing and is independent of its input. Its derivative is zero. For $m>1$, its nonzero derivative enters the curvature resolvent formula of Section~\ref{sec:curvature}. The contraction factor $\alpha$ depends only on $m$, so it does not measure geometric variation. In rank one $G(s,s)=\mathcal N_m(s)$. In a product family $T$ and $G$ are constant in the product frame for every $m$. The convergence estimates above apply to metrics and weights, without implying convergence of curvature tensors.
\end{remark}

\subsection{Birational change of variables}
\begin{proposition}\label{prop:birational-invariance}
The construction is independent of all the frames used above and commutes with isomorphisms and with base change wherever the complete section spaces commute with base change. Birational smooth projective varieties have isometric fixed-point metrics under the canonical identification of their complete pluricanonical spaces. More generally, a connected normal projective effective klt pair $(Z,\Delta)$ with $m(K_Z+\Delta)$ Cartier and a nonzero space of global sections has the same intrinsic construction, and crepant birational identifications preserving the rational log-pluricanonical forms identify its gauge, every Hilbert input, and its fixed-point metric.
\end{proposition}
\begin{proof}
The covariance formulas and \eqref{eq:frame-cancellation} give frame invariance. Under the stated base-change hypothesis, the fiber integrals agree, and uniqueness identifies their fixed points. For a smooth birational morphism $\pi:\widehat Z\to Z$, the pullback of $\zeta$ is $J(w)dw$. The coefficient row of the pulled-back $m$-canonical forms is $\widehat F=J^m(F\circ\pi)$, so
\begin{equation}\label{eq:birational-jacobian}
 \widehat F^*\widehat F(\widehat FA^{-1}\widehat F^*)^{-\alpha}dV_w
 =|J|^2(F^*FB_A^{-\alpha})\circ\pi\,dV_w.
\end{equation}
Here $2m(1-\alpha)=2$. Pluricanonical sections pull back and descend through such a morphism: a section upstairs defines a rational section downstairs without poles off a codimension-two set, and normality extends it. Passing through a smooth common resolution identifies the complete pluricanonical spaces. Change of variables on the isomorphism locus, together with integrability, then identifies every operator $T(A)$ and the scalar gauges.

For the normal pair choose a log resolution and write
\[
 K_{\widehat Z}+B-P=\pi^*(K_Z+\Delta),
 \qquad B,P\geq0,
\]
with disjoint supports. The divisor $B$ has coefficients less than one, $P$ is exceptional, and $mB,mP$ are integral. Since a rational function whose poles are confined to exceptional divisors is regular on the normal variety downstairs, $\pi_*\mathcal O_{\widehat Z}(mP)=\mathcal O_Z$. The projection formula identifies the full section spaces:
\begin{equation}\label{eq:normal-section-identification}
 H^0\bigl(Z,m(K_Z+\Delta)\bigr)
 \simeq H^0\bigl(\widehat Z,m(K_{\widehat Z}+B)\bigr).
\end{equation}
The map is pullback followed by multiplication by the canonical section of $mP$. To check its effect on the integrals, if $b$ and $a$ are local equations of $mB$ and $mP$, respectively, put $e_B=(dw)^m/b$ and $e_{B-P}=a e_B$. A pulled-back section $h e_{B-P}$ maps to $ha e_B$. Thus the row is multiplied by $a$, with divisor $mP$. A Hilbert kernel gains $|a|^{2-2\alpha}=|a|^{2/m}$, and the scalar density gains the same factor. Against the effective boundary weight $B$ this is exactly the density for the signed divisor $B-P$. This recovers the change-of-variables formula for $\pi^*(K_Z+\Delta)$. All integrals on the effective klt model are finite. A common resolution proves independence of the chosen model. The exceptional fixed factor in \eqref{eq:normal-section-identification} is part of the evaluation row.
\end{proof}

\subsection{Relative case and continuity of the fixed-point metric}
Let $f:X\to Y$ be a smooth projective morphism of smooth complex varieties with connected fibers. In the unweighted case, \cite[Corollary~0.3]{Siu2002} gives constancy of the plurigenera. The sheaf $\mathcal O_X(mK_{X/Y})$ is flat over the reduced base, and $f$ is proper. Grauert's theorem then gives local freeness and
\begin{equation}\label{eq:complete-base-change}
 (E_m)_y=H^0(X_y,mK_{X_y}),\qquad E_m=f_*\mathcal O_X(mK_{X/Y})
\end{equation}
\cite[\S7, Satz~5]{Grauert1960}, with the correction \cite{Grauert1963Correction}. We assume its rank is positive.

For a weighted family put
\[
 E_m=f_*\mathcal O_X(m(K_{X/Y}+\Delta)),
\]
where $\Delta\geq0$ is horizontal, $m\Delta$ is integral, and $(X,\Delta)$ is klt. We restrict to a Zariski-open subset on which every fiber pair is klt, $E_m$ is locally free, and the maps $(E_m)_y\to H^0(X_y,m(K_{X_y}+\Delta_y))$ are isomorphisms. A relative log resolution and generic smoothness give the first condition on a dense Zariski-open subset. Generic flatness and cohomology and base change give the remaining conditions after further shrinking. The regularity locus constructed below also requires all resolved boundary strata to be smooth over the base.

\begin{lemma}\label{lem:holder-continuity}
For $\Delta=0$, the metric $G$ is locally $C^{0,2/m}$ on all of $Y$. The corresponding Fubini--Study weight has the same regularity off the relative base locus.
\end{lemma}
\begin{proof}
For $A$ in a compact positive-matrix set put $H_A(v)=v^*v(vA^{-1}v^*)^{-\alpha}$, with $H_A(0)=0$. This kernel is homogeneous of degree $\beta=2/m\leq1$. Uniform bounds on the unit sphere give
\[
 \|H_A(v)\|\leq C|v|^\beta,\qquad
 \|D_vH_A(v)\|\leq C|v|^{\beta-1}\quad(v\ne0).
\]
If $\delta=|v-w|$ and $\min(|v|,|w|)\leq2\delta$, the first bound gives $\|H_A(v)-H_A(w)\|\leq C\delta^\beta$. Otherwise the segment stays away from zero and the second bound gives $C\delta\min(|v|,|w|)^{\beta-1}\leq C\delta^\beta$. Cover a fiber by finitely many relative charts. After shrinking the base, properness and smooth cutoffs express the fiber integrals over fixed bounded domains. The holomorphic evaluation rows and the cutoffs have bounded derivatives there, so
\begin{equation}\label{eq:operator-holder}
 \|T_y(A)-T_z(A)\|\leq C|y-z|^\beta
\end{equation}
uniformly on compact positive-matrix sets. First fix $A_0=G_{y_0}$. Contraction and the triangle inequality give
\[
 (1-\alpha)d_{\mathrm{Th}}(G_y,A_0)
 \leq d_{\mathrm{Th}}(T_y(A_0),T_{y_0}(A_0)).
\]
This first proves continuity and places $G_y$ in a compact subset of the positive cone near $y_0$. Repeating the argument with $A_0=G_z$ and using \eqref{eq:operator-holder} proves the stated exponent, since matrix norm and Thompson distance are locally equivalent. Inversion and the logarithm of the positive function $FG^{-1}F^*$ preserve this regularity off the base locus.
\end{proof}

\subsection{Smoothness on the regularity locus}
Fix a smooth connected Zariski-open subset $Y_0\subseteq Y$ on which $f$ is smooth, $E_m$ is locally free of positive rank, and the fiberwise base-change maps are isomorphisms. In the weighted setting also require that every fiber pair is klt. In the unweighted smooth-family setting one may take $Y_0=Y$. Let $L_m=\mathcal O_X(m(K_{X/Y}+\Delta))$ and define the relative evaluation ideal on $X_{Y_0}$ by
\[
 \operatorname{Im}(f^*E_m\longrightarrow L_m)=\mathfrak b_mL_m.
\]
The base-change identifications show that its restriction to $X_y$ is the base ideal of $|m(K_{X_y}+\Delta_y)|$. The evaluation map is surjective on $X_y$ exactly when $m(K_{X_y}+\Delta_y)$ is globally generated, or equivalently when this complete linear system is base-point-free. This fiberwise condition is distinct from generic global generation of $E_m$ by sections on the base. Positive rank and the same identifications ensure that $\mathfrak b_m$ is not identically zero on any fiber.

Choose a simultaneous log resolution $\pi:\widetilde X\to X_{Y_0}$ principalizing $\mathfrak b_m$ and resolving the boundary, with SNC divisor $D$ containing the evaluation divisor, the pulled-back boundary and every exceptional divisor \cite[Theorem~1.0.1]{Wlodarczyk2004}. Put $p=f\pi$. Let $\mathcal S$ consist of $\widetilde X$ and the irreducible components of all nonempty intersections of components of $D$. These are finitely many smooth closed subvarieties. Define
\begin{equation}\label{eq:regularity-open}
 U_m^{\mathrm{reg}}=Y_0\setminus\left(
 \bigcup_{\substack{S\in\mathcal S\\p(S)\ne Y_0}}p(S)
 \ \cup\!
 \bigcup_{\substack{S\in\mathcal S\\p(S)=Y_0}}
       p\bigl(\operatorname{Crit}(p|_S)\bigr)\right).
\end{equation}
We call $U_m^{\mathrm{reg}}$ the \emph{regularity locus} associated with the chosen resolution. This choice remains fixed when smaller subsets are used to study variation.

Every restriction $p|_S$ is proper. Its critical locus is closed, so all excluded images in \eqref{eq:regularity-open} are closed. For a dominating stratum, its generic fiber is smooth in characteristic zero, and its critical locus does not dominate $Y_0$. Consequently $U_m^{\mathrm{reg}}$ is a dense Zariski-open subset of $Y_0$. Over it, $p$ and every divisor intersection are submersions. The SNC equations therefore extend to relative coordinates $(y,w)$ in which $p(y,w)=y$ and the components of $D$ are coordinate divisors $w_E=0$.

Each $\pi_y:\widetilde X_y\to X_y$ is birational. Indeed, $\pi_y$ is surjective with connected fibers, as follows from the same properties of the proper birational map $\pi$. Since $X_y$ is connected, $\widetilde X_y$ is connected. It is smooth by construction and hence irreducible. The exceptional divisors meet it in divisors or not at all, so the isomorphism locus of $\pi$ meets it densely. Thus fiberwise change of variables applies to $\pi_y$.

\begin{theorem}[Smoothness of the fixed-point metric]\label{thm:controlled-regularity}
In either family setting above, $(y,A)\mapsto T_y(A)$ and $G_y$ are smooth on $U_m^{\mathrm{reg}}$. At a fixed point the inverse of $I-D_AT_y$ has order norm at most $m$. After factoring the evaluation divisor on the principalization, the Fubini--Study weight has a smooth residual part. It is smooth on $X_{U_m^{\mathrm{reg}}}$ off the relative base locus. In the unweighted case, if the relative evaluation map is surjective, these assertions hold on the whole smooth base. In the weighted case the assertion should be restricted to $U_m^{\mathrm{reg}}$.
\end{theorem}
\begin{proof}
In a relative resolution chart over $U_m^{\mathrm{reg}}$, write
\[
 F\circ\pi=v\prod_Ew_E^{a_E}U,\qquad
 J=v_J\prod_Ew_E^{k_E},\qquad
 \pi^*\Delta=\sum_Eb_EE,
\]
where $v,v_J$ are units and the holomorphic row $U$ has no common zero. The pulled-back matrix density is
\begin{equation}\label{eq:resolved-kernel}
 \prod_E|w_E|^{2\gamma_E}
 U^*U(UA^{-1}U^*)^{-\alpha}\rho(y,w)\,dV_w,
 \qquad \gamma_E=\frac{a_E}{m}+k_E-b_E,
\end{equation}
where $\rho$ is smooth and positive. The numerator contributes $2a_E$, the denominator $-2\alpha a_E$, and the Jacobian $2k_E$. Thus the remaining evaluation exponent is $2a_E/m$, and
\begin{equation}\label{eq:weighted-exponent}
 \gamma_E+1=\frac{a_E}{m}+A_{X,\Delta}(E)>0.
\end{equation}
These exponents therefore give an integrable density. For $\Delta=0$ there are no negative exponents.

Choose nested relative charts covering a fiber, shrink the base, and use properness to make the smaller charts cover the entire tube over a relatively compact base neighborhood. A subordinate smooth partition with a base cutoff expresses $T_y(A)$ as a finite sum of
\begin{equation}\label{eq:localized-integral}
 \int_W\prod_E|w_E|^{2\gamma_E}b(y,w,A)\,dV_w
\end{equation}
on fixed bounded polydiscs, with smoothly extended compactly supported amplitudes. On their supports, for $\lambda_-I\leq A\leq\lambda_+I$,
\begin{equation}\label{eq:residual-lower-bound}
 0<\frac{\min|U|^2}{\lambda_+}
 \leq Q_A:=UA^{-1}U^*
 \leq\frac{\max|U|^2}{\lambda_-}.
\end{equation}
The residual amplitude is smooth near the compact chart and parameter supports, and its denominator is bounded away from zero by \eqref{eq:residual-lower-bound}. Hence, for every pair of real multi-indices $I,J$, its mixed derivatives satisfy $\|D_y^I D_A^J b(y,w,A)\|\leq C_{I,J}$. The monomial in \eqref{eq:localized-integral} is independent of $y$ and $A$ and integrable by \eqref{eq:weighted-exponent}. Thus each differentiated integrand is dominated by $C_{I,J}\prod_E|w_E|^{2\gamma_E}$, and dominated differentiation proves joint smoothness.

Differentiating with respect to the matrix gives
\begin{equation}\label{eq:matrix-derivative}
 D_AT_y(A)[S]=\alpha\int_{X_y}F^*F B_A^{-\alpha-1}
 (FA^{-1}SA^{-1}F^*)\,d\mu_\Delta.
\end{equation}
It is a positive real-linear map on Hermitian matrices. At a fixed point $A$ homogeneity gives $D_AT_y(A)[A]=\alpha A$, hence its order norm is at most $\alpha$. Therefore
\begin{equation}\label{eq:implicit-inverse}
 (I-D_AT_y(A))^{-1}=\sum_{j\geq0}(D_AT_y(A))^j,
 \qquad\|(I-D_AT_y(A))^{-1}\|_{A,\infty}\leq m.
\end{equation}
The real implicit-function theorem applied to $A-T_y(A)$ proves smoothness of the unique fixed point. Finally
\[
 \varphi_{G_y}\circ\pi
 =\frac1m\log\left|v\prod_Ew_E^{a_E}\right|^2
   +\frac1m\log(UG_y^{-1}U^*)
\]
has the asserted smooth residual part. In the unweighted case, global generation of $mK_{X_y}$ persists near $y$: the base-change isomorphisms identify the fiber evaluation maps, and the proper image of the evaluation cokernel's support is closed. On that neighborhood the row $F$ has no common zero, so $B_A$ is uniformly bounded below on compact chart supports and compact positive-matrix sets. The same differentiation and implicit-function argument apply wherever the relative evaluation map is surjective.
\end{proof}

The normal effective-klt construction uses the effective resolution model of Proposition~\ref{prop:birational-invariance}. For a family, the same argument applies on any open subset admitting a simultaneous resolution, fiberwise base-change isomorphisms, and smooth relative divisor strata. Its exceptional fixed factor is included in the integers $a_E$ in \eqref{eq:resolved-kernel}. The resulting smoothness on the regularity locus justifies the weighted curvature calculation.

\subsection{Refined regularity}
In this subsection $\Delta=0$ and $f:X\to Y$ is smooth and projective with connected fibers. Assume that $E_m$ is locally free of positive rank and that the fiberwise base-change maps in \eqref{eq:complete-base-change} are isomorphisms. Fix $y_0\in Y$. Let $\mathfrak b_{m,y_0}$ be the base ideal of the complete linear system $|mK_{X_{y_0}}|$, and denote its minimum log canonical threshold by
\begin{equation}\label{eq:central-threshold}
 c_{y_0}=\inf_{x\in X_{y_0}}
           \operatorname{lct}_x(\mathfrak b_{m,y_0}).
\end{equation}
Our convention is that $|F_{y_0}|^{-2c}$ is locally integrable for $c<c_{y_0}$. The value is positive, since the ideal is nonzero and the fiber is smooth and compact. We set $c_{y_0}=+\infty$ when $|mK_{X_{y_0}}|$ is base-point-free.

\begin{theorem}[Threshold regularity]\label{thm:threshold-regularity}
Let $k\geq0$ be an integer and $0<\gamma<1$. If
\begin{equation}\label{eq:threshold-regularity-range}
 k+\gamma<\frac2m+2c_{y_0},
\end{equation}
then $(y,A)\mapsto T_y(A)$ is jointly $C^{k,\gamma}$ near $(y_0,G_{y_0})$, and $G$ is $C^{k,\gamma}$ near $y_0$. Its Fubini--Study weight has the same regularity off the relative base locus.
\end{theorem}
The proof is given in Appendix~\ref{app:refined-regularity}.

\begin{corollary}\label{cor:threshold-C2}
Under the hypotheses of Theorem~\ref{thm:threshold-regularity}, if $c_{y_0}>1-1/m$, then $G$ is $C^{2,\gamma}$ near $y_0$ for some $0<\gamma<1$.
\end{corollary}
\begin{proof}
Choose $0<\gamma<\min\{1,2(c_{y_0}-1+1/m)\}$ and take $k=2$ in \eqref{eq:threshold-regularity-range}.
\end{proof}

We should mention that the threshold is only a sufficient condition: a persistent base divisor in a product family can give finite $c_{y_0}$ even though $G$ is constant. Smoothness on $U_m^{\mathrm{reg}}$ follows independently by differentiating the smooth amplitude in \eqref{eq:localized-integral}.

\section{Curvature resolvent and determinant pinching}\label{sec:curvature}

Let $U\subseteq U_m^{\mathrm{reg}}$ be a analtyic open subset, let $E=E_m$ have rank $r$, and put $d=\dim U$. We work in either the unweighted or the effective-klt weighted setting of Theorem~\ref{thm:controlled-regularity}, with the fiberwise base-change isomorphisms. Differentiating the fixed-point equation will express the curvature of $G$ as a completely positive resolvent applied to the curvature of an auxiliary metric.

\subsection{Curvature conventions and the frozen metric}
For a squared-norm matrix $A$, our Chern connection and curvature are
\begin{equation}\label{eq:chern-convention}
 \theta=A^{-1}\partial A,\qquad \Theta=\bar\partial\theta,
 \qquad
 \Theta_{a\bar b}=A^{-1}A_{\bar b}A^{-1}A_a-A^{-1}A_{a\bar b},
\end{equation}
where the last coefficients are written against $dy_a\wedge d\bar y_b$. In particular $A=e^{-\varphi}$ gives $\Theta_{a\bar b}=\varphi_{a\bar b}$. At a point with $A=I$, the Nakano form evaluated on $\sum_a\partial_{y_a}\otimes u_a$ is
\begin{equation}\label{eq:nakano-block}
 \sum_{a,b}A(\Theta_{a\bar b}u_a,u_b)
 =\sum_{a,b}u_a^*\Theta_{b\bar a}u_b.
\end{equation}
We denote this curvature tensor by $K=[\Theta_{b\bar a}]_{a,b=1}^d$, viewed as a Hermitian form on $T_y^{1,0}Y\otimes E_y$. Its matrix uses reversed base indices and the usual row and column order for bundle indices.

Fix a point $y_0$ and a holomorphic frame on a neighborhood, with $G(y_0)=I$. Holding that frame and the matrix $I$ fixed, define the \emph{frozen metric}
\[
 H_{y_0}(y)=T_y(I).
\]
The construction uses the full section space $E_y$ on each fiber.

\begin{proposition}\label{prop:frozen-nakano}
The metric $H_{y_0}$ is smooth and Nakano semipositive on a neighborhood of $y_0$ in $U$. This assertion holds for the effective-klt weighted family as well as for $\Delta=0$.
\end{proposition}
\begin{proof}
Work over a bounded coordinate polydisc $V\Subset U$. Write $L_m=m(K_{X/V}+\Delta)$ and $L=L_m-K_{X/V}=(m-1)K_{X/V}+m\Delta$. The frozen metric on $L$ has weight
\begin{equation}\label{eq:frozen-weight}
 \psi=\alpha\log(FF^*)+\sum_j\delta_j\log|d_j|^2.
\end{equation}
Indeed $L=\alpha L_m+\Delta$ as rational line bundles. The first term is the weight of an evaluation quotient of a trivial Hermitian bundle, and the second is the effective divisor metric. Both have semipositive curvature currents. Its adjoint $L^2$ pairing is exactly $T_y(I)$.

To obtain a residual metric with logarithmic coefficients in $[0,1)$, use the simultaneous principalization from Theorem~\ref{thm:controlled-regularity}. Let $a_E,b_E,k_E$ be the orders of evaluation, boundary and Jacobian, respectively. Then
\[
 \psi\circ\pi=\sum_Ec_E\log|w_E|^2+q,
 \quad c_E=\alpha a_E+b_E,
 \quad q=\alpha\log(UU^*)+\text{pluriharmonic unit terms}.
\]
Set
\begin{equation}\label{eq:divisor-absorption}
 C=\sum_E\lfloor c_E\rfloor E,\qquad
 L^\sharp=\pi^*L-C,\qquad
 \psi^\sharp=\sum_E\{c_E\}\log|w_E|^2+q.
\end{equation}
The fractional residues lie in $[0,1)$, and
\[
 i\Theta(L^\sharp)=\sum_E\{c_E\}\,i\partial\bar\partial
 \log|w_E|^2+\alpha i\partial\bar\partial\log(UU^*)\geq0.
\]
Here $\mathcal I(e^{-\psi})$ consists of the holomorphic germs $h$ for which $|h|^2e^{-\psi}$ is locally integrable. The one-variable integral $\int_0^\eta t^{2k-2c+1}dt$ is finite precisely when $k>c-1$. Thus
\begin{equation}\label{eq:multiplier-absorption}
 \mathcal I(e^{-\psi\circ\pi})=\mathcal O_{\widetilde X}(-C),
 \qquad \mathcal I(e^{-\psi^\sharp})=\mathcal O_{\widetilde X}.
\end{equation}
The fixed relative SNC coordinates give the same calculation on every fiber, so each residual multiplier ideal is trivial.

To identify the full adjoint section spaces and their $L^2$ pairing, observe that on $X$,
\[
 |f_i|^2(FF^*)^{-\alpha}\prod_j|d_j|^{-2\delta_j}
 \leq |F|^{2/m}\prod_j|d_j|^{-2\delta_j}
\]
is locally integrable by the klt hypothesis, including after multiplication by a smooth base volume form. Every local section of $E$ is a holomorphic linear combination of the frame sections. Hence, as sheaves on $V$,
\begin{equation}\label{eq:multiplier-no-loss}
 f_*\bigl(\mathcal O_X(L_m)\otimes\mathcal I(e^{-\psi})\bigr)=E|_V.
\end{equation}
The multiplier transformation formula \cite[Proposition~5.8 and Remark~5.9]{Demailly2012} reads in this convention
\[
 \mathcal I(e^{-\psi})
 =\pi_*\bigl(\mathcal O(K_{\widetilde X/X})
                   \otimes\mathcal I(e^{-\psi\circ\pi})\bigr).
\]
Demailly uses $e^{-2\varphi}$, with $\varphi=\psi/2$ here. Projection formula and \eqref{eq:multiplier-absorption}-- \eqref{eq:multiplier-no-loss} give
\begin{equation}\label{eq:resolved-direct-image}
 p_*\mathcal O_{\widetilde X}(K_{\widetilde X/V}+L^\sharp)=E|_V.
\end{equation}
Apply the same multiplier formula to each birational morphism $\pi_y$ and repeat the integrability estimate for every element of $H^0(X_y,L_m|_{X_y})$. This also identifies the full fiber spaces:
\begin{equation}\label{eq:complete-adjoint-space}
 H^0(\widetilde X_y,K_{\widetilde X_y}+L^\sharp|_{\widetilde X_y})
 =H^0(X_y,m(K_{X_y}+\Delta_y)).
\end{equation}

Holomorphy of the weighted adjoint pullback follows directly from the divisor orders. An adjoint pullback has coefficient $J(f_i\circ\pi)$ in a mixed canonical and pulled-back $L$ frame. After division by the equation $s_C$ of $C$, its order along $E$ is at least
\begin{equation}\label{eq:weighted-no-loss-order}
 a_E+k_E-\lfloor\alpha a_E+b_E\rfloor\geq0,
\end{equation}
because
\[
 \alpha a_E+b_E
 =a_E+k_E+1-\left(\frac{a_E}{m}+A_{X,\Delta}(E)\right)
 <a_E+k_E+1.
\]
The corresponding pairing satisfies
\begin{equation}\label{eq:frozen-isometry}
 \left|\frac{J(f_i\circ\pi)}{s_C}\right|^2e^{-\psi^\sharp}
 =|J|^2|f_i\circ\pi|^2e^{-\psi\circ\pi}.
\end{equation}
The Jacobian gives the change-of-variables factor, so polarization identifies all entries of $H_{y_0}$ with the $L^2$ metric on \eqref{eq:resolved-direct-image}. Smoothness follows from \eqref{eq:resolved-kernel}--\eqref{eq:localized-integral}.

Zou's theorem \cite[Case Z, Theorem~1.1 and Example~1.3(3)]{Zou2021} gives Nakano semipositivity for the complete adjoint $L^2$ metric of a proper K\"ahler submersion over a bounded pseudoconvex base, provided the line metric is semipositive, has a smooth residual weight and relative SNC logarithmic poles with fixed residues in $[0,1)$, and satisfies the hypothesis of Case Z concerning a Stein open subset. The reference is to arXiv v2. Its Theorems~2.7--2.8 give the simultaneous base-direction estimate and its Nakano interpretation. In this application, $p$ is projective and smooth, the divisors are relative SNC, \eqref{eq:divisor-absorption} has the required residues and curvature, and \eqref{eq:complete-adjoint-space} is the complete fiber space. To obtain the auxiliary Stein open subset, embed the projective family in $\mathbb P^N\times V$ and remove a hyperplane containing no fiber component after shrinking $V$. The remaining closed subspace of $\mathbb C^N\times V$ is Stein. The removed divisor also supplies the required negligible analytic set. Zou's theorem therefore proves Nakano semipositivity on the entire tangent--fiber tensor product. When $\Delta=0$ and the relative evaluation map is surjective, the smooth theorem applies directly \cite[Theorem~1.2]{Berndtsson2009}.
\end{proof}

\subsection{The curvature resolvent formula}
A complex-linear map on matrices is \emph{completely positive} if its entrywise action preserves positive block matrices of every size. The following matrix estimate will apply to the linearization of the fixed-point equation.

\begin{lemma}[Pinching for completely positive resolvents]
\label{lem:cp-block-pinching}
Let $H,V$ be finite-dimensional Hermitian spaces, with $r=\dim V\geq1$, and let $L:\operatorname{End}(V)\to\operatorname{End}(V)$ be complex-linear and completely positive, with $L(I)=\alpha I$ and $0<\alpha<1$. Then
\[
 \|L\|=\alpha,\qquad R=(I-L)^{-1}=\sum_{j\geq0}L^j,
 \qquad R(I)=(1-\alpha)^{-1}I,
\]
where the norm is induced by the operator norm on complex matrices, and $R$ and $R-I$ are completely positive. For $Q\geq0$ on $H\otimes V$, put
\[
 K=(\mathrm{id}_{\operatorname{End}(H)}\otimes R)Q,
 \qquad \tau=\operatorname{Tr}_V K.
\]
Then $K\geq0$. If, in addition, $L-\varepsilon D$ is completely positive for some $\varepsilon>0$, where $D(S)=\operatorname{Tr}(S)I$, then, with $\sigma=\varepsilon/(1-\alpha+r\varepsilon)$,
\begin{equation}\label{eq:abstract-block-pinching}
 0<\sigma\leq\frac\alpha r,\qquad
 \sigma\tau\otimes I\leq K
 \leq\bigl[r-(r^2-1)\sigma\bigr]\tau\otimes I,
\end{equation}
and
\[
 \ker K=(\ker\tau)\otimes V,\qquad
 \operatorname{rank}K=r\operatorname{rank}\tau.
\]
The inequalities hold as Hermitian forms on the entire space $H\otimes V$. Under the separate trace-scaling assumption $\operatorname{Tr}L(S)=\alpha\operatorname{Tr}S$ for every $S$, one also has
\[
 \operatorname{Tr}_V Q=(1-\alpha)\tau.
\]
\end{lemma}
\begin{proof}
For any completely positive map $\Phi$ with $\Phi(I)=cI$, positivity gives $c\geq0$. Apply $\mathrm{id}_{M_2}\otimes\Phi$ to the positive block matrix
\[
 \begin{pmatrix}\|S\|I&S\\S^*&\|S\|I\end{pmatrix}.
\]
The resulting positive block matrix has diagonal entries $c\|S\|I$, so $\|\Phi(S)\|\leq c\|S\|$. Equality on $I$ gives $\|\Phi\|=c$. In particular $\|L\|=\alpha<1$. The Neumann series defining $R$ and $R-I$ converge in norm and are completely positive, since composition, sums and norm limits preserve positive block matrices of every size. Their scalar action on $I$ gives the displayed formula for $R(I)$. Consequently $K\geq0$ and $\tau\geq0$.

Now put $N=L-\varepsilon D$. It is completely positive and $N(I)=(\alpha-r\varepsilon)I$, so $0\leq b:=\alpha-r\varepsilon<1$. The preceding argument gives $\|N\|=b$. Hence $R_N=\sum_{j\geq0}N^j$ is completely positive and $R_N(I)=(1-\alpha+r\varepsilon)^{-1}I$. Also $r\varepsilon\leq\alpha$, which, since $1-\alpha>0$, is equivalent to $\sigma\leq\alpha/r$.

The equation $K=Q+(\mathrm{id}\otimes L)K$ becomes
\[
 (\mathrm{id}\otimes(I-N))K=Q+\varepsilon\tau\otimes I.
\]
Applying $\mathrm{id}\otimes R_N$ and then taking the partial trace gives
\begin{equation}\label{eq:pinching-decomposition}
 K=B+\sigma\tau\otimes I,
 \qquad B=(\mathrm{id}\otimes R_N)Q\geq0,
 \qquad\operatorname{Tr}_V B=(1-r\sigma)\tau.
\end{equation}
Every positive semidefinite endomorphism $B$ of $H\otimes V$ satisfies
\begin{equation}\label{eq:partial-trace-bound}
 B\leq r(\operatorname{Tr}_V B)\otimes I.
\end{equation}
To prove this, it suffices to take $B=ww^*$. A Schmidt decomposition gives $w=\sum_{j=1}^s a_je_j\otimes v_j$ with $s\leq r$, orthonormal vectors in each factor and $a_j\geq0$. For $z=\sum_{j=1}^s e_j\otimes z_j+z^\perp$, Cauchy--Schwarz yields
\[
 |\langle z,w\rangle|^2
 \leq s\sum_{j=1}^s a_j^2|\langle z_j,v_j\rangle|^2
 \leq r\sum_{j=1}^s a_j^2\|z_j\|^2.
\]
The last sum is the quadratic form of $r(\operatorname{Tr}_V ww^*)\otimes I$. Summing rank-one decompositions proves \eqref{eq:partial-trace-bound}. Applied to \eqref{eq:pinching-decomposition}, this gives the upper coefficient $r(1-r\sigma)+\sigma=r-(r^2-1)\sigma$. The lower bound follows from $B\geq0$. Both coefficients are positive, so the two bounds give the asserted kernel and rank.

Finally, under the trace-scaling assumption, taking the partial trace of $K=Q+(\mathrm{id}\otimes L)K$ gives $\tau=\operatorname{Tr}_V Q+\alpha\tau$. This assumption is used only for this last identity.
\end{proof}

At $y_0$, choose a Chern normal holomorphic frame for $G$, so $G(y_0)=I$ and $G_a(y_0)=0$. Define the linearized operator $L$, its resolvent $R$, and the frozen curvature tensor $Q$ by
\begin{equation}\label{eq:L-and-R}
 L=D_AT_{y_0}|_I,\qquad R=(I-L)^{-1},\qquad
 Q=[\Theta^{H_{y_0}}_{b\bar a}(y_0)].
\end{equation}
The real derivative on Hermitian matrices has a unique complex-linear extension to all of $M_r(\mathbb C)$, denoted by the same letter.

\begin{lemma}\label{lem:cp-resolvent}
With the notation \eqref{eq:evaluation-moment} on $X_{y_0}$,
\begin{equation}\label{eq:moment-operator}
 L(S)=\alpha\int P_xSP_x\,d\nu,
 \quad\int P_xd\nu=I,\quad\nu(X_{y_0})=r.
\end{equation}
The maps $L$, $R$, and $R-I$ are completely positive, and
\begin{equation}\label{eq:operator-normalizations}
 L(I)=\alpha I,\quad \operatorname{Tr}L(S)=\alpha\operatorname{Tr}S,
 \quad\|L\|=\alpha,\quad R=\sum_{j\geq0}L^j,\quad R(I)=mI.
\end{equation}
Here the map norm is induced by the operator norm on complex matrices.
\end{lemma}
\begin{proof}
Formula \eqref{eq:matrix-derivative} at $I$ gives $\alpha\int F^*F(FSF^*)(FF^*)^{-\alpha-1}d\mu_\Delta$. Since $P$ has rank one, this is \eqref{eq:moment-operator}. For every positive block matrix $S=[S_{ab}]$ of size $q$,
\[
 [L(S_{ab})]=\alpha\int(I_q\otimes P_x)S(I_q\otimes P_x)d\nu\geq0.
\]
Thus $L$ is completely positive. Since $P^2=P$, $L(I)=\alpha I$, and cyclicity of trace gives the trace identity. The norm and resolvent assertions now follow from the first part of Lemma~\ref{lem:cp-block-pinching}, with $(1-\alpha)^{-1}=m$.
\end{proof}

\begin{theorem}[Curvature resolvent formula]\label{thm:curvature-resolvent}
The metric $G$ is Nakano semipositive on $U$. In the Chern-normal convention above,
\begin{equation}\label{eq:curvature-resolvent}
 K=(\mathrm{id}_{M_d}\otimes R)Q,\qquad
 K=Q+(\mathrm{id}_{M_d}\otimes L)K,\qquad Q\geq0.
\end{equation}
In particular, if $\tau=\operatorname{Tr}_E K$, then
\begin{equation}\label{eq:frozen-trace}
 \operatorname{Tr}_E Q=(1-\alpha)\tau=\frac1m\tau.
\end{equation}
\end{theorem}
\begin{proof}
Write $A(y)=G_y$ in the chosen holomorphic frame and differentiate $A=T_y(A)$. Theorem~\ref{thm:controlled-regularity} justifies these derivatives. Partial $y$-derivatives of $T$ hold the matrix argument fixed. The chain rule yields
\[
 (I-D_AT)A_a=T_a
\]
and
\begin{align*}
 (I-D_AT)A_{a\bar b}
 ={}&T_{a\bar b}+(D_AT_a)[A_{\bar b}]
     +(D_AT_{\bar b})[A_a]
     +D_A^2T[A_a,A_{\bar b}].
\end{align*}
At $y_0$, $A_a=A_{\bar b}=0$, so $T_a=T_{\bar b}=0$ and the last three terms vanish. The frozen metric $H_{y_0}$ is therefore also Chern-normal there. By \eqref{eq:chern-convention},
\[
 \Theta^G_{a\bar b}=-A_{a\bar b}
 =(I-L)^{-1}(-T_{a\bar b})
 =(I-L)^{-1}\Theta^{H_{y_0}}_{a\bar b}.
\]
Reversing the base indices as in \eqref{eq:nakano-block} proves the first two identities in \eqref{eq:curvature-resolvent}. Proposition~\ref{prop:frozen-nakano} gives $Q\geq0$ on the entire tangent--fiber tensor product, and Lemma~\ref{lem:cp-resolvent} preserves this positivity under $\mathrm{id}_{M_d}\otimes R$. The trace assertion of Lemma~\ref{lem:cp-block-pinching}, using \eqref{eq:operator-normalizations}, gives \eqref{eq:frozen-trace}.
\end{proof}

\subsection{Nondegeneracy of the evaluation quadratic form}
To obtain a completely positive lower bound for $L$, define the \emph{evaluation quadratic form} on all of $M_r(\mathbb C)$:
\begin{equation}\label{eq:full-moment}
 \mathfrak q_y(M)=\alpha\int_{X_y}
       |\operatorname{Tr}(M^*P_x)|^2d\nu,
 \qquad
 \varepsilon_y=\min_{\|M\|_{\mathrm{HS}}=1}\mathfrak q_y(M),
\end{equation}
where the matrices are written in a $G_y$-orthonormal frame and $\|M\|_{\mathrm{HS}}^2=\operatorname{Tr}(M^*M)$. Unitary conjugation shows that $\varepsilon_y$ is intrinsic.

\begin{lemma}\label{lem:strict-cp}
For every $y\in U$, $\varepsilon_y>0$, and
\begin{equation}\label{eq:cp-minorization}
 L-\varepsilon_y D\text{ is completely positive},
 \qquad D(S)=\operatorname{Tr}(S)I.
\end{equation}
The function $\varepsilon_y$ is continuous and positive on $U$. In particular $L$ and $R$ send every nonzero positive semidefinite endomorphism to a positive definite one.
\end{lemma}
\begin{proof}
If $\mathfrak q_y(M)=0$, choose a coordinate ball disjoint from the base locus and the boundary. Its density is positive, so continuity gives $F(z)M^*F(z)^*=0$ on that ball. Expand $F(z)=\sum_\gamma a_\gamma z^\gamma$ at an interior point. Mixed holomorphic and antiholomorphic derivatives give $a_\gamma M^*a_\eta^*=0$ for every $\gamma,\eta$. Thus $F(z)M^*F(w)^*=0$ for independent $z,w$ in a smaller ball. The evaluation rows on that ball span $\mathbb C^r$: a vector annihilating them gives a global section vanishing on an open subset of the connected fiber, hence the zero section. It follows that $M=0$. Compactness of the Hilbert--Schmidt unit sphere proves $\varepsilon_y>0$.

To prove complete positivity, decompose an arbitrary positive block matrix into rank-one terms and consider $S_{ab}=v_av_b^*$. For test vectors $u_a$ put $M=\sum_a u_av_a^*$. Direct calculation gives
\begin{align*}
 \sum_{a,b}u_a^*L(v_av_b^*)u_b
 &=\alpha\int\left|\sum_a u_a^*P_xv_a\right|^2d\nu
 =\mathfrak q_y(M),\\
 \sum_{a,b}u_a^*D(v_av_b^*)u_b
 &=\sum_{a,b}(v_b^*v_a)(u_a^*u_b)=\|M\|_{\mathrm{HS}}^2.
\end{align*}
The difference is nonnegative. This holds for every block size and proves \eqref{eq:cp-minorization}. In particular $L(S)\geq\varepsilon_y\operatorname{Tr}(S)I$ for $S\geq0$. The same improvement follows for $R$ from $R=I+L+L^2+\cdots$.

On $U$, the projection $P$ is unchanged when $F$ is replaced by its nowhere-zero residual row. The cancellation in \eqref{eq:resolved-kernel} therefore writes the integrand defining $\mathfrak q_y$ as a fixed integrable monomial times a smooth function of $y$ and $M$. It follows that $\mathfrak q_y$ varies continuously, uniformly on the matrix unit sphere. Its minimum $\varepsilon_y$ is consequently continuous.
\end{proof}

\begin{proposition}
\label{prop:compact-moment}
Suppose $\Delta=0$ and $f:X\to Y$ is smooth projective with connected fibers, $E_m$ has positive rank and satisfies the fiberwise base-change isomorphisms, and $Y$ is compact. The evaluation quadratic form \eqref{eq:full-moment} is defined and continuous on the endomorphism bundle over all of $Y$, including points where the evaluation ideal changes. In particular $\varepsilon_y$ extends continuously to $Y$, and
\begin{equation}\label{eq:compact-moment-constant}
 \varepsilon_*:=\min_{y\in Y}\varepsilon_y>0.
\end{equation}
With this constant, $\mathfrak q_y(M)\geq\varepsilon_*\|M\|_{\mathrm{HS},G_y}^2$ for all $y$ and $M$.
\end{proposition}
\begin{proof}
Lemma~\ref{lem:holder-continuity} gives a continuous positive metric $G$ on $Y$. Apply $G^{-1/2}$ to a holomorphic frame to obtain a continuous $G$-orthonormal frame. In relative coordinates its evaluation row $F$ is continuous. The integrand defining $\mathfrak q_y$ is
\begin{equation}\label{eq:moment-continuity-density}
 \alpha|FM^*F^*|^2|F|^{-4+2/m}dV_n
 \quad(F\ne0).
\end{equation}
Its absolute value is bounded by $\alpha\|M\|_{\mathrm{HS}}^2|F|^{2/m}dV_n$, so extending it by zero at $F=0$ makes it jointly continuous in $y,w,M$. Relative charts and compactly supported cutoffs as in Lemma~\ref{lem:holder-continuity} now give continuous fiber integrals, uniformly for $M$ in a bounded set. The minimum over the matrix unit sphere is therefore continuous in $y$, even when the base ideal changes. Positivity at each point follows from the same polarization argument as above, which needs no base derivatives. The Hilbert--Schmidt unit sphere bundle of $\operatorname{End}E$ over compact $Y$ is compact. The continuous positive function $\mathfrak q$ on it has a positive minimum.
\end{proof}

\subsection{Determinant pinching}
The determinant curvature $\tau=\operatorname{Tr}_E K$ is a semipositive Hermitian form on $T_y^{1,0}Y$. Its real $(1,1)$ form is $i\operatorname{Tr}\Theta^G$, and $c_1(E,G)=i\operatorname{Tr}\Theta^G/(2\pi)$.

\begin{theorem}[Determinant pinching]\label{thm:determinant-pinching}
Choose $0<\varepsilon\leq\varepsilon_y$ and put
\begin{equation}\label{eq:sigma-definition}
 \sigma=\frac{\varepsilon}{1-\alpha+r\varepsilon}.
\end{equation}
Then $0<\sigma\leq\alpha/r$, and at $y\in U$ the curvature tensor satisfies the inequalities in the Nakano sense:
\begin{equation}\label{eq:determinant-pinching}
 \sigma\tau\otimes I\leq K
 \leq\bigl[r-(r^2-1)\sigma\bigr]\tau\otimes I.
\end{equation}
In particular
\begin{equation}\label{eq:analytic-kernel}
 \ker K=(\ker\tau)\otimes E_y,\qquad
 \operatorname{rank}K=r\operatorname{rank}\tau.
\end{equation}
Thus $\tau>0$ is equivalent to strict Nakano positivity of $G$ at $y$. The constants may be chosen locally uniformly on $U$. Under the compact unweighted smooth-family hypotheses of Proposition~\ref{prop:compact-moment}, one may use $\sigma_*={\varepsilon_*}/({1-\alpha+r\varepsilon_*})>0$ at every point of $U$.
\end{theorem}
\begin{proof}
Apply Lemma~\ref{lem:cp-block-pinching} with $H=T_y^{1,0}Y$, $V=E_y$ and $L=D_AT_y|_I$ in a $G$-Chern-normal frame. Lemma~\ref{lem:cp-resolvent} verifies complete positivity and both normalizations in \eqref{eq:operator-normalizations}. The frozen curvature tensor $Q$ is Nakano semipositive by Proposition~\ref{prop:frozen-nakano}, and differentiating the fixed-point equation gives $K=(\mathrm{id}\otimes(I-L)^{-1})Q$ by Theorem~\ref{thm:curvature-resolvent}.

Lemma~\ref{lem:strict-cp} verifies the required minorization at $\varepsilon_y$. The map $D$ is completely positive: in matrix units, $D(S)=\sum_{i,j}E_{ij}SE_{ij}^*$. Hence, for $0<\varepsilon\leq\varepsilon_y$, the decomposition
\[
 L-\varepsilon D=(L-\varepsilon_yD)+(\varepsilon_y-\varepsilon)D
\]
verifies it at $\varepsilon$. The resulting bounds and kernel formula are \eqref{eq:determinant-pinching}--\eqref{eq:analytic-kernel}. Continuity and positivity of $\varepsilon_y$ in Lemma~\ref{lem:strict-cp} give locally uniform choices. Proposition~\ref{prop:compact-moment} supplies $\varepsilon_*$ on the compact unweighted base, giving the stated $\sigma_*$.
\end{proof}

\begin{corollary}\label{cor:directional-nullity}
For a tangent vector $\xi$ at a point of $U$, let $K_\xi$ and $Q_\xi$ be the contractions of the curvature tensors $K$ and $Q$ with $\xi$ in the base factor. Then the following are equivalent:
\[
 \tau(\xi,\xi)=0,\qquad K_\xi=0,\qquad Q_\xi=0.
\]
If $Q_\xi\ne0$, then $K_\xi$ is positive definite.
\end{corollary}
\begin{proof}
The matrices $K_\xi,Q_\xi$ are positive semidefinite, and $K_\xi=R(Q_\xi)$. Positivity improvement gives the last assertion. The trace identity \eqref{eq:frozen-trace} gives the equivalences.
\end{proof}

Section~\ref{sec:variation} will identify these curvature null spaces geometrically. If $G$ is $C^2$, continuity extends the pinching inequalities from $U_m^{\mathrm{reg}}$. The threshold criterion of Theorem~\ref{thm:threshold-regularity} gives this optional global differentiability in Proposition~\ref{prop:global-curvature}. 

\section{Strict positivity and geometric variation}\label{sec:variation}

Throughout this section, $f:X\to U$ is an unweighted smooth projective family with smooth connected fibers, and $m\geq2$. The direct image $E=E_m$ is locally free, and the fiberwise base-change maps $E_y\to H^0(X_y,mK_{X_y})$ are isomorphisms. We compute curvature on the fixed regularity locus $U_m^{\mathrm{reg}}$ of \eqref{eq:regularity-open}, and also wherever the relative evaluation map $f^*E\to\omega_{X/U}^{\otimes m}$ is surjective. Retain $\alpha=(m-1)/m$, the fixed-point metric $G$, and the frozen metric $H_y=T_\bullet(I)$ in a $G$-Chern-normal frame at $y$. As in Section~\ref{sec:curvature}, write $K$ for the curvature tensor, $Q$ for the frozen curvature tensor, and $\tau=\operatorname{Tr}_E K$ for the determinant curvature.

The $m$-canonical map $\varphi_m$ is defined by the complete linear system $|mK_{X_y}|$ and takes values in $\mathbb P_q(E_y)$, equivalently in the space of lines in $E_y^*$. Birational variation $\operatorname{Var}(f)$ is the least transcendence degree over $\mathbb C$ of a field of definition of a birational model of the geometric generic fiber, allowing finite extensions. For general-type fibers it equals the dimension of the birational classifying image. The curvature conclusions concern $E_m$ and $G_m$. An auxiliary multiple $N=km$ will be used only to identify their curvature rank in case (G).

\subsection{The globally generated case}
Write $\rho_y:T_yU\to H^1(X_y,T_{X_y})$ for the Kodaira--Spencer map. On the constant-rank locus constructed below, its null directions give holomorphically trivial families along local integral curves, and hence null directions for curvature.

\begin{lemma}[K-S nullity implies determinant vanishing]
\label{lem:ks-nullity-fixed-metric}
There is a nonempty Zariski-open subset $U^{\mathrm{KS}}\subseteq U$ on which $R^1f_*T_{X/U}$ is locally free with fiberwise base-change isomorphisms and the Kodaira--Spencer sheaf map has constant rank. For every $y\in U^{\mathrm{KS}}$ and $\xi\in\ker\rho_y$, there is a holomorphic curve $\gamma:(D,0)\to(U,y)$ with $\gamma'(0)=\xi$ whose pulled-back family is holomorphically trivial. For every level $j\geq2$ with positive-rank direct image $E_j$, the fixed-point metric $G_j$ is constant in the transported holomorphic frame on $D$. In particular, on $U^{\mathrm{KS}}\cap U_j^{\mathrm{reg}}$,
\[
 \ker\rho_y\subseteq\ker\tau_{j,y}.
\]
\end{lemma}
\begin{proof}
Generic cohomology and base change, followed by restriction to the constant-rank locus of the Kodaira--Spencer map, gives the subset $U^{\mathrm{KS}}$. Since the fibers are connected, $f_*\mathcal O_X=\mathcal O_U$. The direct-image sequence of
\[
 0\longrightarrow T_{X/U}\longrightarrow T_X
 \longrightarrow f^*T_U\longrightarrow0
\]
is therefore exact at the middle term in
\[
 f_*T_X\longrightarrow T_U
 \xrightarrow{\kappa}R^1f_*T_{X/U}.
\]
On $U^{\mathrm{KS}}$, the kernel of $\kappa$ is a holomorphic subbundle with fiber $\ker\rho_y$. Extend $\xi$ to a local section $v$ of this kernel. Sheaf exactness gives, after shrinking an analytic neighborhood $W$ of $y$, a holomorphic vector field $\widetilde v$ on the entire $f^{-1}W$ with $df(\widetilde v)=f^*v$.

Choose nested relatively compact base neighborhoods whose closures lie in $W$. Properness makes their closed inverse images compact. Holomorphic ODE existence on a neighborhood of these compact sets gives the flows $\Phi_t$ of $\widetilde v$ and $\phi_t$ of $v$ for a common small complex time, including negative time. Uniqueness gives $f\circ\Phi_t=\phi_t\circ f$. Shrinking the time disk keeps the relevant fibers in the smaller neighborhoods, so
\[
 \Phi_t:X_y\xrightarrow{\sim}X_{\gamma(t)},
 \qquad \gamma(t)=\phi_t(y),
\]
has inverse $\Phi_{-t}$ on the whole target fiber. Joint holomorphic dependence identifies the pulled-back family with $X_y\times D$. The zero vector may also be treated by the constant curve.

Fix $j\geq2$ with $E_j=f_*\mathcal O_X(jK_{X/U})\ne0$. In this unweighted smooth projective family, invariance of plurigenera and Grauert's theorem give local freeness and fiberwise base-change isomorphisms for $E_j$, by \eqref{eq:complete-base-change} with $m=j$. Transport a basis of $H^0(X_y,jK_{X_y})$ by the biholomorphisms above. It is therefore a holomorphic frame of $\gamma^*E_j$. Change of variables preserves every intrinsic integral defining the Hilbert--Fubini--Study operator, as in Proposition~\ref{prop:birational-invariance}. Thus the operator is constant in this frame, and Theorem~\ref{thm:fixed-point-construction} makes its unique fixed matrix constant. Wherever $G_j$ is smooth near $y$, naturality of Chern curvature gives
\[
 \gamma^*\Theta_{G_j}=0,\qquad
 (\Theta_{G_j})_y(\xi,\bar\xi)=0.
\]
On $U_j^{\mathrm{reg}}$, $\tau_j$ is semipositive, so $\tau_j(\xi,\xi)=0$ places $\xi$ in its kernel.
\end{proof}

\begin{proposition}
\label{prop:generated-nullity}
Suppose that $mK_{X_y}$ is globally generated for every $y\in U$ and that the associated morphisms $\varphi_m$ are generically finite. Then
\[
 Q_\xi=0\quad\Longrightarrow\quad\rho_y(\xi)=0,
\]
where $\rho_y:T_yU\to H^1(X_y,T_{X_y})$ is the Kodaira--Spencer map.
\end{proposition}
\begin{proof}
Fix $y$ and restrict to a disk tangent to $\xi$, with parameter $t$. Write $F=X_y$ and $\Phi=[f_1:\cdots:f_r]$ in the frozen holomorphic frame. The adjoint twist $L=(m-1)K_{X/U}$ has smooth metric weight $\psi=\alpha\log\sum_j|f_j|^2$ and semipositive curvature
\begin{equation}\label{eq:generated-frozen-form}
 \Omega=i\partial\bar\partial\psi=\alpha\Phi^*\omega_{\mathrm{FS}}.
\end{equation}
Here $\omega_{\mathrm{FS}}=i\partial\bar\partial\log|z|^2$ in homogeneous coordinates. Smoothness follows from global generation of $mK_{X_y}$. The vertical form may degenerate along the ramification locus.

We use the normalized-representative form of the smooth direct-image curvature identity \cite[Proposition~4.2, (4.8) and (5.1)]{Berndtsson2009}. For each $u\in H^0(F,K_F+L_F)$, choose an $H_y$-Chern-normal holomorphic extension and a smooth $L$-valued $(n,0)$ representative $\widehat u$ on the total space. Along $F$ it satisfies
\[
 \bar\partial\widehat u=\eta_u\wedge dt,\qquad
 \eta_u\wedge\omega_F=0,\qquad \partial^\psi\widehat u=0.
\]
The formula expresses $\langle Q_\xi u,u\rangle$ as the sum of $\|\eta_u\|^2$ and the fiber integral of the nonnegative curvature density $c_n\widehat u\wedge\overline{\widehat u}\wedge \Omega e^{-\psi}$, interpreted as the coefficient of $i\,dt\wedge d\bar t$. Here $c_n=i^{n^2}/2^n$ is our top-form normalization, also used in the primitive-form norm. If $Q_\xi=0$, both terms vanish for every $u$. Semipositivity of $\Omega$ then gives the pointwise identities
\begin{equation}\label{eq:generated-equality}
 \eta_u=0,\qquad \widehat u\wedge\Omega=0
 \quad\text{along }F.
\end{equation}
The second equality in \eqref{eq:generated-equality} is an equality of ambient differential forms at the points of $F$, not merely an equality after exterior pullback to $F$.

In submersion coordinates, with $dz=dz^1\wedge\cdots\wedge dz^n$ and a holomorphic frame $e_L$, write
\[
 \widehat u=(g\,dz+dt\wedge\beta)\otimes e_L,
 \qquad \beta=\sum_j b_j\,\iota_{\partial_{z^j}}dz.
\]
The first identity in \eqref{eq:generated-equality} makes $\beta|_F$ holomorphic. On $W_u=\{u\ne0\}\subset F$ set
\begin{equation}\label{eq:generated-local-lift}
 h_u=\partial_t-\sum_j\frac{b_j}{g}\partial_{z^j},
 \qquad \theta^j=dz^j+\frac{b_j}{g}dt.
\end{equation}
Then $\widehat u=g\theta^1\wedge\cdots\wedge\theta^n\otimes e_L$. The vector $h_u$ is intrinsically the unique vector annihilated by this decomposable $n$-form with $df(h_u)=\partial_t$. Thus the coordinate expressions in \eqref{eq:generated-local-lift} agree and define a holomorphic lift along $W_u$. In the coframe $(dt,\theta^1,\ldots,\theta^n)$, the second identity in \eqref{eq:generated-equality} kills the horizontal row of $\Omega$. Hermitian symmetry kills its column. Hence
\[
 0=\Omega(h_u,\overline{h_u})
   =\alpha|d\Phi(h_u)|_{\mathrm{FS}}^2,
 \qquad d\Phi(h_u)=0.
\]

For a basis $u_1,\ldots,u_r$, the sets $W_{u_j}$ cover $F$ by global generation. On an overlap, $h_{u_i}-h_{u_j}$ is vertical, holomorphic, and killed by $d_F\Phi$. Generic finiteness in characteristic zero makes $d_F\Phi$ injective on a dense Zariski-open subset of $F$. The difference therefore vanishes on a dense open subset of each overlap, and hence on the whole overlap. The lifts glue across ramification and zeros of individual sections to a holomorphic splitting of
\[
 0\longrightarrow T_F\longrightarrow T_{X_{\mathrm{disk}}}|_F
 \longrightarrow\mathcal O_F\xi\longrightarrow0.
\]
The extension class evaluated on $1$ is $\rho_y(\xi)$, which is zero.
\end{proof}

\begin{lemma}
\label{lem:variation-ks}
If the fibers are of general type and $\operatorname{Var}(f)=\dim U$, then $\rho_y$ is injective on a nonempty Zariski-open subset.
\end{lemma}
\begin{proof}
Work on $U^{\mathrm{KS}}$ from Lemma~\ref{lem:ks-nullity-fixed-metric}. If the kernel subbundle has positive rank, choose a nonzero vector in its fiber at a very general point $y$. The projectable flow constructed in that lemma gives a nonconstant analytic arc through $y$ of isomorphic fibers. The no-arc consequence of maximal birational variation for general-type families, recorded in \cite[Lemmas~3.11--3.12]{PopaTajiWu2019} from Koll\'ar's birational classifying theorem, excludes such an arc. Thus the generic kernel is zero, and the constant-rank locus has the required injectivity.
\end{proof}

\begin{theorem}[Curvature kernel, trivial variation and strict positivity, Case (G)]
\label{thm:strictness-generated}
Suppose the fibers are of general type and $mK_{X_y}$ is globally generated on the general fiber. On a nonempty Zariski-open subset contained in $U^{\mathrm{KS}}\cap U_m^{\mathrm{reg}}$ and the locus where the relative evaluation map is surjective,
\begin{equation}\label{eq:generated-exact-kernel}
 \ker\tau_m=\ker\rho,\qquad
 \ker K_m=(\ker\rho)\otimes E_m,\qquad
 \operatorname{rank}K_m=r_m\operatorname{rank}\rho,
\end{equation}
where $r_m=\operatorname{rank}E_m$. At every point where the relative evaluation map is surjective and $\rho_y$ is injective, the metric $G$ is strictly Griffiths and strictly Nakano positive. In particular maximal birational variation gives such positivity on a nonempty Zariski-open subset.
\end{theorem}
\begin{proof}
The relative evaluation map is surjective over a dense Zariski-open subset: its cokernel has closed support, whose proper image misses the general point by the fiberwise base-change isomorphisms. There $mK_{X_y}$ is big and globally generated, so $\varphi_m$ is generically finite. Otherwise, write $mK_{X_y}=\Phi^*\mathcal O_Z(1)$ for the image $Z$ of $\Phi=\varphi_m$. The coherent sheaf $\Phi_*\mathcal O_{X_y}$ would then give $h^0(X_y,kmK_{X_y})=O(k^{\dim Z})$ with $\dim Z<n$, contradicting bigness.

Intersect with $U^{\mathrm{KS}}\cap U_m^{\mathrm{reg}}$. Corollary~\ref{cor:directional-nullity} and Proposition~\ref{prop:generated-nullity} give
\[
 \xi\in\ker\tau_{m,y}\ \Longrightarrow\ Q_{m,\xi}=0
 \ \Longrightarrow\ \rho_y(\xi)=0.
\]
Lemma~\ref{lem:ks-nullity-fixed-metric} gives the reverse inclusion. The kernel and rank of $K_m$ in \eqref{eq:generated-exact-kernel} now follow from Theorem~\ref{thm:determinant-pinching}. This proves the kernel identities without identifying the rank of $\rho$ with birational variation.

For pointwise strictness, Theorem~\ref{thm:controlled-regularity} makes $G$ smooth wherever the relative evaluation map is surjective. The curvature calculation and nondegeneracy of the evaluation quadratic form in Section~\ref{sec:curvature} apply directly there, without principalization. At any injective-KS point there, Proposition~\ref{prop:generated-nullity} gives $Q_\xi\ne0$ for every $\xi\ne0$. Corollary~\ref{cor:directional-nullity} gives $K_\xi>0$, hence $\tau>0$, and Theorem~\ref{thm:determinant-pinching} gives strict Nakano positivity. This argument does not require the constant-rank restriction. Lemma~\ref{lem:variation-ks} proves the final assertion.
\end{proof}

\subsection{Base ideals and infinitesimal deformations of relative pluricanonical images}
Assume now that $\varphi_m$ is birational onto its image on a general fiber. We relate frozen curvature to the infinitesimal deformation of the relative pluricanonical image. Starting inside $U_m^{\mathrm{reg}}$, principalization resolves $\varphi_m$ to a morphism $\Psi:\widetilde X_U\to\mathbb P_q(E)$. After a further shrinking, its scheme image is a flat projective family
\begin{equation}\label{eq:flat-image-family}
 \mathcal Z\subset\mathbb P_q(E|_U)
\end{equation}
with geometrically integral, hence reduced, fibers and fixed Hilbert polynomial. To obtain this family, first form the scheme-theoretic image, before taking its reduction. The scheme-image/base-change stratification \cite[Tag 0H41, Lemma~37.54.5]{StacksProject} permits formation of this image to commute with geometric fibers on a dense Zariski-open subset. Each geometric resolved fiber is integral. On an affine chart its image ideal is the kernel of the map to the ring of functions on a nonempty open subset of that integral fiber, and is prime. Generic flatness then gives \eqref{eq:flat-image-family}. Birationality also spreads: the generic birational map has an inverse on a dense open subset, which meets every integral image fiber densely after shrinking $U$. In an algebraic frame the image family gives a morphism
\[
 a:U\longrightarrow\operatorname{Hilb}_P(\mathbb P^{r-1}).
\]

\begin{lemma}
\label{lem:geodesic-image-tangent}
At $y\in U$, frozen nullity $Q_\xi=0$ makes the Hilbert tangent of \eqref{eq:flat-image-family} zero in a $G$-Chern-normal holomorphic frame. In any other holomorphic frame it is an infinitesimal $\operatorname{PGL}_r$ orbit tangent.
\end{lemma}
\begin{proof}
Restrict to a disk tangent to $\xi$. The frozen twist
\[
 L=(m-1)K_{X/U},\qquad \psi=\alpha\log\sum_j|f_j|^2
\]
has analytic singularities and semipositive curvature. For a local section $s$ of $E_m$, Cauchy--Schwarz gives
\begin{equation}\label{eq:geodesic-no-loss}
 |f_s|^2e^{-\psi}\leq C_s\Bigl(\sum_j|f_j|^2\Bigr)^{1/m}.
\end{equation}
Thus the adjoint multiplier direct image and every fiber multiplier space identify with $E_m$ and the full space $H^0(X_y,mK_{X_y})$, respectively, as proved with the Jacobian and multiplier transformation in Proposition~\ref{prop:frozen-nakano}. Its pairing is precisely the smooth metric $H_y$.

Off the base locus, $\Omega=i\partial\bar\partial\psi =\alpha\Phi^*\omega_{\mathrm{FS}}$. On the dense locus where the vertical differential of $\Phi$ has rank $n$, its vertical restriction is positive definite. The absolutely continuous vertical top mass is positive and finite on every fiber over $U$: on the principalization the residual form is $\alpha\Psi^*\omega_{\mathrm{FS}}$. The smooth projective disk family, analytically singular semipositive twist and positive vertical mass satisfy Theorem~B of \cite{CaoGuenanciaPaun2021}. Its Remark~5.5 applies at our selected point, since $H_y$ is smooth nearby by Theorem~\ref{thm:controlled-regularity} and its fiber is the full multiplier section space $H^0(X_y,mK_{X_y})$. Their directional lower bound is
\begin{equation}\label{eq:geodesic-lower-bound}
 \langle Q_\xi s,s\rangle
 \geq\int_{X_y}c_\xi(\Omega)|s|_{e^{-\psi}}^2\geq0,
\end{equation}
Here $|s|_{e^{-\psi}}^2$ denotes the adjoint top-degree density, and $c_\xi(\Omega)$ is the geodesic-curvature coefficient in the sense of \cite[Definition~2.3]{CaoGuenanciaPaun2021}, including the case of degenerate vertical curvature. On the locus where the vertical restriction of $\Omega$ is positive definite, it is given by the Schur-complement expression below. Nakano semipositivity of the frozen curvature tensor $Q$ was established separately in Section~\ref{sec:curvature}.

At a nonbase point $x$ where $d_{X_y}\Phi$ is injective, take any lift $w$ of $\xi$. The Schur complement of the vertical block gives
\begin{align}\label{eq:geodesic-normal-motion}
 c_\xi(\Omega)(x)
 &=\min_{v\in T_xX_y}\Omega(w+v,\overline{w+v})\notag\\
 &=\alpha\left|\operatorname{pr}_{d_{X_y}\Phi(T_xX_y)^\perp}
                       d\Phi(w)\right|_{\mathrm{FS}}^2.
\end{align}
If $Q_\xi=0$, sum \eqref{eq:geodesic-lower-bound} over a basis. The summed density is strictly positive off the base locus. The function $c_\xi$ is nonnegative and smooth on the full-rank locus, so it vanishes there. At every such $x$ there is consequently a tangent vector $h_x$ with
\begin{equation}\label{eq:pointwise-image-lift}
 df(h_x)=\xi,\qquad d\Phi(h_x)=0.
\end{equation}
These are pointwise vectors, which need not extend across a base divisor.

Let $Z=\mathcal Z_y$. The first-order flat image family in direction $\xi$ has Hilbert tangent $\delta_\xi:\mathcal I_Z\to\mathcal O_Z$. For a local equation $r$ of $Z$, lift it to $r+t\dot r$ modulo $t^2$ and take $\delta_\xi(r)=\dot r\bmod\mathcal I_Z$. Flatness makes this well-defined: if two lifts differ by $t h$, the class of $h$ in the flat quotient is killed by $t$. The kernel of multiplication by $t$ is its image, so the class of $h$ belongs to $t$ times that quotient. Thus $h$ has zero reduction in $\mathcal O_Z$, as required. Pull back the lifted equation to the first-order source and differentiate along \eqref{eq:pointwise-image-lift}. This yields
\[
 \dot r(\Phi(x))+dr_{\Phi(x)}(d\Phi(h_x))=0.
\]
Thus $\delta_\xi(r)$ vanishes on a dense subset of $Z$. Reducedness makes it zero on all of $Z$, for every local equation $r$. The Hilbert tangent is zero in the normal frame. Differentiating a holomorphic frame change adds the orbit tangent of a single infinitesimal projective linear transformation, proving the final assertion.
\end{proof}

\subsection{The birational classifying map and variation}
Let $S$ be the reduced closure of the $\operatorname{PGL}_r$-saturation of $a(U)$ in the Hilbert scheme. It is integral, as the closure of the image of the integral variety $\operatorname{PGL}_r\times U$. Rosenlicht's theorem gives a dense invariant Zariski-open subset $S^0$ and a geometric quotient $q:S^0\to B$ with geometric orbit fibers, also after field extension \cite[Theorem~1.1]{BellGhiocaReichstein2014}. The inverse image $a^{-1}(S^0)$ is nonempty: otherwise the invariant closed complement would contain $a(U)$, its saturation, and its closure $S$. Restrict to that inverse image and replace $B$ by the closure of the image of
\begin{equation}\label{eq:birational-classifier}
 b=q\circ a:U\longrightarrow B.
\end{equation}
Shrink again so that the image is smooth and $b$ is a submersion onto it. Such a shrinking is available by generic smoothness in characteristic zero. Denote the resulting dense Zariski-open subset by
\begin{equation}\label{eq:variation-open}
 U_m^{\mathrm{var}}:=U\subseteq U_m^{\mathrm{reg}}.
\end{equation}
The fixed set $U_m^{\mathrm{reg}}$ is unchanged by these classifying constructions. In particular, hypotheses imposed on its complement refer to that original complement.

\begin{lemma}
\label{lem:quotient-birational-variation}
Two fibers over $U_m^{\mathrm{var}}$ have the same $b$-value exactly when they are birational. Moreover $\dim B=\operatorname{Var}(f)$, and $Q_\xi=0$ implies $db(\xi)=0$.
\end{lemma}
\begin{proof}
Projectively equivalent images are birational to the corresponding source fibers, hence give birational source fibers. Conversely, a birational map between smooth projective fibers identifies their complete pluricanonical spaces as rational forms. Their level-$m$ images therefore differ by a projective linear change of basis. This proves the assertion about quotient fibers, including independence of a nonminimal model.

For the dimension statement put $K=\mathbb C(U)$ and $L=\mathbb C(B)$. A closed point of the generic quotient fiber over $L$ is defined over a finite extension $L'/L$. Embed $L'$ into an algebraic closure of $K$. Its Hilbert point and $a(\eta_U)$ lie in one geometric quotient fiber, hence in one projective orbit. The geometric generic image, and therefore the geometric generic source fiber, has a birational model over $L'$. This gives $\operatorname{Var}(f)\leq\dim B$.

In the reverse direction, let $k\subset\overline K$ be a field of definition of a birational model with transcendence degree $\operatorname{Var}(f)$, allowing a finite extension. Resolve the model over $k$. Its complete pluricanonical space and image are defined over $k$ and commute with extension to $\overline K$. Birational invariance identifies that image projectively with the original one. Its Hilbert point lies in $S^0$ after extension, by invariance of $S^0$, and has the same $q$-value. This generic $q$-value is defined over $k$, so $\mathbb C(B)$ embeds in $k$. Consequently $\dim B\leq\operatorname{Var}(f)$. Finally Lemma~\ref{lem:geodesic-image-tangent} gives a zero Hilbert tangent modulo a projective orbit tangent, and $dq$ kills that orbit tangent. Hence $db(\xi)=0$.
\end{proof}

\begin{theorem}[Curvature kernel, trivial variation and strict positivity, Case (B)]\label{thm:strictness-birational}
On $U_m^{\mathrm{var}}$, the curvature of $G$ satisfies
\begin{equation}\label{eq:birational-exact-kernel}
 \ker\tau=\ker db,\qquad
 \ker K=(\ker db)\otimes E,
 \qquad\operatorname{rank}K=r\operatorname{Var}(f).
\end{equation}
The determinant curvature descends locally to a K\"ahler form on the variation base. If $\operatorname{Var}(f)=\dim U$, then $G$ is strictly Griffiths and strictly Nakano positive on a nonempty Zariski-open subset. No assumption on the base ideal, beyond its being the base ideal of $|mK_{X_y}|$, is imposed.
\end{theorem}
\begin{proof}
By Corollary~\ref{cor:directional-nullity}, a vector in $\ker\tau$ has $Q_\xi=0$. Lemma~\ref{lem:quotient-birational-variation} gives $\ker\tau\subset\ker db$.

For the reverse inclusion choose $\xi\in\ker db_y$ and a holomorphic disk $D$ in the smooth local fiber of $b$, with tangent $\xi$. Its Hilbert points lie in a single projective orbit. The orbit map from $\operatorname{PGL}_r$ is a holomorphic submersion and has local holomorphic sections. After shrinking the disk, lift such a section to $\operatorname{GL}_r$ and change the frame. The image schemes are now one fixed reduced integral $Z_0\subset\mathbb P^{r-1}$.

It remains to normalize the scalar left by projective constancy. Use the central fiber as a smooth projective model $F_0$ of $Z_0$, and let $s_1,\ldots,s_r$ be the corresponding basis of $H^0(F_0,mK_{F_0})$. The maps give birational identifications of all nearby fibers with $F_0$. Pull the adjusted frame $e_j(t)$ back as rational $m$-canonical forms to $F_0$. Their span is exactly $V_0=H^0(F_0,mK_{F_0})$ by birational invariance. Equality of the projective coordinates gives, for each fixed $t$,
\begin{equation}\label{eq:projective-frame-scalar}
 e_j(t)=c_t(z)s_j\quad\text{for every }j,
 \qquad 0\ne c_t\in\mathbb C(F_0).
\end{equation}
Multiplication by $c_t$ maps $V_0$ onto itself. It has a nonzero eigenvector $v$, with eigenvalue $\lambda\ne0$. The equality of rational forms $c_tv=\lambda v$ forces $c_t=\lambda$ in the function field. Thus the scalar in \eqref{eq:projective-frame-scalar} is independent of the fiber point.

To prove holomorphic dependence on $t$, choose a smooth point $z_0\in Z_0$ and a point of $F_0$ above it where the map is defined with invertible vertical differential, and a basis section nonzero there. The map $(\Phi,f)$ is a local biholomorphism onto a neighborhood of $(z_0,0)$ in $Z_{0,\mathrm{reg}}\times D$. Its inverse is the birational inverse on each nearby fiber. In these coordinates the pulled-back relative pluricanonical forms are holomorphic in $(z,t)$. Evaluating the ratio of a nonzero basis section to $s_j$ at $z_0$ shows that $c_t=c(t)$ is holomorphic and nonvanishing after shrinking $D$. Dividing every $e_j(t)$ by $c(t)$ gives a holomorphic linear frame whose elements are the same rational pluricanonical forms on $F_0$. Proposition~\ref{prop:birational-invariance} now identifies all fiber Hilb--FS operators in this frame with $T_0$. Uniqueness of the fixed point makes $G$ constant. Its curvature on the disk vanishes, so $K_\xi=0$ and $\xi\in\ker\tau$.

This proves the first identity in \eqref{eq:birational-exact-kernel}. The remaining two follow from Theorem~\ref{thm:determinant-pinching} and $\dim B=\operatorname{Var}(f)$. The real form $i\operatorname{Tr}\Theta^G$ is closed and annihilates every vertical tangent of the local submersion $b$. In coordinates where $b$ is a projection, closedness makes its coefficients independent of the vertical coordinates, or equivalently Cartan's formula makes it invariant under vertical flows. It is therefore the pullback of a local closed $(1,1)$ form on $B$. The exact kernel and semipositivity make this form positive definite, hence K\"ahler. When variation is maximal, $db$ is injective on $U_m^{\mathrm{var}}$. Thus $\tau>0$ and Theorem~\ref{thm:determinant-pinching} gives the asserted strictness.
\end{proof}

\begin{corollary}[Curvature rank and variation]
\label{cor:unified-variation-rank}
Suppose the general fiber $F$ is of general type and at least one of the following holds:
\begin{enumerate}
\item[(G)] $mK_F$ is globally generated;
\item[(B)] the $m$-canonical map $\varphi_m$ is birational onto its image.
\end{enumerate}
Put $v=\operatorname{Var}(f)$ and $r_m=\operatorname{rank}E_m$. There is a nonempty Zariski-open subset
\begin{equation}\label{eq:common-variation-open}
 U_m^\circ\subseteq U_m^{\mathrm{reg}}
\end{equation}
on which
\begin{equation}\label{eq:common-variation-rank}
 \operatorname{rank}\tau_m=v,\qquad
 \ker K_m=(\ker\tau_m)\otimes E_m,\qquad
 \operatorname{rank}K_m=r_m v.
\end{equation}
In case (G), $\ker\tau_m=\ker\rho$. In case (B), one may take $U_m^\circ\subseteq U_m^{\mathrm{var}}$ and $\ker\tau_m=\ker db_m$, where $b_m$ is the level-$m$ classifier \eqref{eq:birational-classifier}. In either case the determinant form $\eta_m=i\operatorname{Tr}\Theta_{G_m}$ descends locally to a K\"ahler form on a submersion base of dimension $v$.
\end{corollary}
\begin{proof}
Case (B) is Theorem~\ref{thm:strictness-birational}. In case (G), choose $N=km$ sufficiently large that the $N$-canonical map $\varphi_N$ is birational onto its image. Such a divisible level exists directly by \cite[Theorem~1.1]{HaconMcKernan2006}, which gives birationality at every sufficiently large integer level. Also $NK_F$ is globally generated: at any $x\in F$, a section $s\in H^0(F,mK_F)$ with $s(x)\ne0$ has $s^k(x)\ne0$. Thus these powers certify global generation of $NK_F$, without requiring surjectivity of $\operatorname{Sym}^kH^0(F,mK_F)\to H^0(F,NK_F)$.

Apply the kernel identities of Theorem~\ref{thm:strictness-generated} separately at levels $m$ and $N$, and apply Theorem~\ref{thm:strictness-birational} at level $N$. Intersect the level-$m$ and level-$N$ regularity loci, the loci where their relative evaluation maps are surjective, $U^{\mathrm{KS}}$, and the level-$N$ classifying stratum $U_N^{\mathrm{var}}$. These finitely many dense Zariski-open subsets have a nonempty intersection, which we take for $U_m^\circ$. On it,
\begin{equation}\label{eq:auxiliary-level-kernels}
 \ker\tau_m=\ker\rho=\ker\tau_N=\ker db_N.
\end{equation}
The classifier $b_N$ has rank $v$ by Lemma~\ref{lem:quotient-birational-variation}. The equality of kernels on the same base tangent space gives $\operatorname{rank}\tau_m=\operatorname{rank}\rho=v$, and determinant pinching gives the remaining assertions in \eqref{eq:common-variation-rank}. If both (G) and (B) hold, intersect also with $U_m^{\mathrm{var}}$ to obtain both kernel interpretations. This uses only the kernel part of the globally generated case and the independently proved birational theorem.

For local descent use $b_N$ in case (G) and $b_m$ in case (B). The closed semipositive form $\eta_m$ has exactly the vertical kernel of the chosen submersion. The coordinate argument in the proof of Theorem~\ref{thm:strictness-birational} therefore makes it the pullback of a local closed positive definite $(1,1)$ form. This is the asserted local K\"ahler form.
\end{proof}

For families of good minimal models, the generic equality between Kodaira--Spencer rank and birational variation is already contained in the classifying construction recalled by Taji \cite[Set-up~3.15(3.15.4) and Theorem--Definition~3.16]{Taji2023GoodMinimalModels}. In case (G), surjectivity of the relative evaluation map makes $K_{X/U}$ relatively semiample, so this locus falls within that setting. Our equality argument identifies this generic Kodaira--Spencer kernel with the determinant-curvature kernel of the metric $G_m$. Case (B) permits arbitrary base ideals of $|mK_F|$ and retains the classifying kernel $\ker db_m$, including for nonminimal smooth models. The common rank in Corollary~\ref{cor:unified-variation-rank} is the input for the global positivity and vanishing applications below.

\section{Green--Griffiths bigness conjecture}
\label{sec:bigness}

Let $f:X\to Y$ be a surjective projective morphism of smooth connected complex projective varieties, with connected fibers. Put $d=\dim Y>0$, $n=\dim X-d>0$, and
\[
 E_m=f_*\omega_{X/Y}^{\otimes m},\qquad m\geq2,
\]
and suppose that $E_m$ has positive rank $r$. The sheaf $E_m$ is torsion free: any section annihilated by a nonzero function on $Y$ vanishes on a dense Zariski-open subset of the integral variety $X$, hence everywhere. The construction of Sections~\ref{sec:construction}--\ref{sec:curvature} defines its fixed-point metric $G$ on the smooth base-change locus and proves smooth Nakano semipositivity on $U_m^{\mathrm{reg}}$. Choose a smooth projective modification
\begin{equation}\label{eq:actual-modification}
 \pi:\widetilde Y\longrightarrow Y,\qquad
 V_m=\pi^*E_m/\mathrm{tors},
\end{equation}
where $V_m$ is locally free and $\pi$ is an isomorphism over a dense Zariski-open subset of $U_m^{\mathrm{reg}}$. To construct it, choose a surjection $(A^{-N})^{\oplus s}\twoheadrightarrow E_m$ with $A$ ample. Over the locally free locus, the rank-$r$ quotients define a section of the relative Grassmannian of $(A^{-N})^{\oplus s}$. Take the reduced closure of this section and resolve it. The universal quotient is locally free and receives the pullback of $E_m$: its defining relations vanish on the dense section and hence everywhere. The resulting surjection is generically an isomorphism, so its kernel is exactly the torsion. This gives \eqref{eq:actual-modification} and the asserted isomorphism locus. Passing to $V_m$ retains the vanishing conditions in $E_m$ that could be lost upon taking its reflexive hull.

\subsection{Projectivization and the principal component}
For a coherent sheaf $\mathcal F$ of positive generic rank on an integral variety $Y$, our convention is
\begin{equation}\label{eq:principal-projectivization}
 \mathbb P_q(\mathcal F)=\operatorname{Proj}_Y
                  \operatorname{Sym}^{\bullet}\mathcal F,
 \qquad
 \mathbb P_{\mathrm{pr}}(\mathcal F)
 =\overline{\mathbb P_q(\mathcal F|_U)}_{\mathrm{red}}
 \subset\mathbb P_q(\mathcal F),
\end{equation}
where $U$ is any dense Zariski-open subset on which $\mathcal F$ is locally free. Thus points of the bundle locus represent one-dimensional quotients, and there is a universal quotient $p^*\mathcal F\twoheadrightarrow \mathcal O(1)$. The reduced closure in \eqref{eq:principal-projectivization} is integral and independent of $U$. By \emph{tautological bigness of $\mathcal F$} we mean bigness of the restriction of this $\mathcal O(1)$ to $\mathbb P_{\mathrm{pr}}(\mathcal F)$. Other components, supported over proper subsets of $Y$, do not enter this definition. All symmetric powers below are ordinary coherent symmetric powers.

\subsection{Extension of the fixed-point metric}
A singular Griffiths-semipositive metric on a torsion-free sheaf will mean the following: the metric is positive definite and finite almost everywhere on the locally free locus, and the logarithmic dual norm of each local holomorphic dual section is plurisubharmonic, including its extension across the nonfree set. Hacon--Popa--Schnell's \textit{minimal extension property} \cite[Definition~20.1]{HaconPopaSchnell2018} requires a fixed nowhere dense closed analytic set $Z$ containing the nonfree locus such that, for every holomorphic embedding $\iota:B_1\hookrightarrow Y$ with $\iota(0)\notin Z$ and every unit vector $v$ at its center, there is a section $s\in H^0(B_1,\iota^*\mathcal E)$ with $s(0)=v$ and
\[
 \frac{1}{\operatorname{vol}(B_1)}
 \int_{B_1}|s|_{\iota^*h}^2\,d\lambda\leq1.
\]
Here $B_1\subset\mathbb C^d$ is the unit ball.

For comparison arguments, we use the following \emph{locally uniform bounded extension property}. Fix a proper closed analytic set $Z\subset Y$ containing the nonfree locus, and write $Y^\circ=Y\setminus Z$. For each $p\in Y$ and every sufficiently small coordinate ball $W=B(p,4R)$, there is a constant $C_W<\infty$ such that every center $y\in B(p,R/2)\cap Y^\circ$ at which $h_y$ is finite and positive definite, and every $v\in\mathcal E_y$, admit a section $s\in H^0(B(y,3R),\mathcal E)$ with
\begin{equation}\label{eq:bounded-actual-extension}
 s(y)=v,\qquad
 \int_{B(y,3R)}|s|_h^2\,d\lambda
 \leq C_W\operatorname{vol}(B_{3R})|v|_{h,y}^2.
\end{equation}
Such centers have full measure in $Y^\circ$. The constant may depend on the coordinate neighborhood and radius, but is independent of $y$ and $v$, including as $y$ approaches $Z$. The section is required to lie in $\mathcal E$ on the whole ball. Translation and dilation give this property from minimal extension with $C_W=1$. Bounded metric comparison will give $C_W=r^{m-1}$ for $G$. Pointwise metric inequalities below use the representatives determined by the dual psh norms.

\begin{lemma}
\label{lem:extension-flattening}
Let $\mathcal E$ be a torsion-free coherent sheaf on a smooth complex manifold $Y$, with a singular Griffiths-semipositive metric $h$ satisfying the locally uniform bounded extension property \eqref{eq:bounded-actual-extension}. Let $\pi:\widetilde Y\to Y$ be a proper modification, isomorphic over $Y^\circ$ after enlarging $Z$ if necessary, such that
\[
 V=\pi^*\mathcal E/\mathrm{tors}
\]
is locally free and $\widetilde Y$ is smooth. For any smooth reference metric $h_0$ on $V$ and any compact subset $K\subset\widetilde Y$, there is a constant $c>0$ such that $\pi^*h\geq c h_0$ on $K\cap\pi^{-1}(Y^\circ)$. Consequently $\pi^*h$ extends as a singular Griffiths-semipositive metric on the bundle $V$. In particular the conclusion holds under the minimal extension property.
\end{lemma}
\begin{proof}
Shrink the coordinate ball in \eqref{eq:bounded-actual-extension} so that there is a surjection $\mathcal O_W^N\twoheadrightarrow\mathcal E^*$, and consider the nested balls
\[
 W=B(p,4R),\qquad B=B(p,2R),\qquad C=B(p,R/2).
\]
Let $\lambda_1,\ldots,\lambda_N$ be the images of the standard basis under this surjection. Torsion freeness makes the evaluation map
\[
 j:\mathcal E\longrightarrow\mathcal O_W^N,
 \qquad j(s)=(\lambda_1(s),\ldots,\lambda_N(s))
\]
injective. Its kernel vanishes on the locally free locus and would therefore be torsion. The psh functions $\log|\lambda_a|_{h^*}^2$ are bounded above on $\overline B$, so a constant $C_0$ gives
\begin{equation}\label{eq:dual-generator-control}
 |j(s)|^2\leq C_0|s|_h^2\quad\text{almost everywhere on }B.
\end{equation}

The space $H^0(B,\mathcal E)$ is a closed subspace of $H^0(B,\mathcal O^N)$ through $j$. Indeed, suppose $j(s_\nu)$ converges locally uniformly to $t$. At each point $p'$ and each integer $k$, convergence of derivatives shows that the $k$-jet of $t$ lies in the image of $\mathcal E_{p'}$ in the finite-dimensional jet space. That image is a closed linear subspace. The class of $t_{p'}$ in the module of finite type $\mathcal O_{p'}^N/j(\mathcal E_{p'})$ belongs to every power of its maximal ideal, and is zero by Krull intersection. Hence $t$ is a section of $j(\mathcal E)$. This gives $H^0(B,\mathcal E)$ its Fr\'echet topology of locally uniform convergence under $j$.

Consider the linear map
\[
 \Phi:H^0(B,\mathcal E)\longrightarrow H^0(\pi^{-1}B,V),
 \qquad s\longmapsto\pi^*s\bmod\mathrm{tors}.
\]
Its graph is closed. If $s_\nu\to s$ under $j$ and $\Phi(s_\nu)\to t$ locally uniformly upstairs, then over $\pi^{-1}(Y^\circ)\cap\pi^{-1}(B)$ the locally free inclusion $j$ identifies the two limits, so $t=\Phi(s)$ there. Holomorphy gives equality everywhere. Since both spaces are Fr\'echet, the closed graph theorem makes $\Phi$ continuous. Equivalently, this is the coherent-section continuity used in \cite[Propositions~15.2--15.3]{HaconPopaSchnell2018}. Properness makes $\pi^{-1}(\overline C)$ compact. Continuity, followed by the holomorphic submean estimate for the components of $j(s)$, therefore gives a fixed compact $K\Subset B$ and fixed constants with
\begin{equation}\label{eq:uniform-coherent-section-estimate}
 \sup_{\pi^{-1}(\overline C)}|\Phi(s)|_{h_0}^2
 \leq C_1\sup_K|j(s)|^2
 \leq C_2\int_B|s|_h^2\,d\lambda.
\end{equation}
For the second inequality, cover $K$ by finitely many smaller balls with closures in $B$ and combine the submean estimate with \eqref{eq:dual-generator-control}. The neighborhoods and constants in \eqref{eq:uniform-coherent-section-estimate} are independent of the evaluation point and vector.

Now choose $y\in C\cap Y^\circ$ at which $h_y$ is finite and positive definite, and $v\in\mathcal E_y$. The ball $B(y,3R)$ contains $B$ and is contained in $W$. Applying \eqref{eq:bounded-actual-extension} there and restricting to $B$ gives a section $s\in H^0(B,\mathcal E)$ with $s(y)=v$ and
\[
 \int_B|s|_h^2\,d\lambda\leq C_3|v|_{h,y}^2,
\]
where $C_3=C_W\operatorname{vol}(B_{3R})$ is independent of $y,v$. At the unique point $z$ above $y$, $\Phi(s)(z)=v$, whence
\[
 |v|_{h_0,z}^2\leq C_2C_3|v|_{\pi^*h,z}^2.
\]
As $\pi$ is biholomorphic over $Y^\circ$, the metric lower bound holds almost everywhere on $\pi^{-1}(C\cap Y^\circ)$.

For any local holomorphic $\eta\in V^*$, its squared dual norm is psh on $\pi^{-1}(Y^\circ)$, being the exponential of its psh logarithmic norm. The lower bound just obtained gives $|\eta|_{(\pi^*h)^*}^2\leq c^{-1}|\eta|_{h_0^*}^2$ almost everywhere on these neighborhoods. The submean inequality and continuity of the right-hand side extend this inequality to every point, including zeros of $\eta$. Applying it to all constant-coefficient dual sections in a local frame gives the pointwise metric lower bound. Finitely many such open neighborhoods cover any compact subset and give a common positive constant.

The logarithmic dual norm consequently satisfies
\[
 \log|\eta|_{(\pi^*h)^*}^2
 \leq\log\bigl(c^{-1}|\eta|_{h_0^*}^2\bigr).
\]
The right side is locally bounded above. The removable-singularity theorem extends the left side uniquely across the analytic complement. The extended dual norms agree on overlaps by uniqueness and determine the same measurable Hermitian form almost everywhere, with Griffiths semipositivity. Finiteness and positive definiteness almost everywhere hold on $\pi^{-1}(Y^\circ)$. All sections used in the estimates lie in $\mathcal E$, so its vanishing conditions have been retained.
\end{proof}

\begin{proposition}
\label{prop:ns-actual-extension}
The Hermitian metric $G_{\mathrm{NS}}$ of \eqref{eq:ns-hermitian} extends to a singular Griffiths-semipositive metric on the ordinary $E_m$ with the minimal extension property. In particular, it satisfies \eqref{eq:bounded-actual-extension} with $C_W=1$.
\end{proposition}
\begin{proof}
Let $h_m$ be the relative $m$th Narasimhan--Simha metric on $\omega_{X/Y}$. On a smooth fiber its local squared norm is $q^{-2/m}$, with $q$ as in \eqref{eq:ns-gauge}. Put
\[
 L=\omega_{X/Y}^{\otimes(m-1)},\qquad h=h_m^{m-1},\qquad
 \mathcal F=f_*\bigl(\omega_{X/Y}\otimes L\otimes\mathcal I(h)\bigr).
\]
Theorem~27.1(a) of \cite{HaconPopaSchnell2018} supplies the semipositive metric $h_m$ across the singular fibers. Theorem~21.1 equips this adjoint multiplier direct image $\mathcal F$ with a singular semipositive $L^2$ metric with minimal extension. On each smooth fiber the norm is
\[
 \int_{X_y}|f_s|^2q^{2/m-2}\,dV_n
 =G_{\mathrm{NS},y}(s,s).
\]
The normalization of the canonical integral is the one fixed in Section~\ref{sec:construction}.

The inclusion $\mathcal F\hookrightarrow E_m$ is an isomorphism over the smooth locus by Theorem~27.1(b). At the fiber level this no-loss property also follows from $|f_s|^2q^{2/m-2}\leq\ell(s)^{2-2/m}|f_s|^{2/m}$. Corollary~21.2 of the same reference transfers the metric and minimal extension through this generically isomorphic inclusion: an extended section in $\mathcal F$ is a section of $E_m$, with the same almost-everywhere norm. The extension thus takes place on $E_m$ itself, without passing to a reflexive hull. Definition~20.1 gives the unit-ball average bound. Affine translation and dilation give \eqref{eq:bounded-actual-extension} with $C_W=1$ on the balls used above. The cited numbering is that of the 9 May 2017 author manuscript.
\end{proof}

\begin{theorem}[Extension of the fixed-point metric]
\label{thm:actual-sheaf-extension}
The fixed-point metric $G$ extends to a singular Griffiths-semipositive metric on the ordinary sheaf $E_m$ and satisfies \eqref{eq:bounded-actual-extension} with $C_W=r^{m-1}$. For $\pi$ and $V_m$ in \eqref{eq:actual-modification}, its pullback extends to a singular Griffiths-semipositive metric on $V_m$. Locally, for a smooth reference metric,
\begin{equation}\label{eq:fixed-flattened-lower}
 \pi^*G\geq r^{-(m-1)}\pi^*G_{\mathrm{NS}}\geq c h_0,
 \qquad c>0.
\end{equation}
No global-generation or variation hypothesis is needed for this extension.
\end{theorem}
\begin{proof}
Proposition~\ref{prop:ns-actual-extension} extends $G_{\mathrm{NS}}$ on $E_m$. For each local holomorphic section $\lambda$ of $E_m^*$, comparison \eqref{eq:ns-comparison} gives
\[
 |\lambda|_{G^*}^2\leq r^{m-1}|\lambda|_{G_{\mathrm{NS}}^*}^2.
\]
On $U_m^{\mathrm{reg}}$ the logarithm on the left is psh by Theorem~\ref{thm:curvature-resolvent}. The logarithm on the right is locally bounded above by the Griffiths extension of $G_{\mathrm{NS}}$. Removable singularities therefore extend the logarithmic dual norms across $Y\setminus U_m^{\mathrm{reg}}$. On the locally free locus these norms define the singular Griffiths-semipositive extension of $G$. They also extend plurisubharmonically across the nonfree locus for every local section of $E_m^*$, as required by the sheaf convention above.

The two inequalities in \eqref{eq:ns-comparison} give the bounded extension property for this metric. Choose $Y^\circ\subset U_m^{\mathrm{reg}}$ avoiding the analytic set excluded by the NS minimal extension property. On the ball $B(y,3R)$, the NS minimal extension of $v$ is a section of the original $E_m$ and satisfies
\[
 \int|s|_G^2\leq\int|s|_{G_{\mathrm{NS}}}^2
 \leq\operatorname{vol}(B_{3R})|v|_{G_{\mathrm{NS}}}^2
 \leq r^{m-1}\operatorname{vol}(B_{3R})|v|_G^2.
\]
The constant is independent of the center and vector. Lemma~\ref{lem:extension-flattening} now applies directly to $G$. Applying it also to $G_{\mathrm{NS}}$ and using the lower comparison gives \eqref{eq:fixed-flattened-lower}. The resulting extension constant for $G$ is bounded, rather than necessarily optimal.
\end{proof}

\begin{corollary}[Nakano semipositivity of the extended metric]
\label{cor:singular-nakano-extension}
Under the hypotheses of Theorem~\ref{thm:actual-sheaf-extension}, the extended metric on $V_m$ is Nakano semipositive in the approximation sense of \cite[Definition~2.4]{InayamaMatsumura2024} and in the global $L^2$-estimate sense of \cite[Definition~2.6(1)]{InayamaMatsumura2024}, with lower form $\theta=0$. The same conclusions hold for $G$ on the locally free locus of $E_m$, and hence on all of $Y$ if $E_m$ is locally free.
\end{corollary}
\begin{proof}
Choose a dense Zariski-open subset $W\subset U_m^{\mathrm{reg}}$ over which $\pi$ is an isomorphism. For $M=\widetilde Y$, put $\Sigma=\widetilde Y\setminus\pi^{-1}(W)$. For the locally free locus $M$ of $E_m$, put $\Sigma=M\setminus W$. In both cases $\Sigma$ is a proper closed analytic subset and the extended metric $h$ is smooth and Nakano semipositive on $M\setminus\Sigma$ by Theorem~\ref{thm:curvature-resolvent}, using the biholomorphism over $W$ on $\widetilde Y$. It is measurable, finite and positive definite almost everywhere. Choose its representative given by the extended dual psh norms. Then $|\lambda|_{h^*}^2$ is upper semicontinuous for every local holomorphic dual section $\lambda$. Fix a Hermitian form on $M$ and an exhaustion by open subsets $M_j\Subset M_{j+1}\Subset M$, allowed to stabilize at $M$ when $M$ is compact. In Definition~2.4 take
\[
 \Sigma_{j,s}=\Sigma\cap M_j,\qquad
 h_{j,s}=h|_{M_j\setminus\Sigma},\qquad
 \lambda_{j,s}=\lambda_j=0,\qquad \theta=0.
\]
The approximants are smooth on the required complements, their norms form constant nondecreasing sequences converging to the $h$-norm, and their curvature errors are continuous and identically zero. These data verify Definition~2.4. Theorem~1.2 of the same reference gives Definition~2.6(1).
\end{proof}

These singular positivity statements do not assign a classical curvature tensor or a determinant-pinching inequality at points of $\Sigma$.

\subsection{Tautological pullback and symmetric powers}
\begin{lemma}\label{lem:exact-tautological-pullback}
Let $\mathcal F$ have positive generic rank on a smooth integral projective $Y$, and let $V=\pi^*\mathcal F/\mathrm{tors}$ be locally free on a smooth projective modification. For every integer $\ell\geq1$,
\begin{equation}\label{eq:ordinary-symmetric-pullback}
 \pi^*\operatorname{Sym}^{\ell}\mathcal F/\mathrm{tors}
 \simeq\operatorname{Sym}^{\ell}V.
\end{equation}
There is a projective birational morphism $q_\ell:\mathbb P_q(\operatorname{Sym}^{\ell}V)\to \mathbb P_{\mathrm{pr}}(\operatorname{Sym}^{\ell}\mathcal F)$, and
\begin{equation}\label{eq:exact-tautological-pullback}
 q_\ell^*\mathcal O_{\mathbb P_{\mathrm{pr}}
               (\operatorname{Sym}^{\ell}\mathcal F)}(1)
 =\mathcal O_{\mathbb P_q(\operatorname{Sym}^{\ell}V)}(1).
\end{equation}
\end{lemma}
\begin{proof}
Pullback commutes with the symmetric algebra. The quotient $\pi^*\mathcal F\twoheadrightarrow V$ gives a surjection $\pi^*\operatorname{Sym}^{\ell}\mathcal F\twoheadrightarrow \operatorname{Sym}^{\ell}V$, an isomorphism at the generic point. Its kernel is torsion, and contains all torsion because the target is locally free over the integral variety $\widetilde Y$. This proves \eqref{eq:ordinary-symmetric-pullback}, including possible torsion in the ordinary symmetric power downstairs. The universal rank-one quotient gives a morphism to the relative Proj in \eqref{eq:principal-projectivization}. The source is integral, and the morphism is an isomorphism over a dense Zariski-open subset of $Y$ where $\mathcal F$ is locally free and $\pi$ is an isomorphism. Properness makes its image closed. It therefore factors through, and surjects onto, the reduced principal component. This is $q_\ell$. The universal quotient line pulls back to the same quotient line upstairs, which proves \eqref{eq:exact-tautological-pullback} as an identity of line bundles. In particular, the identity has no exceptional-divisor correction.
\end{proof}

\begin{lemma}
\label{lem:symmetric-bigness}
If a positive-generic-rank coherent sheaf $\mathcal F$ on a smooth integral projective variety is tautologically big, then so is $\operatorname{Sym}^{\ell}\mathcal F$ for every integer $\ell\geq1$.
\end{lemma}
\begin{proof}
Suppose first that $V$ is a vector bundle, and write $p:\mathbb P_q(V)\to Y$, and $L=\mathcal O(1)$. For an ample line bundle $A$ on $Y$, bigness of $L$ gives some $q>0$ and a nonzero section
\begin{equation}\label{eq:big-section-criterion}
 s\in H^0(\mathbb P_q(V),L^q\otimes p^*A^{-1})
   =H^0(Y,\operatorname{Sym}^qV\otimes A^{-1}).
\end{equation}
This is the fixed-twist criterion for a big line bundle. Conversely, such a section implies bigness of $L$: choose $c>0$ so that $L\otimes p^*A^c$ is ample, using relative ampleness. In additive notation
\[
 (cq+1)L=(L+c p^*A)+c(qL-p^*A)
\]
is the sum of an ample line bundle and an effective divisor.

The power $s^\ell$ is nonzero in $H^0(Y,\operatorname{Sym}^{q\ell}V\otimes A^{-\ell})$, since the symmetric algebra of the generic fiber is an integral domain. In characteristic zero there is a natural split injection
\[
 \operatorname{Sym}^{q\ell}V
 \lhook\joinrel\longrightarrow
 \operatorname{Sym}^{q}(\operatorname{Sym}^{\ell}V).
\]
Indeed, embed the source in $V^{\otimes q\ell}$ by averaging over all permutations, arrange the factors in $q$ groups of size $\ell$, and take the symmetric quotients within each group and under permutations of the groups. Multiplication back into $\operatorname{Sym}^{q\ell}V$ is the identity, so this construction is injective and commutes with changes of frame. The image of $s^\ell$ is therefore a nonzero section of $\operatorname{Sym}^{q}(\operatorname{Sym}^{\ell}V)\otimes (A^\ell)^{-1}$. Applying \eqref{eq:big-section-criterion} to $\operatorname{Sym}^{\ell}V$ proves the bundle assertion, also recorded in \cite[Proposition~3.3]{FulgerRay2023}.

For general $\mathcal F$, choose a smooth projective modification as in Lemma~\ref{lem:exact-tautological-pullback}. Bigness is invariant under proper birational pullback of a line bundle on integral projective varieties. Thus the identity with $\ell=1$ transfers bigness to $V$. The bundle assertion and \eqref{eq:exact-tautological-pullback} for each $\ell$ transfer it back to the ordinary symmetric power of $\mathcal F$.
\end{proof}

\subsection{Tautological bigness}
\begin{theorem}[Tautological bigness of relative pluricanonical direct images]
\label{thm:prescribed-bigness}
Let $f:X\to Y$ be a surjective projective morphism of smooth connected complex projective varieties, with connected fibers, $d=\dim Y>0$ and $n=\dim X-\dim Y>0$. Fix an integer $m\geq2$, put $E_m=f_*\omega_{X/Y}^{\otimes m}$, and let $V_m=\pi^*E_m/\mathrm{tors}$ on a smooth projective modification $\pi:\widetilde Y\to Y$ for which $V_m$ is locally free. Assume the general fiber $F$ is of general type and $\operatorname{Var}(f)=d$. Assume that, for a general fiber $F$, at least one of the following holds:
\begin{enumerate}
\item the line bundle $mK_F$ is globally generated, or
\item the rational map $\varphi_m$ associated with $|mK_F|$ is birational onto its image.
\end{enumerate}
Then the pullback of $G$ to $V_m$ has the singular extension of Theorem~\ref{thm:actual-sheaf-extension} and is smooth and strictly Nakano positive on a nonempty Zariski-open subset. The torsion-free sheaf $E_m$ is tautologically big, as is each ordinary symmetric power $\operatorname{Sym}^{\ell}E_m$ for $\ell\geq1$, in convention \eqref{eq:principal-projectivization}.
\end{theorem}
\begin{proof}
On a dense Zariski-open subset of $Y$, the morphism $f$ is smooth and the fiberwise base-change maps for $E_m$ are isomorphisms. Under either alternative, Corollary~\ref{cor:unified-variation-rank} gives determinant-curvature rank $\operatorname{Var}(f)=d$ on $U_m^\circ$. Its semipositivity therefore makes it positive definite there, and Theorem~\ref{thm:determinant-pinching} gives strict Nakano positivity of $G$. Intersect $U_m^\circ$ with the isomorphism locus of $\pi$ to obtain the same strictness for its pullback to $V_m$.

Let $Z=\mathbb P_q(V_m)$, $p:Z\to\widetilde Y$ and $L=\mathcal O_Z(1)$. The singular metric from Theorem~\ref{thm:actual-sheaf-extension} induces the quotient metric on $L$. In a bundle frame with matrix $A$ and projective quotient row $\eta$, its local weight is
\begin{equation}\label{eq:fixed-tautological-weight}
 \varphi=\log(\eta A^{-1}\eta^*).
\end{equation}
Over the regularity locus this is psh. Estimate \eqref{eq:fixed-flattened-lower} bounds $A^{-1}$ above in every local frame of $V_m$, so these weights are locally bounded above on projective affine charts and extend across the entire complement. They glue, with the transition law of $L$, to a positive current
\[
 T=\frac{i}{2\pi}\partial\bar\partial\varphi\in c_1(L).
\]
At a point of the strict locus use a $G$-Chern-normal frame. For a base tangent $\xi$ and a projective tangent represented by a row $\zeta$, the Hermitian form underlying its curvature, apart from the common factor $1/(2\pi)$, is
\begin{equation}\label{eq:tautological-curvature-split}
 \frac{\eta K_\xi\eta^*}{|\eta|^2}
 +\frac{|\zeta|^2|\eta|^2-|\zeta\eta^*|^2}{|\eta|^4}.
\end{equation}
The mixed derivatives vanish in this normal frame. The first term is positive for $\xi\ne0$. The second is the vertical Fubini--Study form, positive for a nonzero projective tangent. Hence $T$ is smooth and positive definite over the strict locus, also when $r=1$ and the vertical space is zero.

The smooth projective $Z$ is compact K\"ahler. For $N=d+r-1$, Boucksom's volume inequality \cite[Proposition~3.1 and Theorem~1.2]{Boucksom2002} gives
\begin{equation}\label{eq:current-volume}
 \operatorname{vol}(L)\geq\int_Z(T_{\mathrm{ac}})^N>0.
\end{equation}
Strict positivity already gives positive mass on a relatively compact coordinate ball above the strict locus. Here $(T_{\mathrm{ac}})^N$ is the coefficientwise top power of the measurable density $T_{\mathrm{ac}}$, defined almost everywhere. Thus $L$ is big. The exact identity \eqref{eq:exact-tautological-pullback} for $\ell=1$ proves bigness of $E_m$, and Lemma~\ref{lem:symmetric-bigness} proves it for every positive ordinary symmetric power. 
\end{proof}

\begin{corollary}
\label{cor:all-large-levels}
Under the hypotheses on the family and its variation in Theorem~\ref{thm:prescribed-bigness}, there is an integer $M(n)$, depending only on the fiber dimension, such that for every integer $m\geq M(n)$ and every integer $\ell\geq1$, $\operatorname{Sym}^{\ell}E_m$ is tautologically big. At each such level its fixed-point metric has the extension and interior strictness in that theorem. No divisibility restriction on $m$ is required.
\end{corollary}
\begin{proof}
Hacon--McKernan \cite[Theorem~1.1]{HaconMcKernan2006} provides an integer $r_n$ such that the map associated with $|mK_F|$ is birational for every smooth projective general-type $n$-fold and every integer $m\geq r_n$. Set $M(n)=\max\{2,r_n\}$ and apply the second branch of Theorem~\ref{thm:prescribed-bigness} separately at each level. The regularity locus and metric constants may depend on $m$.
\end{proof}

\subsection{Viehweg bigness for smooth families}
For a bundle $V$ on a projective variety, weak positivity over a fixed nonempty Zariski-open subset $U$ means that, for every integer $p>0$ and an ample line bundle $A$, some integer $q>0$ makes $\operatorname{Sym}^{pq}V\otimes A^q$ globally generated over $U$ by its sections on $Y$. The choice of $A$ is immaterial \cite[Definition~2.11 and Lemma~2.14]{Viehweg1995}. We use the following criterion for Viehweg bigness:
\begin{equation}\label{eq:viehweg-big-criterion}
 \operatorname{Sym}^{N}V\otimes A^{-1}
 \text{ is weakly positive over some nonempty Zariski-open subset},\qquad N>0.
\end{equation}
Indeed, weak positivity with $p=2$ and multiplication of symmetric powers give generic global generation of $\operatorname{Sym}^{2qN}V\otimes A^{-q}$ for some $q>0$. This is the criterion in \cite[Theorem~1.1(3)]{FulgerRay2023}. These notions concern evaluation by global sections on the base. By contrast, global generation of $mK_F$ means surjectivity of $H^0(F,mK_F)\otimes\mathcal O_F\to mK_F$. When this holds on the geometric generic fiber, relative evaluation is surjective over a dense Zariski-open subset of the base.

\begin{corollary}[Viehweg bigness for smooth families]
\label{cor:smooth-viehweg-bigness}
Let $f:X\to Y$ be a smooth surjective projective morphism between smooth connected complex projective varieties, with connected positive-dimensional fibers, and assume $\dim Y>0$. Suppose the general fiber is of general type and $\operatorname{Var}(f)=\dim Y$. Fix $m\geq2$ such that either $mK_F$ is globally generated on a general fiber, or $\varphi_m$ is birational onto its image. Then
\[
 M=\det E_m\text{ is big},\qquad
 \operatorname{Sym}^{\ell}E_m\text{ is Viehweg big for every }\ell\geq1.
\]
These conclusions hold for every integer $m\geq M(n)$, where $M(n)$ is the threshold in Corollary~\ref{cor:all-large-levels}.
\end{corollary}
\begin{proof}
Smoothness of $f$ gives local freeness of $E_m$ and isomorphisms of its fiberwise base-change maps over all of $Y$. Lemma~\ref{lem:holder-continuity} makes $G$ continuous and positive definite there. Proposition~\ref{prop:compact-moment} and Theorem~\ref{thm:determinant-pinching} give one constant
\[
 \sigma_*=\frac{\varepsilon_*}{1-\alpha+r\varepsilon_*}>0
\]
such that the directional curvature on $U_m^{\mathrm{reg}}$ satisfies
\begin{equation}\label{eq:smooth-directional-control}
 i\Theta_{E_m,G}(\xi,\bar\xi)
 \geq\sigma_*\eta(\xi,\bar\xi)I,
 \qquad \eta=i\operatorname{Tr}\Theta_{E_m,G}.
\end{equation}
The determinant weights are continuous on $Y$ and psh on $U_m^{\mathrm{reg}}$. Removable singularities extend their positivity across the analytic complement. Under either alternative, Corollary~\ref{cor:unified-variation-rank} gives $\operatorname{rank}\eta=\operatorname{Var}(f)=\dim Y$ on $U_m^\circ$. Thus the positive current $c_1(M,\det G)$ is smooth and strictly positive there. The volume inequality \eqref{eq:current-volume}, applied to $M$ on $Y$, proves that $M$ is big.

Fix $\ell\geq1$, set $V=\operatorname{Sym}^{\ell}E_m$, and choose positive integers $a,b$ with $b<a\sigma_*$. The tensor curvature formula and its parallel symmetric summands give, on $U_m^{\mathrm{reg}}$, Griffiths semipositivity of
\begin{equation}\label{eq:smooth-negative-determinant-twist}
 W=\operatorname{Sym}^{a}V\otimes M^{-b\ell},
 \qquad i\Theta_W\geq_{\mathrm{Gr}}
             \ell(a\sigma_*-b)\eta\otimes I.
\end{equation}
The induced metric is continuous and positive definite on all of $Y$. Its logarithmic dual norms are psh on $U_m^{\mathrm{reg}}$ and extend to their continuous values across the complement. The induced tautological line metric is therefore continuous and semipositive. Its pullback to the normalization of any curve has nonnegative degree, so $W$ is nef. It is consequently weakly positive over $Y$ \cite[Proposition~2.9 and Remark~2.12(2)]{Viehweg1995}.

Since $M$ is big, for a fixed ample $A$ some positive integer $k$ gives $M^{b\ell k}=A\otimes\mathcal O_Y(D)$ with $D\geq0$. The bundle $\operatorname{Sym}^kW$ is weakly positive, and the multiplication map has quotient
\[
 \operatorname{Sym}^{ak}V\otimes M^{-b\ell k}.
\]
Weak positivity passes to this quotient and, after multiplication by the section of $D$, to $\operatorname{Sym}^{ak}V\otimes A^{-1}$ over $Y\setminus\operatorname{Supp}D$ \cite[Lemma~2.16(c) and Corollary~2.20]{Viehweg1995}. Criterion~\eqref{eq:viehweg-big-criterion} proves the assertion for $V$. The final assertion uses uniform pluricanonical birationality as in Corollary~\ref{cor:all-large-levels}. Only global continuity and the uniform lower bound for the evaluation quadratic form enter this argument. No log canonical threshold condition is needed.
\end{proof}

\begin{remark}[The positivity notion and the sheaf]
\label{rem:bigness-boundaries}
For vector bundles, tautological bigness is weaker than Viehweg bigness, which requires weak positivity after a negative ample twist of a positive reflexive symmetric power. It also does not imply generic global generation. For example $\mathcal O_{\mathbb P^1}(1)\oplus\mathcal O_{\mathbb P^1}(-1)$ is tautologically big by \eqref{eq:big-section-criterion}, but every positive symmetric power has a negative quotient line and is neither generically globally generated nor Viehweg big. Theorem~\ref{thm:prescribed-bigness} gives tautological bigness for ordinary symmetric powers of $E_m$. 

For nonfree sheaves, bigness of reflexive symmetric powers and tautological bigness retain different codimension-two information. The sheaf cannot be replaced during the proof by its double dual. For $\mathcal F=\mathcal I_p\otimes\mathcal O_{\mathbb P^2}(1)$, the latter is $\mathcal O(1)$, whereas the principal projectivization is $\operatorname{Bl}_p\mathbb P^2$ and its tautological line is $\mathcal O(H-D)$, with $D$ the exceptional divisor. Its $k$th power has $k+1$ sections and is not big on that surface, although the sheaf is Viehweg big under the reflexive-power convention. In the present proof the estimate for sections of $\mathcal E$ in \eqref{eq:uniform-coherent-section-estimate} and the universal quotient identity \eqref{eq:exact-tautological-pullback} retain precisely this exceptional vanishing information.
\end{remark}
\subsection{A counterexample at level one}\label{sec:counterexample}

At $m=1$ the linearized operator in Section~\ref{sec:curvature} vanishes. Global generation and maximal birational variation then need not imply tautological bigness. The construction below uses the constant-Prym and product-quotient examples of \cite[\S7]{Mumford1974}, \cite[Example~3.5 and Theorem~6.1]{CilibertoPardiniTovena2000}, and \cite[Proposition~2.3]{BauerCatanese2004}.

The canonical sections will come from a constant anti-invariant summand associated with a fixed abelian surface, while the quotient surfaces have nonconstant birational class.

\begin{theorem}\label{thm:level-one-counterexample}
There is a flat surjective projective morphism $f:X\longrightarrow B=\mathbb P^1$ with $X$ smooth and connected, connected fibers, and reduced normal-crossing singular fibers, such that the smooth general fiber $S$ is a surface with ample and globally generated canonical bundle and $\operatorname{Var}(f)=1$, but
\begin{equation}\label{eq:counterexample-direct-image}
 f_*\omega_{X/B}\simeq\mathcal O_B^{\oplus4}.
\end{equation}
All positive ordinary symmetric powers of this bundle fail to be tautologically big. The level-one fixed-point metric is flat on the smooth locus and extends as a smooth flat metric on the bundle in \eqref{eq:counterexample-direct-image}. Local monodromy on the cohomology of the smooth fibers is unipotent.
\end{theorem}

\begin{proof}
\emph{The pencil.} Let $E_a,E_b$ be elliptic curves with origins and set
\[
 A=E_a\times E_b,\qquad
 L=\mathcal O_{E_a}(2o_a)\boxtimes\mathcal O_{E_b}(2o_b),\qquad
 \sigma=(-1,-1).
\]
Each degree-two system is the quotient by inversion. Thus, for $q:A\longrightarrow Q=\mathbb P^1\times\mathbb P^1$,
\[
 L=q^*\mathcal O_Q(1,1),\qquad
 H^0(A,L)=q^*H^0(Q,\mathcal O_Q(1,1)),\qquad
 L^2=8.
\]
In particular, $L$ is ample and globally generated, $h^0(A,L)=4$, and the pullback linearization acts trivially on its section space. K\"unneth and positivity of the two degree-two factors give $H^1(A,L)=0$. Write the branch divisor of $q$ as four vertical and four horizontal rulings, with sixteen crossing points $p_{ij}$.

Choose a pencil in $|\mathcal O_Q(1,1)|=\mathbb P^3$ with two transverse base points outside the branch divisor. Require its parameter line to meet the quadric of reducible divisors transversely in two members with no branch ruling, and to meet the sixteen hyperplanes of divisors through the $p_{ij}$ at distinct points off that quadric. Such pencils form a nonempty Zariski-open subset of the Grassmannian of lines: the transversality and base-point conditions define dense Zariski-open subsets, and the forbidden hyperplane intersections, hyperplane--quadric intersections and reducible divisors containing a branch ruling all have codimension at least two in $\mathbb P^3$.

Choose a basis $s_0,s_1$ of the pulled-back pencil and set
\[
 \mathcal C=\{(a,[u:v])\in A\times B:us_0(a)+vs_1(a)=0\},
 \qquad g:\mathcal C\longrightarrow B.
\]
The eight base points on $A$ are transverse and avoid $A[2]$. Consequently $\mathcal C$ is their blowup and is smooth and projective. The involution $\sigma$ lifts, acts trivially on $B$, and has precisely sixteen fixed points, with tangent action $-I$. The exceptional curves over the base points are exchanged in pairs.

A smooth $(1,1)$ member is the graph of an isomorphism of the two projective lines. Its inverse image is smooth except possibly over a branch crossing: over one branch ruling the local equation is $w=\phi(z^2)$, and away from the branch divisor the covering is \'etale. At a crossing the quotient coordinates are $(z^2,w^2)$, and the pencil parameter satisfies
\[
 t-t_0=z^2A(z^2,w^2)+w^2B(w^2),
\]
where $A,B$ are units because both projections of the graph are isomorphisms. Taking analytic square roots of the units and a linear change of variables gives
\begin{equation}\label{eq:counterexample-fixed-node}
 t-t_0=xy,\qquad \sigma(x,y)=(-x,-y).
\end{equation}
There are sixteen such nodal members, at distinct parameters.

Each of the two reducible members lifts to two copies of $E_a$ and two copies of $E_b$, meeting at four transverse nodes away from $A[2]$. The pencil is a submersion near the blown-up base points. These are therefore all singular fibers on the complete parameter line $B$, and every fiber of $g$ is reduced and nodal.

A general member $C_t$ is a smooth connected curve. For connectedness, use $H^1(A,L^{-1})=H^1(A,L)^*=0$ in the divisor sequence to obtain $H^0(C_t,\mathcal O_{C_t})=\mathbb C$. Adjunction gives $g(C_t)=1+L^2/2=5$. It avoids $A[2]$, so its involution is free.

\emph{The surface family and its boundary.} Fix a smooth genus-two curve $H$ with hyperelliptic involution $\iota$. On $Z=\mathcal C\times H$ let $\tau=\sigma\times\iota$. The fixed locus consists of $16\cdot6=96$ points, and the tangent action at each is $-I$ in dimension three. Blow up precisely these points and take the quotient:
\[
 b:\widetilde Z\longrightarrow Z,\qquad
 p:\widetilde Z\longrightarrow X=\widetilde Z/\langle\tau\rangle,
 \qquad f:X\longrightarrow B.
\]
In each blowup chart the action and quotient map are
\[
 (a,b_1,c_1)\longmapsto(-a,b_1,c_1),\qquad
 (a,b_1,c_1)\longmapsto(a^2,b_1,c_1).
\]
Hence $X$ is smooth. It is projective, since it is a finite quotient of the projective variety $\widetilde Z$.

At a fixed point, use \eqref{eq:counterexample-fixed-node} and a coordinate $z$ on $H$ with $\iota(z)=-z$ to determine the boundary fibers. Before blowing up the local equation is $t-t_0=xy$ and $\tau(x,y,z)=(-x,-y,-z)$. In the $x$-chart of the blowup, write $y=xb_1$, $z=xc_1$. The quotient coordinates $u=x^2,b_1,c_1$ give $t-t_0=ub_1$. The $y$-chart is identical. In the $z$-chart, write $x=zb_1$, $y=zc_1$. The quotient equation is
\begin{equation}\label{eq:counterexample-quotient-fiber}
 t-t_0=ub_1c_1,\qquad u=z^2.
\end{equation}
All components have multiplicity one. Away from these fixed points the quotient is \emph{\'etale}. The nodal product fibers there are already normal crossing. Thus all singular fibers of $f$ are reduced normal crossing, in the local analytic sense.

The general fiber is
\[
 S_t=(C_t\times H)/\langle\sigma\times\iota\rangle,
 \qquad \pi_t:C_t\times H\longrightarrow S_t.
\]
Here $\pi_t$ is an \emph{\'etale} double cover, since $\sigma$ is free on $C_t$. These fibers are connected. Stein factorization, with normal base $B$, then gives connected fibers everywhere. The nonconstant morphism from the smooth integral variety $X$ to the smooth curve $B$ is flat: its local rings are torsion-free over the corresponding base DVRs, hence flat.

Outside small node neighborhoods, the family is a proper submersion over the disk and is differentiably trivial. Local monodromy of $g$ is therefore a product of Dehn twists about the disjoint vanishing cycles $\delta_j$ of its nodes. The annulus $xy=t$ contributes its core circle to a crossing cycle, giving
\[
 (T-I)\gamma=\sum_j\epsilon_j
 \langle\gamma,\delta_j\rangle\delta_j,\qquad (T-I)^2=0
 \quad\text{on }H_1(C_t,\mathbb Q).
\]
Here $\epsilon_j=\pm1$ depends on the orientation convention. The square vanishes because the $\delta_j$ have zero pairwise intersection. Monodromy on $H^0$ and $H^2$ is the identity. Pullback and the degree-two trace identify
\[
 H^k(S_t,\mathbb Q)=H^k(C_t\times H,\mathbb Q)^{\sigma\times\iota}.
\]
K\"unneth and constancy of $H$ therefore give unipotent monodromy on every cohomology group of $S_t$.

\emph{Global generation of the canonical bundle and variation.} The \emph{\'etale} cover gives $\pi_t^*K_{S_t}=K_{C_t}\boxtimes K_H$. Both factors are ample, so $K_{S_t}$ is ample. Let $\alpha_1,\alpha_2$ be a basis of $H^0(A,\Omega_A^1)$ and $\beta_1,\beta_2$ a basis of $H^0(H,K_H)$. The conormal sequence gives a surjection
\[
 \mathcal O_{C_t}^{\oplus2}=\Omega_A^1|_{C_t}
       \longrightarrow K_{C_t}.
\]
Thus the restricted $\alpha_i$ globally generate $K_{C_t}$. Inversion acts as $-1$ on these forms. The $\beta_j$ globally generate $K_H$ and are anti-invariant under the hyperelliptic involution. Therefore the four invariant products
\begin{equation}\label{eq:counterexample-product-sections}
 \alpha_i|_{C_t}\boxtimes\beta_j\qquad (1\leq i,j\leq2)
\end{equation}
descend and globally generate $K_{S_t}$.

For variation, the equivariant normal sequence of a smooth member $C=C_t\subset A$ gives
\[
 0\longrightarrow T_C\longrightarrow T_A|_C
          \longrightarrow N_{C/A}\longrightarrow0.
\]
Since $C$ is connected, $H^0(C,T_A|_C)=\operatorname{Lie}(A)$, and inversion acts as $-1$. Taking invariants, which is exact in characteristic zero, gives an injection
\[
 H^0(C,N_{C/A})^\sigma\hookrightarrow H^1(C,T_C)^\sigma.
\]
A nonzero pencil tangent is a nonzero invariant normal section: in $0\to\mathcal O_A\to L\to L|_C\to0$, the kernel of restriction on sections is precisely the line defining $C$, so an independent pencil member does not restrict to zero. Its image is the Kodaira--Spencer class $\kappa_C$, which is consequently nonzero.

As both curves have genus at least two, they have no vector fields, and K\"unneth gives
\[
 H^1(C_t\times H,T_{C_t\times H})
   =H^1(C_t,T_{C_t})\oplus H^1(H,T_H).
\]
The product family has class $(\kappa_C,0)\ne0$. The \emph{\'etale} quotient identifies this class with the pullback of the Kodaira--Spencer class of $S_t$, so that class is also nonzero. For a canonically polarized manifold, the local moduli space is the finite-automorphism quotient of a local universal deformation base \cite[\S2, Proposition~4]{Schumacher2018Moduli}. If our moduli map were constant, its local lift from a connected disk to that deformation base would have finite image and hence would be constant, forcing its Kodaira--Spencer map to vanish. Thus the moduli map is nonconstant. Finally $K_{S_t}$ is ample, so $S_t$ is its own canonical model and two such surfaces are birational only if they are isomorphic. It follows that the \emph{birational} variation is one.

\emph{The direct image on the complete base.} Let $j:\mathcal C\hookrightarrow A\times B$ be the incidence inclusion. The invariant volume form $\alpha_1\wedge\alpha_2$ trivializes $K_A$. Relative adjunction for this smooth incidence divisor is equivariant and gives
\[
 0\longrightarrow\mathcal O_{A\times B}
 \longrightarrow L\boxtimes\mathcal O_B(1)
 \longrightarrow j_*\omega_{\mathcal C/B}\longrightarrow0.
\]
Projection to $B$, together with $H^1(A,L)=0$, yields the exact sequence
\begin{equation}\label{eq:counterexample-adjunction}
 \begin{split}
 0\longrightarrow\mathcal O_B
  &\longrightarrow H^0(A,L)\otimes\mathcal O_B(1)
   \longrightarrow g_*\omega_{\mathcal C/B}\\
  &\longrightarrow H^1(A,\mathcal O_A)\otimes\mathcal O_B
   \longrightarrow0.
 \end{split}
\end{equation}
The first two terms are invariant and inversion acts by $-1$ on $H^1(A,\mathcal O_A)$. Taking the anti-invariant summands therefore gives an isomorphism on the whole of $B$:
\begin{equation}\label{eq:counterexample-anti-invariant}
 (g_*\omega_{\mathcal C/B})^-
   \simeq H^1(A,\mathcal O_A)\otimes\mathcal O_B
   \simeq\mathcal O_B^{\oplus2}.
\end{equation}

The ambient one-forms define a global frame of this summand. At a critical point the relative dualizing section corresponding to $\alpha$ is $(\alpha\wedge dg)\otimes(dt)^{-1}$, with the adjunction sign fixed once. Thus restriction defines a holomorphic map
\[
 H^0(A,\Omega_A^1)\otimes\mathcal O_B
 \longrightarrow(g_*\omega_{\mathcal C/B})^-.
\]
On a smooth ample member it is injective: a nonzero ambient form vanishing on its tangent would give $0=\int_A\frac i2\alpha\wedge\bar\alpha\wedge c_1(L)>0$. Both bundles are trivial of rank two by \eqref{eq:counterexample-anti-invariant}. The map has constant matrix on $\mathbb P^1$, so it is an isomorphism everywhere.

Hartogs extension of top-form coefficients across the blowup centers gives $b_*\omega_{\widetilde Z}=\omega_Z$. The extended form pulls back to the original one because they agree off the exceptional divisor. In a quotient chart $(a,b_1,c_1)\mapsto(u=a^2,b_1,c_1)$, an invariant top form is exactly
\[
 a h(a^2,b_1,c_1)\,da\wedge db_1\wedge dc_1
   =p^*\!\left(\tfrac12h(u,b_1,c_1)\,
                    du\wedge db_1\wedge dc_1\right).
\]
Together with \emph{\'etale} descent off the branch locus, this proves $(p_*\omega_{\widetilde Z})^\tau=\omega_X$. Tensoring by the pullback of $\omega_B^{-1}$ gives both relative identities. Invariants commute with pushforward, because the averaging idempotent splits the sheaf. Since $H^0(H,K_H)$ is two-dimensional and entirely anti-invariant,
\begin{align*}
 f_*\omega_{X/B}
 &\simeq
 \bigl(g_*\omega_{\mathcal C/B}\otimes H^0(H,K_H)\bigr)^{\sigma\times\iota}\\
 &\simeq (g_*\omega_{\mathcal C/B})^-\otimes H^0(H,K_H)
 \simeq\mathcal O_B^{\oplus4}.
\end{align*}
Thus \eqref{eq:counterexample-direct-image} holds over all of $B$. On the smooth locus, base change gives $p_g(S_t)=4$, and the descended sections \eqref{eq:counterexample-product-sections} form the indicated global frame.

\emph{The metric and symmetric powers.} At $m=1$ the operator is independent of its matrix input, so the fixed-point metric is the canonical $L^2$ metric and the linearized operator in \eqref{eq:moment-operator} vanishes. In the global frame $u_{ij}$ descending from \eqref{eq:counterexample-product-sections}, our conjugate-linear-first convention gives
\begin{equation}\label{eq:counterexample-flat-matrix}
 G_{ij,kl}(t)=\frac12
 \left(\int_{C_t}\frac i2\alpha_k\wedge\overline{\alpha_i}\right)
 \left(\int_H\frac i2\beta_l\wedge\overline{\beta_j}\right).
\end{equation}
The factor $1/2$ is the reciprocal of the degree of the \emph{\'etale} quotient. The second factor is constant, and the first equals
\[
 \int_A\frac i2\alpha_k\wedge\overline{\alpha_i}\wedge c_1(L),
\]
since all members have class $c_1(L)$. Thus the matrix in \eqref{eq:counterexample-flat-matrix} is constant positive definite. It defines the smooth flat extension of $G$ on $f_*\omega_{X/B}$ over $B$.

For every $\ell\geq1$,
\[
 \operatorname{Sym}^{\ell}(f_*\omega_{X/B})
    =\mathcal O_B^{\oplus R_\ell},\qquad R_\ell=\binom{\ell+3}{3}.
\]
Its quotient projectivization is $B\times\mathbb P^{R_\ell-1}$, and its tautological line bundle comes from the second factor. Its Iitaka dimension is $R_\ell-1$, smaller than the dimension $R_\ell$ of the projective bundle. It is not big.
\end{proof}

\section{Nefness and ampleness over curves}\label{sec:curves}

Unweighted direct images are nef over a curve even in the presence of arbitrary singular fibers. For log pairs satisfying the log canonical scheme fiber condition, a logarithmic boundary estimate removes the determinant curvature atoms and yields the ampleness--flatness dichotomy.

\subsection{General nefness result}
We begin with a quotient-degree estimate that applies in both settings.

\begin{lemma}\label{lem:curve-quotient-degree}
Let $E$ be a vector bundle of positive rank on a smooth connected complex projective curve $C$, and let $U\subset C$ have finite complement. Suppose that $h$ is a smooth Griffiths-semipositive Hermitian metric on $E|_U$ and that, locally on $C$, $h\geq c h_0$ on $U$ for a smooth reference metric $h_0$ and a constant $c>0$. For every quotient vector bundle $E\twoheadrightarrow Q$ of positive rank, the determinant of its quotient metric extends to a singular metric on the ordinary line bundle $\det Q$, with positive curvature current $T_Q$ and total mass $\deg Q$.

If a continuous semipositive $(1,1)$ form $\eta$ on $U$ satisfies
\[
 \frac{i}{2\pi}\Theta^h\geq\eta\otimes I_E,
\]
then
\begin{equation}\label{eq:curve-metric-quotient-degree}
 \deg Q\geq\operatorname{rk}(Q)\int_U\eta,
 \qquad \mu_{\min}(E)\geq\int_U\eta\geq0.
\end{equation}
The integral is finite. In particular $E$ is nef, and it is ample if $\int_U\eta>0$.
\end{lemma}
\begin{proof}
Write $q:E\twoheadrightarrow Q$ and let $h_Q$ and $(h_0)_Q$ be the quotient metrics. Their minimum-norm description gives, for $v\in Q_y$ and $y\in U$,
\[
 |v|_{h_Q}^2=\inf_{q(w)=v}|w|_h^2
 \geq c\inf_{q(w)=v}|w|_{h_0}^2
 =c|v|_{(h_0)_Q}^2.
\]
The reference quotient metric is smooth and positive definite across $C\setminus U$. Quotient curvature is Griffiths semipositive, so in any local holomorphic frame of $Q$ the weight $\phi_Q=-\log\det h_Q$ is subharmonic on $U$. The lower metric bound makes $\phi_Q$ locally bounded above. It therefore extends subharmonically across the finitely many omitted points. If the frame changes by a holomorphic invertible matrix $B$, its weight changes by $-\log|\det B|^2$, which is harmonic on the whole overlap. Uniqueness of subharmonic extension preserves this transformation law. Thus the extended weights define a metric on $\det Q$ itself, with
\[
 T_Q=\frac{i}{2\pi}\partial\bar\partial\phi_Q\geq0,
 \qquad \int_C T_Q=\deg Q.
\]
The mass identity follows by comparison with a smooth metric on $\det Q$: the weight difference is globally defined and locally integrable, and its $\partial\bar\partial$ integrates to zero on $C$.

On $U$, the quotient-curvature formula and its nonnegative second fundamental form term give
\[
 T_Q|_U\geq\operatorname{rk}(Q)\eta.
\]
Integration over $U$ and nonnegativity of the remaining mass of $T_Q$ prove the degree bound. Taking $Q=E$ shows that $\int_U\eta$ is finite. Taking the last Harder--Narasimhan quotient, which is a vector bundle on the smooth curve, gives the minimum-slope bound. The characteristic-zero curve criteria then give nefness and the asserted ampleness \cite[Proposition~6.4.7 and Theorem~6.4.15]{Lazarsfeld2004II}. The degree bound applies to quotients of every positive rank, including $E$ itself.
\end{proof}

\begin{theorem}[General nefness]\label{thm:curve-unweighted-nef}
Let $f:X\to C$ be a surjective morphism from a smooth connected complex projective variety to a smooth connected projective curve, with connected fibers. Then the ordinary direct image
\[
 E_m=f_*\omega_{X/C}^{\otimes m}
\]
is nef for every integer $m\geq1$, with the zero-rank case understood vacuously. No condition on the singular fibers is imposed.
\end{theorem}
\begin{proof}
Each $E_m$ is torsion free: a local section annihilated by a nonzero base function vanishes on a dense Zariski-open subset of $X$, hence everywhere. It is therefore locally free on $C$. Rank zero therefore means that $E_m=0$. If $\dim X=1$, the morphism is finite, and its connected general fiber has degree one in characteristic zero. Thus $f$ is an isomorphism and $E_m=\mathcal O_C$. Assume that $\dim X>1$ and $E_m\ne0$.

For $m\geq2$, choose a dense Zariski-open subset $U\subset U_m^{\mathrm{reg}}$ on which $f$ is smooth and the fiberwise base-change maps are isomorphisms. The metric $G$ is smooth and Nakano semipositive there. Since $E_m$ is already locally free, the identity of $C$ is an admissible modification in Theorem~\ref{thm:actual-sheaf-extension}. That theorem extends $G$ on $E_m$ and gives $G\geq c h_0$ locally, without any hypothesis on the singular fibers. Lemma~\ref{lem:curve-quotient-degree}, with $\eta=0$, proves $\mu_{\min}(E_m)\geq0$ and nefness.

For $m=1$, apply the adjoint direct-image theorem \cite[Theorem~21.1]{HaconPopaSchnell2018} with the trivial twisting line bundle $L=\mathcal O_X$ and its constant metric $1$. Its multiplier ideal is $\mathcal O_X$, so the resulting sheaf is precisely $f_*\omega_{X/C}$. The theorem gives its canonical $L^2$ metric a singular Griffiths-semipositive extension with the minimal extension property. On a smooth base-change locus this is the smooth fiberwise canonical $L^2$ pairing. Lemma~\ref{lem:extension-flattening}, with the identity modification, gives its local positive lower bound on $E_1$. Lemma~\ref{lem:curve-quotient-degree} again applies with $\eta=0$. Here the canonical $L^2$ extension theorem supplies the level-one argument.
\end{proof}

\subsection{Interior curvature and boundary mass}
Retain the unweighted setting of Theorem~\ref{thm:curve-unweighted-nef}, with $\dim X>1$, $m\geq2$ and $r=\operatorname{rk}E_m>0$. Put $\alpha=(m-1)/m$, and choose $U$ as above. The extended determinant metric has positive curvature current
\[
 T=c_1(\det E_m,\det G)
   =\frac{i}{2\pi}\partial\bar\partial(-\log\det G).
\]
The coefficient from \eqref{eq:full-moment} and \eqref{eq:sigma-definition} is
\[
 \sigma(y)=\frac{\varepsilon_y}{1-\alpha+r\varepsilon_y},
 \qquad 0<\sigma(y)\leq\alpha/r.
\]
This function is continuous on $U$, and Theorem~\ref{thm:determinant-pinching} gives
\[
 \frac{i}{2\pi}\Theta^G\geq\sigma(y)T\otimes I_{E_m}
 \quad\text{on }U.
\]

\begin{proposition}[Interior slope bound]\label{prop:curve-interior-slope}
In this setting, put $\varepsilon_{\mathrm{int}}=\int_U\sigma(y)T$. Every positive-rank quotient bundle $E_m\twoheadrightarrow Q$ satisfies
\begin{equation}\label{eq:curve-interior-slope}
 \deg Q\geq\operatorname{rk}(Q)\varepsilon_{\mathrm{int}},
 \qquad \mu_{\min}(E_m)\geq\varepsilon_{\mathrm{int}}\geq0,
 \qquad \varepsilon_{\mathrm{int}}\leq\frac{\alpha}{r}\deg E_m.
\end{equation}
Writing $a_p=T(\{p\})$ for the atom mass at $p\in C\setminus U$, one has the measure decomposition
\begin{equation}\label{eq:curve-boundary-mass}
 T=\mathbf1_U T+\sum_{p\in C\setminus U}a_p\delta_p,
 \qquad
 \deg E_m=\int_U T+\sum_{p\in C\setminus U}a_p,
 \qquad a_p\geq0,
\end{equation}
where $\delta_p$ has total mass one. The bundle $E_m$ is ample if $\varepsilon_{\mathrm{int}}>0$. In particular, $\int_U T>0$ implies $\varepsilon_{\mathrm{int}}>0$ and ampleness.
\end{proposition}
\begin{proof}
Apply Lemma~\ref{lem:curve-quotient-degree} to $G$ and the continuous form $\eta=\sigma T|_U$. This gives the quotient and minimum-slope bounds by integrating the quotient-curvature inequality on $U$ and retaining nonnegative boundary mass. The determinant current is a finite positive measure of total mass $\deg E_m$. Since $C\setminus U$ is finite, its restriction there is precisely the sum of atom masses in \eqref{eq:curve-boundary-mass}. Hence
\[
 0\leq\varepsilon_{\mathrm{int}}
 \leq\frac{\alpha}{r}\int_U T
 \leq\frac{\alpha}{r}\deg E_m<\infty.
\]
If $T|_U$ has positive mass, some relatively compact open subset of $U$ has positive mass. The continuous positive function $\sigma$ has a positive minimum on its closure, so its weighted integral is positive. The ampleness conclusion follows from the lemma.
\end{proof}

Positive total degree in the unrestricted setting need not provide positive interior mass, since \eqref{eq:curve-boundary-mass} includes boundary atoms. The lc scheme-fiber hypotheses below remove those atoms. In relative dimension zero the metric $G$ on $\mathcal O_C$ is constant, so $T=0$ and the slope and mass statements are immediate.

\subsection{Boundary estimates for lc scheme fibers}
We now allow a horizontal rational divisor $\Delta\geq0$ and assume
\begin{equation}\label{eq:curve-pair}
 (X,\Delta)\text{ is klt},\qquad
 (X,\Delta+X_c)\text{ is log canonical for every }c\in C,
\end{equation}
where $f:X\to C$ has the smooth projective source and base and connected fibers of Theorem~\ref{thm:curve-unweighted-nef}, and $X_c=f^*c$ is the scheme fiber, with its multiplicities. For each positive integer $j$ with $j\Delta$ integral put
\[
 L_j=\mathcal O_X\bigl(j(K_{X/C}+\Delta)\bigr),\qquad E_j=f_*L_j.
\]
These direct images are torsion free, hence locally free on $C$. Their fibers need not equal the complete section spaces of singular fibers. The weighted boundary, slope and ampleness--flatness statements below use \eqref{eq:curve-pair}. Fix $m\geq2$ with $m\Delta$ integral and $E_m\ne0$, and put $n=\dim X-1$, $r=\operatorname{rk}E_m$ and $\alpha=(m-1)/m$. If $n=0$, then $f$ is an isomorphism, $\Delta=0$ and $E_j=\mathcal O_C$, so the assertions are immediate. Assume henceforth that $n\geq1$.

Choose a dense Zariski-open subset $U\subset C$ as in Theorem~\ref{thm:controlled-regularity}, on which the family is smooth, the fiber pairs are klt, and the fiberwise base-change maps are isomorphisms. Its complement is finite, and $G$ is smooth and Nakano semipositive on $U$. The following boundary estimate supplies the extension in this weighted setting.

\begin{proposition}\label{prop:curve-boundary}
In a holomorphic frame of $E_m$ near any $c\in C$, with $t(c)=0$, there are constants $c_0,C_0>0$ such that, for small $t\ne0$ in $U$,
\begin{equation}\label{eq:curve-log-bound}
 c_0 I\leq G(t)\leq C_0\bigl(1+\log|t|^{-1}\bigr)^{nm}I.
\end{equation}
The metric extends as a singular Griffiths-semipositive metric on $E_m$. The determinant curvature current
\begin{equation}\label{eq:curve-determinant-current}
 T=\frac{i}{2\pi}\partial\bar\partial(-\log\det G)\in c_1(E_m)
\end{equation}
is positive, has no atoms, and satisfies
\begin{equation}\label{eq:curve-interior-mass}
 \int_U T=\int_C T=\deg E_m.
\end{equation}
The tautological quotient metric has zero Lelong numbers.
\end{proposition}
\begin{proof}
First, every fiber is reduced. At a prime component of $X_c$, the coefficient in $\Delta+X_c$ is its multiplicity, since $\Delta$ is horizontal. Log canonicity bounds that coefficient by one. The Cartier fiber is therefore generically reduced. As a hypersurface in smooth $X$ it is Cohen--Macaulay, so it has no embedded components and is reduced.

For the lower bound, multiplication by $t$ and left exactness of direct image give, without a base-change assumption at $c$,
\begin{equation}\label{eq:curve-fiber-injection}
 E_m\otimes\mathbb C(c)\lhook\joinrel\longrightarrow
 H^0(X_c,L_m|_{X_c}).
\end{equation}
Indeed the kernel of restriction of a local direct-image section is exactly $tE_m$, by the sequence $0\to L_m\xrightarrow{t}L_m\to L_m|_{X_c}\to0$. A section of a line bundle on a reduced fiber that vanishes on the smooth locus outside $\Delta$ vanishes everywhere. Finitely many evaluations at such points therefore separate the finite-dimensional image of \eqref{eq:curve-fiber-injection}. Around these points choose small submersion charts $(t,z)$, local frames of $L_m$, and fixed smaller fiber polydiscs. These charts avoid $\Delta$ and cover smooth parts of $X_c$. If $s_v=\sum v_i s_i$ in the chosen bundle frame, the sum $J(t,v)$ of its scalar integrals $\int |s_v|^{2/m}_\Delta$ over these smaller polydiscs is continuous in $(t,v)$. On the unit sphere $|v|=1$ it is strictly positive at $t=0$, since one of the separating evaluations is nonzero. Compactness gives $J(t,v)\geq c>0$ for all such $v$ and small $t$. A fixed finite overlap bound gives
\[
 I_{s_v}(t):=\int_{X_t}|s_v|^{2/m}_\Delta\geq c'|v|^{2/m}.
\]
Balanced comparison \eqref{eq:balanced-comparison} now implies $G(s_v,s_v)\geq r^{-(m-1)}I_{s_v}^m\geq c_0|v|^2$.

For the upper bound, consider a frame section $s$. Its divisor $D_s$ as a section of $L_m$ is effective. Resolve $\Delta+X_c+D_s$. In a normal-crossing chart upstairs, absorb a unit into a vertical coordinate and write
\begin{equation}\label{eq:curve-monomial}
 t=z_1^{b_1}\cdots z_k^{b_k},\qquad
 s=u(z)\prod_{i=1}^{n+1}z_i^{\ell_i}
       \frac{(dz_1\wedge\cdots\wedge dz_{n+1})^{\otimes m}}{(dt)^m},
\end{equation}
where $u$ is a unit and the $\ell_i$ may be negative. For every vertical divisor $E_i$,
\begin{equation}\label{eq:curve-discrepancy}
 1+\ell_i/m=A_{X,\Delta}(E_i)
                +\operatorname{ord}_{E_i}(\pi^*D_s)/m\geq b_i.
\end{equation}
The last inequality is exactly the valuation form of $\operatorname{lc}(X,\Delta+X_c)$, including exceptional divisors. For horizontal coordinates, klt and effectiveness give $1+\ell_i/m>0$, so their integrals are finite.

Eliminating $z_k$ in \eqref{eq:curve-monomial} contributes $|z_k/(b_k t)|^2$ to the scalar density. Integrating the angular and horizontal variables and putting $x_i=-\log|z_i|$, $L=\log|t|^{-1}$, bounds the remaining integral by a constant times
\begin{equation}\label{eq:curve-simplex-bound}
 |t|^{-2}\int_{\substack{x_i\geq0\\\sum_{i=1}^k b_i x_i=L}}
       \exp\left(-2\sum_{i=1}^k(1+\ell_i/m)x_i\right)d\sigma
 \leq C(1+L)^{k-1}.
\end{equation}
Rescaling chart radii only translates the lower limits and changes constants. The finitely many choices of a root for $z_k$ also change only constants, uniformly in $\arg t$. The inequality follows from \eqref{eq:curve-discrepancy}: the exponential is at most $e^{-2L}$, which cancels $|t|^{-2}$, and the simplex has volume $O((1+L)^{k-1})$. For $k=1$ its zero-dimensional measure is a fixed constant. A finite cover of the compact resolved central fiber covers a small tube by properness. Since $k-1\leq n$, we obtain $I_s(t)\leq C_s(1+\log|t|^{-1})^n$. Comparison gives $G(s,s)\leq I_s^m$. Applying this to the frame and using $G\leq(\operatorname{Tr}G)I$ proves the upper matrix bound. For the unweighted Narasimhan--Simha estimate, see also \cite[Theorem~1.1 and Corollary~1.3]{Takayama2016}.

For a local holomorphic dual section $\lambda$, the function $\log|\lambda|_{G^*}^2$ is subharmonic on $U$ and locally bounded above by the lower bound in \eqref{eq:curve-log-bound}. It extends uniquely across the omitted points. This gives the singular Griffiths extension on $E_m$. Likewise $\phi=-\log\det G$ is subharmonic on $U$, and
\begin{equation}\label{eq:curve-weight-bounds}
 -C-nmr\log(1+\log|t|^{-1})\leq\phi(t)\leq C.
\end{equation}
Its extension has Lelong number zero: dividing these bounds by $\log|t|^2$ gives zero for its logarithmic singularity coefficient. On a curve that coefficient is the mass of the atom of $(i/2\pi)\partial\bar\partial\phi$. Thus $T$ has no atoms. Its total mass is $\deg E_m$, since the local weights define a metric on the ordinary determinant line. The finite complement of $U$ carries no mass, proving \eqref{eq:curve-interior-mass}.

In a projective quotient chart with row coordinate $\eta$, the weight of $\mathcal O_{\mathbb P_q(E_m)}(1)$ is $\log(\eta G^{-1}\eta^*)$, as in \eqref{eq:fixed-tautological-weight}. It is psh and bounded above locally, and its lower bound on a compact projective chart is $-C-nm\log(1+\log|t|^{-1})$. Along a transversal disk with fixed projective coordinate its Lelong number is zero. The ambient Lelong number is no larger than this restricted one and is nonnegative. Hence it too is zero, at every point over the boundary.
\end{proof}

\subsection{Quotient curvature and slopes}
For the weighted metric $G$ on $U$, use the same continuous positive coefficient
\[
 \sigma(y)=\frac{\varepsilon_y}{1-\alpha+r\varepsilon_y},
 \qquad 0<\sigma(y)\leq\alpha/r,
\]
from \eqref{eq:full-moment} and \eqref{eq:sigma-definition}. The full pinching theorem, in particular its curve restriction, gives
\begin{equation}\label{eq:curve-pinching}
 \frac{i}{2\pi}\Theta^G\geq \sigma(y)T\otimes I_{E_m}.
\end{equation}
The coefficient need not have a positive lower bound near $C\setminus U$.

\begin{proposition}\label{prop:quotient-slope}
Every nonzero quotient bundle $E_m\twoheadrightarrow Q'$ satisfies
\begin{equation}\label{eq:curve-quotient-slope}
 \deg Q'\geq\operatorname{rk}(Q')\varepsilon_G,\qquad
 \varepsilon_G:=\int_U \sigma(y)T,\qquad
 0\leq\varepsilon_G\leq\frac{\alpha}{r}\deg E_m.
\end{equation}
If $\deg E_m>0$, then $\varepsilon_G>0$.
\end{proposition}
\begin{proof}
Give $Q'$ the quotient metric. The lower bound of \eqref{eq:curve-log-bound} passes to it by the minimum-norm formula in Lemma~\ref{lem:curve-quotient-degree}. A bounded local right inverse of the quotient map gives the same logarithmic upper bound. Apply that lemma with $\eta=\sigma T|_U$. It extends the determinant metric on the ordinary $\det Q'$ and gives the quotient-degree inequality. The bound $\sigma\leq\alpha/r$ and \eqref{eq:curve-interior-mass} give the upper bound for $\varepsilon_G$. If $\deg E_m>0$, the atom-free identity \eqref{eq:curve-interior-mass} gives positive mass on $U$. Some relatively compact open subset has positive mass, and $\sigma$ has a positive minimum on its closure. Thus $\varepsilon_G>0$. For $\Delta=0$, this is precisely the integral $\varepsilon_{\mathrm{int}}$ of Proposition~\ref{prop:curve-interior-slope}.
\end{proof}

\subsection{Ampleness or flatness}
\begin{theorem}[Ampleness or flatness]\label{thm:curve-dichotomy}
Let $f:X\to C$ be a surjective projective morphism with connected fibers, where $X$ is smooth connected projective and $C$ is a smooth connected projective curve. Let $\Delta\geq0$ be a horizontal rational divisor such that $(X,\Delta)$ is klt and $(X,\Delta+X_c)$ is log canonical for every scheme fiber $X_c=f^*c$, with its multiplicities. Put $E_m=f_*\mathcal O_X(m(K_{X/C}+\Delta))$. Fix an integer $m\geq2$ such that $m\Delta$ is integral and $E_m\ne0$. Then $E_m$ is nef. Its fixed-point metric $G$ extends as a singular Griffiths-semipositive metric on $E_m$, and its quotient metric on $\mathcal O_{\mathbb P_q(E_m)}(1)$ has zero Lelong numbers everywhere. Exactly one of the following holds:
\begin{enumerate}
\item $\deg E_m>0$, and $E_m$ is ample.
\item $\deg E_m=0$, and $G$ extends smoothly as a Hermitian-flat metric on $E_m$.
\end{enumerate}
\end{theorem}

No general-type, variation, global-generation or birationality hypothesis is required. The zero-degree case uses the following removability lemma.
\begin{lemma}\label{lem:flat-removability}
Let $G$ be a smooth Hermitian-flat metric on a holomorphically trivial bundle on a punctured disk. If in its holomorphic frame
\[
 cI\leq G(t)\leq C(1+\log|t|^{-1})^N I
\]
for fixed $c,C>0$ and $N\geq0$, then $G$ extends smoothly as a flat metric across the origin in the extension defined by that frame.
\end{lemma}
\begin{proof}
On the universal cover choose a parallel orthonormal holomorphic frame $P$, and express the given single-valued holomorphic frame as $e=PB$. Then $B$ is holomorphic, $G=B^*B$, and
\[
 \|B\|\leq C'(1+\log|t|^{-1})^{N/2},\qquad
 \|B^{-1}\|\leq c^{-1/2}.
\]
The parallel frame has unitary monodromy. After a constant unitary change, its inverse is diagonal with entries $e^{2\pi i\lambda_j}$, $0\leq\lambda_j<1$. The $j$th row $b_j$ of $B$ has that monodromy, so $a_j=t^{-\lambda_j}b_j$ is single-valued holomorphic on the punctured disk. Its Laurent coefficient of order $-k$, $k\geq1$, is bounded, on $|t|=\rho$, by
\[
 C'\rho^{k-\lambda_j}(1+\log\rho^{-1})^{N/2}\longrightarrow0.
\]
Every negative coefficient vanishes and $a_j$ extends holomorphically. If $\lambda_j>0$, then $b_j=t^{\lambda_j}a_j$ tends to zero. But $G\geq cI$ bounds the smallest singular value of $B$ below by $\sqrt c$, and hence bounds every row norm below by $\sqrt c$. This contradiction forces every $\lambda_j=0$. Thus $B$ is single-valued and extends holomorphically, with invertible value at zero by the same singular-value bound. The matrix $B^*B$ is the required smooth flat extension.
\end{proof}

\begin{proof}[Proof of Theorem~\ref{thm:curve-dichotomy}]
Proposition~\ref{prop:curve-boundary} gives the metric extension, zero tautological Lelong numbers and $\deg E_m\geq0$. Apply \eqref{eq:curve-quotient-slope} to the last Harder--Narasimhan quotient of $E_m$ to get
\begin{equation}\label{eq:curve-min-slope}
 \mu_{\min}(E_m)\geq\varepsilon_G.
\end{equation}
Over a complex projective curve the slope criteria give nefness for $\mu_{\min}\geq0$ and ampleness for $\mu_{\min}>0$ \cite[Proposition~6.4.7 and Theorem~6.4.15]{Lazarsfeld2004II}. Since $\varepsilon_G>0$ when $\deg E_m>0$, this proves ampleness in the positive-degree case.

If $\deg E_m=0$, the positive current $T$ is zero. On $U$, the semipositive endomorphism-valued curvature has zero trace and is therefore zero. Apply Lemma~\ref{lem:flat-removability} at each of the finitely many omitted points, using Proposition~\ref{prop:curve-boundary}. Uniqueness on the punctured neighborhoods makes these smooth flat extensions of $G$ agree on overlaps.
\end{proof}
The flatness conclusion is compatible with the degree-zero direct-image theorem of \cite[Theorem~5.2]{CaoPaun2017}. The argument here identifies $G$ directly through its growth and monodromy.

\begin{corollary}\label{cor:curve-multiples}
Under \eqref{eq:curve-pair}, suppose $k\geq1$, $k\Delta$ is integral, and $\deg E_k>0$. Then $E_{qk}$ is ample for every positive integer $q$ such that $qk\geq2$.
\end{corollary}
\begin{proof}
Pure powers of a nonzero generic section show $E_{qk}\ne0$. If it had degree zero, Theorem~\ref{thm:curve-dichotomy} would make it unitary flat. For every $c$, the argument of \eqref{eq:curve-fiber-injection} gives $E_k\otimes\mathbb C(c)\hookrightarrow H^0(X_c,L_k|_{X_c})$. A nonzero section of an invertible sheaf on a reduced scheme has a nonzero $q$th power: a nonzero local coefficient is not nilpotent. The power map into $E_{qk}\otimes\mathbb C(c)$ is therefore nonzero on every nonzero vector, since its subsequent restriction is that nonzero power on $X_c$.

On $p:\mathbb P_q(E_k^*)=\mathbb P_{\mathrm{lines}}(E_k)\to C$, the multiplication morphism consequently gives a nowhere-vanishing line-subbundle map
\[
 \mathcal O(-q)\lhook\joinrel\longrightarrow p^*E_{qk}.
\]
A line subbundle of a unitary-flat bundle has seminegative curvature. Its dual $\mathcal O(q)$ is semipositive. Thus $\mathcal O_{\mathbb P_q(E_k^*)}(1)$ is nef. The curve criterion makes $E_k^*$ nef, forcing $\deg E_k\leq0$, a contradiction. The dichotomy now makes $E_{qk}$ ample. This argument uses only the fiber injection and reducedness, not surjectivity of multiplication or isomorphisms of the base-change maps at singular fibers.
\end{proof}

Returning to the unweighted setting, we obtain both clauses of \cite[Theorem~4.10]{FujinoSemipositivity2015}.
\begin{corollary}[Viehweg's ampleness theorem]\label{cor:curve-fujino}
Let $f:X\to C$ be a surjective morphism from a smooth connected complex projective variety to a smooth connected projective curve, with connected fibers, and put $E_j=f_*\omega_{X/C}^{\otimes j}$.
\begin{enumerate}
\item For every $j\geq1$, $E_j$ is nef, with the zero case understood vacuously.
\item Suppose in addition that $f$ is semistable, so that every scheme fiber is a reduced normal-crossing divisor. If $\deg\det E_k>0$ for some integer $k\geq1$, then $E_{k'}$ is ample for every multiple $k'$ of $k$ with $k'\geq2$.
\end{enumerate}
\end{corollary}
\begin{proof}
The first assertion follows from Theorem~\ref{thm:curve-unweighted-nef}. In the semistable case, the local model is $t=z_1\cdots z_\ell$ in smooth coordinates. The convention is reduced normal crossings, without requiring the irreducible components to be globally smooth. Thus $(X,X_c)$ is log canonical for every $c$, and \eqref{eq:curve-pair} holds with $\Delta=0$. Corollary~\ref{cor:curve-multiples} applies, including when $k=1$.
\end{proof}
The semistability convention is also stated before Theorem~0.3 of \cite{ViehwegZuo2001}. Corollary~\ref{cor:curve-fujino} gives an analytic proof of the theorem requested by Fujino in \cite[Problem~4.11]{FujinoSemipositivity2015}, using the analytic extension inputs in Theorem~\ref{thm:actual-sheaf-extension} and, at level one, \cite[Theorem~21.1]{HaconPopaSchnell2018}. Same-level unweighted ampleness also follows from \cite[Proposition~2.7]{ViehwegZuo2001}. Here $G$ supplies the quotient-slope bound, and the lc-fiber growth estimate removes the boundary atoms.

\subsection{Global generation and interpolation}\label{sec:consequences}

Retain the hypotheses of Theorem~\ref{thm:curve-dichotomy} and suppose that $E=E_m$ has positive degree. Put $M=\det E$, $e=\deg E$, $r=\operatorname{rk}E$, and $g=g(C)$. Proposition~\ref{prop:quotient-slope} gives $\mu_{\min}(E)\geq\varepsilon_G>0$. The last Harder--Narasimhan quotient has positive integral degree and rank at most $r$, so $\mu_{\min}(E)\geq1/r$. Thus
\begin{equation}\label{eq:interpolation-rho}
 \rho=\max\{\varepsilon_G,1/r\}\leq\mu_{\min}(E).
\end{equation}
\begin{corollary}\label{cor:curve-interpolation}
Let $a\geq1$ and $b\geq0$ be integers, let $A$ be a line bundle on $C$, and let $D$ be an effective divisor of degree $s\geq0$, possibly nonreduced. Put
\[
 B_{a,b,A}=\operatorname{Sym}^aE\otimes M^{-b}\otimes A.
\]
If
\begin{equation}\label{eq:interpolation-inequality}
 a\rho-be+\deg A>2g-2+s,
\end{equation}
then $H^1(C,B_{a,b,A}(-D))=0$, and the restriction map
\begin{equation}\label{eq:interpolation-map}
 H^0(C,B_{a,b,A})\longrightarrow H^0(D,B_{a,b,A}|_D)
\end{equation}
is surjective.
\end{corollary}
\begin{proof}
Tensor and symmetric powers of semistable bundles on a complex curve are semistable \cite[Corollary~6.4.14]{Lazarsfeld2004II}. The filtration induced by the Harder--Narasimhan filtration of $E$ on $\operatorname{Sym}^aE$ has graded factors $\bigotimes_j\operatorname{Sym}^{a_j}Q_j$, where $\sum_j a_j=a$. Their slopes are at least $a\mu_{\min}(E)$. The symmetric power of the last quotient attains this slope. Consequently
\[
 \mu_{\min}(\operatorname{Sym}^aE)=a\mu_{\min}(E),\qquad
 \mu_{\min}(B_{a,b,A}(-D))
 =a\mu_{\min}(E)-be+\deg A-s>2g-2.
\]
A nonzero map from $B_{a,b,A}(-D)$ to $K_C$ would have quotient-line image $K_C(-Z)$ for an effective divisor $Z$, contrary to this slope bound. Serre duality gives the vanishing, and the restriction sequence gives surjectivity.
\end{proof}

For divisors of length one, the criterion gives global generation of $B_{a,b,A}$ on $C$.

\section{Numerical dimension and vanishing theorems}\label{sec:numerical-vanishing}

Throughout this section let $f:X\to Y$ be a smooth surjective projective morphism between smooth connected complex projective varieties, with connected positive-dimensional fibers of general type. Write
\[
 d=\dim Y\geq1,\qquad E=E_m=f_*\omega_{X/Y}^{\otimes m},\qquad
 M=\det E,\qquad \alpha=\frac{m-1}{m},
\]
where $m\geq2$ and $r=\operatorname{rank}E>0$. The sheaf $E$ is locally free and its fiberwise base-change maps are isomorphisms on all of $Y$ by \eqref{eq:complete-base-change}. Put $v=\operatorname{Var}(f)$ and assume one of the following conditions on a general fiber $F$:
\begin{enumerate}
\item[(G)] $mK_F$ is globally generated;
\item[(B)] the $m$-canonical map $\varphi_m$ is birational onto its image.
\end{enumerate}
Keep $U_m^{\mathrm{reg}}$ fixed as in \eqref{eq:regularity-open}, and choose the Zariski-open subset $U_m^\circ$ from Corollary~\ref{cor:unified-variation-rank}. On $U_m^\circ$, the determinant form $\eta=i\operatorname{Tr}\Theta_{E,G}$ has rank $v$ and $\ker K=(\ker\eta)\otimes E$. The kernel of $\eta$ is $\ker\rho$ in case (G) and $\ker db_m$ in case (B). In the latter case take $U_m^\circ\subseteq U_m^{\mathrm{var}}$, the domain of the level-$m$ classifier in \eqref{eq:variation-open}. These rank and kernel identities lead to numerical-dimension formulas through bounded-potential intersections, and independently to vanishing through approximate $L^2$ solutions. Neither argument requires a log canonical threshold condition.

\subsection{Continuous metrics and numerical dimension}
Proposition~\ref{prop:compact-moment} proves continuity of the evaluation quadratic form $\mathfrak q_y$ on the complex vector spaces $\operatorname{End}(E_y)$ over all of $Y$ and gives
\begin{equation}\label{eq:global-full-moment}
 \varepsilon_*=\min_{\substack{y\in Y\\
                 \|B\|_{\mathrm{HS},G_y}=1}}\mathfrak q_y(B)>0,
 \qquad
 \sigma_*=\frac{\varepsilon_*}{1-\alpha+r\varepsilon_*}.
\end{equation}
This constant will be used throughout the section. By Theorem~\ref{thm:determinant-pinching}, $0<\sigma_*\leq\alpha/r$.

\begin{proposition}\label{prop:continuous-determinant-current}
The metric $G$ is continuous and positive definite on $Y$, and smooth on $U_m^{\mathrm{reg}}$. On $U_m^{\mathrm{reg}}$ put $\eta=i\operatorname{Tr}\Theta_{E,G}$. Then
\begin{equation}\label{eq:regular-locus-pinching}
 \eta\geq0,\qquad
 \sigma_*\eta\otimes I_E\leq_{\mathrm{Nak}}i\Theta_{E,G}
 \leq_{\mathrm{Nak}}\bigl[r-(r^2-1)\sigma_*\bigr]\eta\otimes I_E.
\end{equation}
On $U_m^\circ$ the form $\eta$ has rank $v$. The determinant metric defines a positive current $T=c_1(M,\det G)$ on $Y$ with continuous local potentials, and $T=\eta/(2\pi)$ on $U_m^{\mathrm{reg}}$.
\end{proposition}
\begin{proof}
Continuity and smoothness follow from Lemma~\ref{lem:holder-continuity} and Theorem~\ref{thm:controlled-regularity}. The uniform lower bound for $\mathfrak q_y$ and Theorem~\ref{thm:determinant-pinching} give \eqref{eq:regular-locus-pinching}. Corollary~\ref{cor:unified-variation-rank} gives rank $v$ on $U_m^\circ$. In a local frame of $M$, the weight $-\log\det G$ is continuous and psh off the proper analytic set $Y\setminus U_m^{\mathrm{reg}}$. The removable-singularity theorem for locally bounded-above psh functions extends it across that set. Continuity identifies the extension with the original weight. These weights define $T$, and its restriction to $U_m^{\mathrm{reg}}$ is $\eta/(2\pi)$.
\end{proof}

The following lemma identifies numerical dimension from curvature rank and continuous potentials.
\begin{lemma}\label{lem:bounded-potential-rank}
Let $W$ be a smooth connected projective variety of dimension $N$, and let $L$ carry a continuous Hermitian metric $h$ with psh local weights. Put $S=c_1(L,h)$. If $S$ is smooth of constant rank $s$ on the complement of a proper closed analytic subset $A$, then $L$ is nef and $\nu(L)=s$.
\end{lemma}
\begin{proof}
For every integral curve $C\subset W$, including a curve contained in $A$, the metric pulls back to finite continuous subharmonic weights on the normalization of $C$. Its curvature is a positive measure whose mass is $\deg(L|_C)$: compare with a smooth metric and integrate the distributional $\partial\bar\partial$ of their global continuous weight difference. By Stokes' theorem, this distributional derivative has total mass zero. Thus every curve has nonnegative degree, so $L$ is nef.

Here and below products of currents with locally bounded psh potentials are Bedford--Taylor products, as recalled in \cite[\S1.2]{BoucksomEyssidieuxGuedjZeriahi2010}. To verify their cohomology classes, choose a smooth reference metric with curvature $\theta\in c_1(L)$ and write
\[
 S=\theta+\frac{i}{2\pi}\partial\bar\partial\psi,
 \qquad \psi\in C^0(W).
\]
The local Bedford--Taylor definition and closedness give, for $j\geq1$,
\[
 S^j-\theta\wedge S^{j-1}
 =\frac{i}{2\pi}\partial\bar\partial(\psi S^{j-1}).
\]
Since $S^{j-1}$ has measure coefficients, multiplication by the continuous function $\psi$ is defined. All terms are global, so induction gives $[S^j]=c_1(L)^j$. Hence, for any K\"ahler form $\Omega$ on $W$,
\begin{equation}\label{eq:bounded-potential-intersections}
 c_1(L)^k\cdot[\Omega]^{N-k}
 =\int_W S^k\wedge\Omega^{N-k},\qquad 0\leq k\leq N.
\end{equation}

Locally the measure on the right is a mixed product of bounded psh potentials, including smooth local potentials for $\Omega$. Such Bedford--Taylor products agree with the non-pluripolar products and put no mass on pluripolar sets \cite[\S1.2]{BoucksomEyssidieuxGuedjZeriahi2010}. A proper analytic set is locally pluripolar, so the measure gives no mass to $A$. For $k=0$ this follows from smoothness of the volume form. These facts use bounded potentials and impose no bigness hypothesis on $L$.

We may therefore compute \eqref{eq:bounded-potential-intersections} on $W\setminus A$. The integrand vanishes for $k>s$ and is strictly positive there for $k=s>0$. A relatively compact coordinate ball in $W\setminus A$ therefore contributes positive mass when $k=s>0$. For $s=0$ all positive-degree intersections vanish and the zeroth intersection is $\int_W\Omega^N>0$. The intersection characterization of numerical dimension for a nef line bundle proves the result.
\end{proof}

\begin{corollary}[Numerical dimension equals variation dimension]\label{cor:numerical-dimension}
Under the common hypotheses of this section, $M$ is nef and
\[
 \nu(M)=v.
\]
\end{corollary}
\begin{proof}
Apply Lemma~\ref{lem:bounded-potential-rank} to the determinant current of Proposition~\ref{prop:continuous-determinant-current}, with $A=Y\setminus U_m^\circ$. Explicitly, for a K\"ahler form $\omega$ on $Y$ and $0\leq k\leq d$,
\begin{equation}\label{eq:numerical-mixed-intersections}
 c_1(M)^k\cdot[\omega]^{d-k}
 =\int_Y T^k\wedge\omega^{d-k}
 =\frac{1}{(2\pi)^k}\int_{U_m^\circ}\eta^k\wedge\omega^{d-k}.
\end{equation}
The positive-degree intersections vanish when $v=0$, while the zeroth one remains positive.
\end{proof}

\begin{corollary}\label{cor:tautological-numerical-dimension}
Under the common hypotheses, the following tautological line bundles are nef, and
\begin{align}
 \nu\bigl(\mathcal O_{\mathbb P_q(E)}(1)\bigr)&=r-1+v,
 \label{eq:tautological-numerical-dimension}\\
 \nu\bigl(\mathcal O_{\mathbb P_q(\operatorname{Sym}^{\ell}E)}(1)\bigr)
 &=R_\ell-1+v\qquad(\ell\geq1),
 \label{eq:symmetric-tautological-numerical-dimension}
\end{align}
where $R_\ell=\operatorname{rank}(\operatorname{Sym}^{\ell}E)$ and projectivization has the quotient convention \eqref{eq:principal-projectivization}.
\end{corollary}
\begin{proof}
Put $V_\ell=\operatorname{Sym}^{\ell}E$ and equip it with the metric $G_\ell$ induced by $G$. On $E^{\otimes\ell}$ the curvature is the sum of the $\ell$ factor curvatures. Each summand is the curvature tensor of one factor tensored with the identity on the remaining factors, so it satisfies the same Nakano bounds. The symmetrizing projection is parallel and orthogonal, so both bounds in \eqref{eq:regular-locus-pinching} restrict to $V_\ell$. With $C_*=r-(r^2-1)\sigma_*$ this gives
\begin{equation}\label{eq:symmetric-power-two-sided}
 \ell\sigma_*\eta\otimes I_{V_\ell}
 \leq_{\mathrm{Nak}}i\Theta_{V_\ell,G_\ell}
 \leq_{\mathrm{Nak}}\ell C_*\eta\otimes I_{V_\ell}
 \quad\text{on }U_m^{\mathrm{reg}}.
\end{equation}

On $p_\ell:Z_\ell=\mathbb P_q(V_\ell)\to Y$, the quotient metric on $L_\ell=\mathcal O_{Z_\ell}(1)$ has weight $\log(\lambda G_\ell^{-1}\lambda^*)$ in a projective quotient chart with row coordinate $\lambda$. The weight is continuous and finite. Griffiths semipositivity makes it psh over $U_m^{\mathrm{reg}}$, and removable singularities extend this property to all of $Z_\ell$.

At $y\in U_m^\circ$, choose a Chern-normal frame for $G_\ell$. As in \eqref{eq:tautological-curvature-split}, the unnormalized quotient curvature has no mixed horizontal--vertical terms. Its vertical block is Fubini--Study of rank $R_\ell-1$. For every quotient row $\lambda\ne0$, its horizontal block $H_\lambda$ satisfies
\[
 \ell\sigma_*\eta\leq H_\lambda\leq\ell C_*\eta.
\]
Consequently $H_\lambda(\xi,\bar\xi)=0$ exactly when $\eta(\xi,\bar\xi)=0$. For a semipositive Hermitian form the zero quadratic vectors are its kernel, so $\ker H_\lambda=\ker\eta$ and $\operatorname{rank}H_\lambda=v$ at every quotient point. The total curvature rank is therefore $R_\ell-1+v$. Lemma~\ref{lem:bounded-potential-rank} on $Z_\ell$, with analytic exceptional set $p_\ell^{-1}(Y\setminus U_m^\circ)$, proves nefness and \eqref{eq:symmetric-tautological-numerical-dimension}. Taking $\ell=1$ gives \eqref{eq:tautological-numerical-dimension}.

When $r=1$ the vertical block is absent and $L_\ell$ identifies with $E^{\otimes\ell}$ on $Y$, giving the same formula. When $v=0$ only the vertical block contributes. The argument uses the induced symmetric-power metric throughout and makes no identification with $E_{\ell m}$.
\end{proof}

\subsection{Vanishing for a continuous metric}
For vanishing, global continuity provides norm equivalence with a smooth metric, while the curvature estimate is used only on the smooth locus.
\begin{lemma}\label{lem:bounded-metric-vanishing}
Let $(Y,\omega)$ be a compact connected K\"ahler manifold of dimension $d$, $F$ a holomorphic vector bundle, and $Z\subset Y$ a proper closed analytic subset. Let $h$ be a continuous positive definite Hermitian metric on $F$, smooth and Nakano semipositive on $U=Y\setminus Z$. Fix $1\leq q\leq d$. If
\[
 B_q=[i\Theta_{F,h},\Lambda_\omega]
\]
on $F$-valued $(d,q)$-forms is positive definite almost everywhere on $U$, then $H^q(Y,K_Y\otimes F)=0$.
\end{lemma}
\begin{proof}
We use the maximal distributional $\bar\partial$ operator in each degree, with domain the $L^2$ forms whose distributional $\bar\partial$ is also $L^2$. The kernel in each degree is closed. All norms and orthogonal projections on $U$ below use the fixed pair $(\omega|_U,h)$. The manifold $U$ admits a complete K\"ahler metric by \cite[Theorem~0.2]{Demailly1982L2}. The restricted $\omega$ need not be complete.

Let $u$ be a smooth global $\bar\partial$-closed $F$-valued $(d,q)$ representative. Nakano semipositivity gives $B_q\geq0$ on $U$. Apply the approximate-solution theorem \cite[Lemma~9.10]{GuanMiYuan2022}, as recorded in \cite[Lemma~2.9]{InayamaMatsumura2024}, with auxiliary functions $\eta=g=1$ and $\lambda=\delta>0$ constant. Their boundedness requirements are then satisfied, $B_q+\delta I>0$, and its inverse integral is at most $\delta^{-1}\|u\|_h^2<\infty$. The theorem permits the fixed, possibly noncomplete metric $\omega$. It gives $v_\delta\in L^2_{d,q-1}(U,F)$ and $\tau_\delta\in L^2_{d,q}(U,F)$ such that
\begin{equation}\label{eq:bounded-metric-approximate-solution}
 \bar\partial v_\delta+P_h(\sqrt\delta\,\tau_\delta)=u,
 \qquad
 \frac12\|v_\delta\|_h^2+\|\tau_\delta\|_h^2
 \leq\int_U\langle(B_q+\delta I)^{-1}u,u\rangle_h\,dV_\omega,
\end{equation}
where $P_h$ is the orthogonal projection onto the maximal $\bar\partial$ kernel in degree $(d,q)$. The coefficient is exactly $(1+1^{-1})^{-1}=1/2$.

Put $r_\delta=P_h(\sqrt\delta\,\tau_\delta)$. Projection decreases the norm, so
\begin{equation}\label{eq:bounded-metric-error}
 \|r_\delta\|_h^2
 \leq\int_U\langle\delta(B_q+\delta I)^{-1}u,u\rangle_h\,dV_\omega
 \longrightarrow0.
\end{equation}
Indeed $0\leq\delta(B_q+\delta I)^{-1}\leq I$, and this endomorphism tends to zero wherever $B_q$ is positive definite. Dominated convergence applies since $u$ is $L^2$.

Continuity and positive definiteness on compact $Y$ make $h$ uniformly equivalent to a smooth reference metric $h_0$. Thus $v_\delta$ and $r_\delta$ have ordinary locally $L^2$ coefficients across $Z$. Extend them measurably over this measure-zero set. In holomorphic frames, \cite[Lemma~6.9]{Demailly1982L2} applies coefficientwise: the unknown is locally $L^2$ and the right side $u-r_\delta$ is locally $L^2$, hence locally $L^1$. It follows that
\[
 \bar\partial v_\delta=u-r_\delta
\]
on all of $Y$ in distributions. In particular the extended $v_\delta$ belongs to the global maximal domain.

Compact Hodge theory gives closed range for the maximal $\bar\partial$ operator on $Y$ with the smooth reference metric $h_0$. Uniform equivalence leaves its domains and ranges unchanged as sets and preserves closedness. Since $u-r_\delta$ lies in this range and converges to $u$, so does $u$. This uses closed range only on compact $Y$ and requires no bound for $v_\delta$ uniform in $\delta$.

Let $H_0$ and $\mathcal G_{q,0}$ denote the harmonic projection and Green operator for $h_0$ and the fixed K\"ahler metric $\omega$. Since $u$ lies in the maximal $\bar\partial$ range, $H_0u=0$. The smooth form $u$ is $\bar\partial$-closed, so smooth Hodge theory gives
\[
 u=\bar\partial\bigl(\bar\partial_{h_0}^{*}\mathcal G_{q,0}u\bigr).
\]
Thus $u$ has a smooth primitive and represents zero in $H^q(Y,K_Y\otimes F)$.
\end{proof}

\subsection{Symmetric-power vanishing theorems}
\begin{theorem}[Variation-dependent vanishing theorem]
\label{thm:variation-vanishing}
Retain the common hypotheses of this section and the constant $\sigma_*$ in \eqref{eq:global-full-moment}. For integers $a\geq1$ and $b\geq0$ satisfying $b<a\sigma_*$, and every $P\in\operatorname{Pic}^0(Y)$,
\begin{equation}\label{eq:variation-vanishing}
 H^q\!\left(Y,K_Y\otimes\operatorname{Sym}^{a}E
                  \otimes M^{-b}\otimes P\right)=0
 \qquad(q>d-v).
\end{equation}
No additional condition \eqref{eq:global-threshold} is required.
\end{theorem}
\begin{proof}
Equip $P$ with a unitary flat metric and write
\[
 V=\operatorname{Sym}^{a}E\otimes M^{-b}\otimes P,
 \qquad c=a\sigma_*-b>0.
\]
The induced metric $h$ is continuous and positive definite on $Y$, and smooth on $U_m^{\mathrm{reg}}$. The tensor curvature sum and the parallel symmetric summand, as in \eqref{eq:symmetric-power-two-sided}, give the lower bound $a\sigma_*\eta\otimes I$ on $\operatorname{Sym}^{a}E$. Twisting by $M^{-b}$ subtracts $b\eta\otimes I$, and the flat twist contributes zero. Therefore
\begin{equation}\label{eq:symmetric-power-nakano}
 i\Theta_{V,h}\geq_{\mathrm{Nak}}c\eta\otimes I_V
 \quad\text{on }U_m^{\mathrm{reg}}.
\end{equation}
In particular $h$ is Nakano semipositive there.

Fix a smooth K\"ahler form $\omega$ on $Y$. At a point of $U_m^{\mathrm{reg}}$, diagonalize $\eta$ relative to $\omega$, with eigenvalues $0\leq\lambda_1\leq\cdots\leq\lambda_d$. On $V$-valued $(d,q)$-forms, \eqref{eq:symmetric-power-nakano} gives
\[
 \langle[i\Theta_{V,h},\Lambda_\omega]u,u\rangle_h
 \geq c\sum_{|J|=q}\left(\sum_{j\in J}\lambda_j\right)|u_J|_h^2.
\]
At every point of $U_m^\circ$, exactly $d-v$ eigenvalues vanish. Thus every displayed $q$-index sum is positive for $d-v<q\leq d$. The complement of $U_m^\circ$ is a proper analytic set of measure zero. Lemma~\ref{lem:bounded-metric-vanishing}, with $Z=Y\setminus U_m^{\mathrm{reg}}$, proves the assertion in these degrees. Degrees $q>d$ vanish by dimension. If $v=0$, only these latter degrees occur.
\end{proof}

For $b=0$ the theorem applies to every $a\geq1$. Maximal variation gives every positive cohomological degree, whereas $v=0$ gives no nontrivial range. The strict bound $b<a\sigma_*$ is determined by the evaluation quadratic form. No universal positive lower bound for $\sigma_*$ in terms of $m$ and $r$ alone is asserted. Over curves, the case $b=0$ already follows from ampleness and Serre duality as in Section~\ref{sec:curves}. The bound $b<a\sigma_*$ is a sufficient range determined by $G$.

Related positivity and vanishing results for induced symmetric-power metrics are established by Liu--Sun--Yang \cite[Theorem~1.1]{LiuSunYang2013}. In their criterion, ampleness of $\operatorname{Sym}^{a+r}E\otimes M^{-(b+1)}\otimes P$ gives Nakano positivity of the bundle in \eqref{eq:symmetric-power-nakano}. Here the lower form may have rank $v<d$, and the evaluation quadratic form gives the stated determinant-twist range. For canonically polarized families, Schumacher's K\"ahler--Einstein method also computes Nakano curvature of the bicanonical direct image \cite[\S5, Theorem~5 and Corollary~4]{Schumacher2018Moduli}. 

\appendix
\section{Refined regularity of the fixed-point metric}
\label{app:refined-regularity}

We use the unweighted smooth-family hypotheses and notation of Theorem~\ref{thm:threshold-regularity}. In particular, $\alpha=(m-1)/m$, $\beta=2/m$, and $c_{y_0}$ is defined by \eqref{eq:central-threshold}. The following kernel estimates are uniform on compact subsets of the positive Hermitian cone.

\begin{lemma}\label{lem:homogeneous-kernel}
Let $\mathcal A$ be a compact subset of the positive Hermitian cone. For $A\in\mathcal A$, a row $v\in\mathbb C^r$ and $\epsilon>0$, set
\[
 H_{A,\epsilon}(v)=v^*v(vA^{-1}v^*+\epsilon)^{-\alpha}.
\]
The unregularized kernel $H_{A,0}$ is defined by zero at $v=0$. For nonnegative integers $j,a$, uniformly for $0\leq\epsilon\leq1$ and $v\ne0$,
\begin{equation}\label{eq:kernel-higher-derivatives}
 \|D_v^jD_A^aH_{A,\epsilon}(v)\|
       \leq C_{j,a}|v|^{\beta-j}.
\end{equation}
All derivatives are real derivatives. For $0<\gamma<1$ and nonzero $v,w$, the corresponding difference is bounded by
\begin{equation}\label{eq:kernel-higher-difference}
 C_{j,a,\gamma}|v-w|^\gamma
 \bigl(|v|^{\beta-j-\gamma}+|w|^{\beta-j-\gamma}\bigr).
\end{equation}
The kernels themselves are uniformly $\beta$-H\"older on bounded sets, including at zero.
\end{lemma}
\begin{proof}
For $\epsilon=0$, homogeneity of degree $\beta$, unchanged by matrix differentiation, reduces \eqref{eq:kernel-higher-derivatives} to bounded derivatives on the unit sphere. These bounds are uniform on $\mathcal A$. For $\epsilon>0$ and $|v|\geq\sqrt\epsilon$, rescale by $|v|$. The rescaled denominator is bounded below and its extra parameter $\epsilon/|v|^2$ lies in $[0,1]$. If $|v|\leq\sqrt\epsilon$, use
\[
 H_{A,\epsilon}(v)=\epsilon^{\beta/2}
                       H_{A,1}(v/\sqrt\epsilon).
\]
For $j\geq1$, bounded derivatives of $H_{A,1}$ give $C\epsilon^{(\beta-j)/2}\leq C|v|^{\beta-j}$ because $\beta-j\leq0$. For $j=0$, including matrix derivatives, direct differentiation gives $C|v|^2(|v|^2+\epsilon)^{-\alpha}\leq C|v|^\beta$. This proves \eqref{eq:kernel-higher-derivatives} in all cases.

Put $\delta=|v-w|$. If $\delta\leq\frac12\min(|v|,|w|)$, the segment stays away from zero and its points have comparable norms. The derivative estimate of order $j+1$ and $\delta^{1-\gamma}\leq C\min(|v|,|w|)^{1-\gamma}$ prove \eqref{eq:kernel-higher-difference}. Otherwise both $|v|$ and $|w|$ are at most $3\delta$, and the two separate order-$j$ estimates prove the same inequality. Taking $j=0$ and using the first-derivative bound in the former case gives the uniform $\beta$-H\"older estimate.
\end{proof}

\begin{proof}[Proof of Theorem~\ref{thm:threshold-regularity}]
Choose $\gamma<\gamma'<1$ such that $k+\gamma'<\beta+2c_{y_0}$. Use finitely many nested relative charts and a smooth partition of unity as in the proof of Lemma~\ref{lem:holder-continuity}. Properness allows one common base neighborhood. Each local integral is over a fixed bounded polydisc $W$ in the fiber coordinates, with compactly supported smooth amplitude $\chi(y,w)$ and holomorphic evaluation row $F_y(w)$. All derivatives of $F$ and $\chi$ needed below are bounded on a slightly larger chart.

For the inverse-power estimate, the functions $\log|F_y|$ converge to $\log|F_{y_0}|$ in $L^1_{\mathrm{loc}}(W)$. Indeed, they are locally uniformly bounded above, converge away from the zero set of the nonzero central tuple, and the compactness theorem for plurisubharmonic functions identifies every $L^1$ subsequential limit with that central logarithm. The alternative of convergence to $-\infty$ is excluded on a small ball where the central tuple is nonzero. Demailly--Koll\'ar's effective semicontinuity theorem \cite[Main Theorem~0.2(2), p.~526]{DemaillyKollar2001}, applied on each compact chart support $K\Subset W$, gives
\begin{equation}\label{eq:threshold-integrability}
 \sup_{y\in V_c}\int_K|F_y(w)|^{-2c}\,dV_w<\infty
 \qquad(0<c<c_{y_0})
\end{equation}
after shrinking $V_c$ about $y_0$. 

In each localized integral replace $H_{A,0}(F_y)$ by $H_{A,\epsilon}(F_y)$ and call the resulting finite sum $T_\epsilon$. These are smooth functions of $(y,A)$ for $\epsilon>0$. Although defined in the fixed charts, this regularization converges to the intrinsic operator $T$. The chain rule and Lemma~\ref{lem:homogeneous-kernel} bound every joint derivative of total order $j\leq k$ of a localized integrand by
\begin{equation}\label{eq:regularized-derivative-majorant}
 C_j\bigl(1+|F_y|^{\min(\beta-j,0)}\bigr).
\end{equation}
To compare derivatives at $(y,A)$ and $(z,B)$, split each chain-rule term into a difference of kernel derivatives and differences of its smooth coefficients. Matrix differentiation does not lower the degree in $F$, so the mean value theorem in $A$, together with \eqref{eq:kernel-higher-difference} in $F$, bounds the order-$k$ difference by
\begin{equation}\label{eq:regularized-holder-majorant}
 \begin{split}
 C\bigl(|y-z|+\|A-B\|\bigr)^{\gamma'}
 \bigl(1+|F_y|^{\min(\beta-k-\gamma',0)}
          +|F_z|^{\min(\beta-k-\gamma',0)}\bigr).
 \end{split}
\end{equation}
Here the parameters lie in a fixed smaller bounded neighborhood. The estimates hold away from the zero sets of the indicated tuples. Those analytic sets have measure zero in each fiber, since the fiberwise base-change isomorphisms and positive rank make each tuple nonzero.

If $k+\gamma'>\beta$, choose $c<c_{y_0}$ with $k+\gamma'-\beta<2c$ and apply \eqref{eq:threshold-integrability} to integrate both majorants. If $k+\gamma'\leq\beta$, boundedness of the rows suffices. Consequently $T_\epsilon$ has uniform $C^{k,\gamma'}$ bounds on smaller parameter neighborhoods. Dominated convergence gives its pointwise limit $T$, since $\|H_{A,\epsilon}(F_y)\|\leq C|F_y|^\beta$. The compact embedding of bounded $C^{k,\gamma'}$ sets into $C^{k,\gamma}$ on smaller closed neighborhoods identifies the limit and its derivatives. This proves joint $C^{k,\gamma}$ regularity.

For $k\geq1$, matrix differentiation under the integral is justified independently by the bound $C|F_y|^\beta$. Formula \eqref{eq:matrix-derivative} therefore still holds at $y_0$. Positivity and homogeneity give $D_AT_{y_0}(G_{y_0})[G_{y_0}]=\alpha G_{y_0}$, so \eqref{eq:implicit-inverse} supplies the inverse with order norm at most $m$. The finite-regularity implicit-function theorem applied to $A-T_y(A)$ gives a $C^{k,\gamma}$ solution. One can see the H\"older part directly from
\[
 D_yG=(I-D_AT_y(G))^{-1}D_yT_y(G):
\]
the ordinary $C^1$ theorem makes $G$ locally Lipschitz, and this identity first gives $C^{1,\gamma}$ and then the higher orders by induction. Uniqueness identifies this solution with $G$.

For $k=0$, Lemma~\ref{lem:holder-continuity} already places the fixed matrices in a compact positive cone near $y_0$. Contraction gives
\[
 (1-\alpha)d_{\mathrm{Th}}(G_y,G_z)
 \leq d_{\mathrm{Th}}(T_y(G_z),T_z(G_z))
 \leq C|y-z|^\gamma.
\]
The last bound uses joint $C^{0,\gamma}$ regularity and local equivalence of matrix norm and Thompson distance. This proves the assertion for $k=0$. Finally inversion and the logarithm preserve the proved regularity wherever $FG^{-1}F^*>0$.
\end{proof}

The threshold criterion gives the following optional global differentiability statement. With $c_y$ as in \eqref{eq:central-threshold}, and $c_y=+\infty$ when $mK_{X_y}$ is globally generated, consider the additional condition
\begin{equation}\label{eq:global-threshold}
 c_y>\alpha\qquad\text{for every }y\in Y\setminus U_m^{\mathrm{reg}}.
\end{equation}
\begin{proposition}\label{prop:global-curvature}
If \eqref{eq:global-threshold} holds in addition to the common hypotheses, $G$ is $C^{2,\gamma}$ on $Y$ for some $0<\gamma<1$. Its globally defined continuous curvature satisfies
\begin{equation}\label{eq:global-nakano-pinching}
 \eta:=i\operatorname{Tr}\Theta_{E,G}\geq0,\qquad
 i\Theta_{E,G}\geq_{\mathrm{Nak}}\sigma_*\eta\otimes I_E.
\end{equation}
The form $\eta$ has rank $v$ on $U_m^\circ$ and rank at most $v$ everywhere.
\end{proposition}
\begin{proof}
For $y\notin U_m^{\mathrm{reg}}$, choose $0<\gamma_y<\min\{1,2(c_y-\alpha)\}$. Then $2+\gamma_y<2/m+2c_y$, and Theorem~\ref{thm:threshold-regularity} gives $C^{2,\gamma_y}$ regularity near $y$. On $U_m^{\mathrm{reg}}$ use Theorem~\ref{thm:controlled-regularity}. A finite subcover of compact $Y$ and an exponent below the finitely many local exponents give the asserted global regularity. Curvature continuity extends \eqref{eq:regular-locus-pinching} from the dense subset $U_m^{\mathrm{reg}}$, proving \eqref{eq:global-nakano-pinching}. Every $(v+1)$-minor of $\eta$ vanishes on $U_m^\circ$, hence everywhere by continuity.
\end{proof}

\bibliographystyle{amsplain}
\bibliography{paper_marked}

\end{document}